%% file: main.tex
\documentclass[11pt, letter]{amsart}
\usepackage{standalone}
\usepackage{import, float, verbatim, enumerate}
\usepackage{soul, amsmath}
\usepackage{tikz}
\usepackage{xcolor}

\input{commands}

\newcommand{\G}{\mathcal{G}}

\newcommand{\tar}{\mathrm{tar}}
\newcommand{\init}{\mathrm{init}}
\newcommand{\fin}{\mathrm{fin}}
\newcommand{\AC}{\mathrm{AC}}
\newcommand{\Mat}{\mathrm{Mat}}

\usepackage{subcaption}

\usepackage{tikz}
\usepackage{tikz-cd}
\usepackage{color}

\definecolor{Purple}{rgb}{.36,.22,.6}
\definecolor{Orange}{rgb}{.9,.38,0}

\newcommand{\sgn}{\mathrm{sgn}}

\makeatletter
  \def\p@subfigure{\thefigure}
\makeatother

\begin{document}

\title[Whittled complex for torus links]{A whittled complex for the Khovanov homology of torus links}

\author[Caprau]{Carmen Caprau}
\address{Department of Mathematics, California State University, Fresno, CA 93740}
\email{ccaprau@csufresno.edu}

\author[N. Gonz{\'a}lez]{Nicolle Gonz{\'a}lez}
\address{
    Department of Mathematics, 
    The University of British Columbia, 
    Vancouver, BC, Canada
}
\email{nicolle@math.ubc.ca}

\author[Lee]{Christine Ruey Shan Lee}
\address{Department of Mathematics, Texas State University, San Marcos, TX 78666}
\email{vne11@txstate.edu}

\author[Sazdanovi\'{c}]{Radmila Sazdanovi\'{c}}
\address{Department of Mathematics, North Carolina State University, Raleigh NC 27695-88205}
\email{rsazdan@ncsu.edu}

\date{\today}

\begin{abstract}
We give an algorithm for reducing the number of generators of the Khovanov chain complex of the torus braid $\ftbraid^k_n = (\sigma_1\sigma_2\cdots \sigma_{n-1})^k$ on $n$ strands by applying Bar-Natan Gaussian elimination along a distinguished set of Gaussian elimination isomorphisms. We call the resulting complex $\FT^k_n$ a \emph{whittled complex} for the Khovanov homology of torus braids. Using this algorithm, we provide a bound for the number of generators at a fixed homological degree in our whittled complex.

\end{abstract}

\maketitle

\tableofcontents

\section{Introduction}
\label{sec:introduction}
Since its introduction, Khovanov's categorification of the Jones polynomial \cite{Khovanov-Jones} has revolutionized the study of quantum link invariants, leading to major applications in low-dimensional topology \cite{Piccirillo-Conway, Rasmussen-slice}. Despite this, little is known about the structure of this homology theory even for specialized classes of knots, such as torus knots \cite{GOR}, due to the combinatorial complexity of the associated chain complexes. 

In this paper we study the structure of the Khovanov homology of infinite torus links. Let $\ftbraid_n = \sigma_1\sigma_2\cdots\sigma_{n-1}$ denote a braid in the braid group $B_n$ consisting of a factor of a full twist braid $\triangle_n^2 = (\sigma_1\sigma_2\cdots\sigma_{n-1})^n$ on $n$ strands. See Notation \ref{n.braidgen} for our convention for the generators of the braid group after Artin \cite{Artin-braids}.
We denote by $\ftbraid_n^k$ the braid word $(\sigma_1\sigma_2\cdots\sigma_{n-1})^{k}$ that is $k$ copies of $\ftbraid_n$.  The Khovanov homology of the \emph{infinite torus braid} on $n$ strands is defined as the direct limit of the homology groups $\Kh(\ftbraid_n^k)$ as $k\rightarrow \infty$ \cite{ROZ}. Taking the closure $\widehat{\ftbraid_n^k}$ of the braid $\ftbraid_n^k$, the Khovanov homology of the \emph{infinite torus link} is defined as the direct limit of $\Kh(\widehat{\ftbraid_n^k})$ as $k\rightarrow \infty$. 
Sto{\v{s}}i{\'c}'s work \cite{stosic2007homological} was the first that shows that Khovanov homology of torus knots stabilizes, which implies that the direct limit exists and is well-defined, and Rozansky's work \cite{ROZ} generalizes the stability of the Khovanov homology groups to tangles. 

The full twists on $n$ strands, denoted in this paper by $\ftbraid^n_n = \triangle^2_n$, is a central object in low-dimensional and quantum topology that has been extensively studied in relation to the braid group \cite{Artin-braids} and the links their closures represent, their Garside normal structures \cite{Dehornoy-Garside}, and braid group representations \cite{Feller-Hubbard}. There have also been many recent developments in its topological applications, especially with respect to knot invariants that behave well with respect to knot cabling, concordances, and filtration structures, a comprehensive account of which is beyond the scope of this paper.

The evaluation of the Khovanov homology of cabled knots is relevant to understanding the geometric information carried by the colored Jones polynomial, which is a generalization of the Jones polynomial that can be written as a sum of the Jones polynomials of cables of a diagram of the knot. For example, the Volume Conjecture posits that the asymptotic values of the colored Jones polynomial of a hyperbolic knot recovers the volume of the knot complement. We refer the reader to the work by Murakami and Murakami in \cite{Murakami-Murakami-cjpvol}, for the authoritative account on the Volume Conjecture. The colored Jones polynomial may be obtained by evaluating the Kauffman bracket of the blackboard cable of a diagram of the link decorated with a Jones-Wenzl projector\footnote{Lickorish does not call the invariant "the colored Jones polynomial."} \cite{Lickorish}.  Rozansky \cite{ROZ} showed that the Khovanov homology of infinite full twists on $n$ strands recovers the categorification of the Jones-Wenzl projector by Cooper-Krushkal \cite{CK-JonesW}, which can then be used to obtain a categorification of the $n$th colored Jones polynomial. This was used in \cite{Lee-StableVolume} to relate Khovanov homology to hyperbolic volume, and in \cite{Lee-ComputeViaJW}, the third author proposes to simplify the computation of Khovanov homology by computing the categorification of the Jones-Wenzl projector. 

Given the potential applications, we are motivated by the problem of evaluating the Khovanov homology of infinite torus braids. In this paper, we develop a computational machinery to simplify the underlying chain complex of $\Kh(\ftbraid_n^k)$ by reducing the number of generators. Our approach is to deploy Gaussian eliminations developed by Bar-Natan in \cite{bar2007fast} at large scale to reduce, or \textit{whittle}, the Khovanov chain complex $\CKh(\ftbraid_n^k)$ into a much more manageable complex. We call this new complex \textit{the whittled complex} $\FT_n^k$ (Definition \ref{d.whittled_complex}). 

\subsection{Main results}
\begin{theorem}
\label{t.deflate}
Let $\FT_n^k$ denote the whittled complex obtained from whittling the Khovanov chain complex $\CKh(\ftbraid_n^k)$, and let $(\sigma, \epsilon)$ be an enhanced Kauffman state (Definition \ref{d.enhanced-K-state}) on $\ftbraid_n^k$. 
Suppose $(\sigma, \epsilon)$ appears in the whittled complex $\FT_n^k$ and let $W$  be the corresponding tangle in the Temperley-Lieb monoid $\TL_n$. 
Then either (1) the word is in the form 
\[ e_{n-1}^{k_0} V_0 e_{n-1}^{k_1} V_1 \cdots V_{r}e_{n-1}^{k_{r}},  \]
where each $V_j = e_{j_1}e_{j_2}\cdots e_{n-1}$ is a subword consisting of consecutive elements in $\TL_n$ ending at $e_{n-1}$ and $k_j \geq 2$ for all $j$. (2) Or, there exists a path of $\TL$-moves (see Definition \ref{n.tlmoves}) transforming $W$ to its Jones Normal Form $\hat W = \JNF(W)$ (Definition \ref{d.jnf}), where each move is one of the following three types:
\begin{itemize}
    \item[(D1)] \label{item:D1} $e_{n-1}^2 \to e_{n-1}$
    \item[(D2)] \label{item:D2} $e_i e_j \to e_je_i$, where $|i-j|\geq 2$
    \item[(D3)] \label{item:D3} $e_ie_{i-1}e_i \to e_i$, for $2 \leq i \leq n-1$
\end{itemize}
\end{theorem}

The advantage of the whittled complex over the original complex is that the number of enhanced Kauffman states in the whittled complex is vastly reduced from the original Khovanov chain complex. Theorem \ref{t.deflate} allows us to bound the number of enhanced Kauffman states in the whittled complex $\FT^k_n$ at a fixed homological degree $h$, in terms of $n$ and $h$.

\begin{theorem}\label{t.count}
        For a fixed homological degree $h$, let $C(\FT^k_n, h)$ be the number of Kauffman states $s$ for which $(s, \epsilon)$ appears in the whittled complex $\FT^k_n$ and the homological degree of $(s, \epsilon) = h$. Then  
            \[   C(\FT^k_n, h) \leq  \left( \sum_{m\leq h} p(h, m) \right) + N(n, h) +  (p(n, 2)+2)C_n. \]
\end{theorem}

In the statement of Theorem~\ref{t.count}, $p(n, k)$ stands for the number of ways one can write a positive integer $n$ as an ordered sum of $k$ positive integers, and $N(n, h)$ denotes the number of words in the Temperley-Lieb monoid $\TL_n$ in Jones Normal Form of a fixed length $h$. Finally, $C_n$ is the $n$th Catalan number.

Our approach shares similarities with that of recent work by Kelomäki \cite{kelomäki2024discretemorsetheorykhovanov, kelomäki2025morsematchingskhovanovhomology}, except that we did not know of the existence of discrete Morse theory since we started this project in 2020. We have independently developed many of the tools in discrete Morse theory for our purposes. As a result, the paper is self-contained. In discrete Morse theory, whether or not a set $\G$ of Gaussian eliminations can be applied to simplify the complex depends on the acyclicity of a directed graph $G$ constructed from the cube of resolutions. We derive this result and construct the graph $G$ in this paper. In addition to validating our algorithm, we prove the acyclicity  of $G$ for braids on all numbers of strands in Theorem \ref{t.G-acyclic}. 

A related work is the Gorsky-Oblomkov-Rasmussen conjecture \cite{GOR}, which predicts a similar reduced number (and form) of generators and differentials for $\Kh(\ftbraid_n^k)$. The GOR conjecture has been partially confirmed by Hogancamp \cite{Hogancamp2019} on the level of the generators of the chain complex. However, the differentials have not been determined explicitly. We believe our approach has the potential of making the differentials explicit for the whittled complex in view of the GOR conjecture, and we plan to investigate this, as well as the computational advantages of the whittled complex, in a future project. 

\subsection{Organization of the paper} \ \\ 
In Section~\ref{sec:background}, we briefly review general preliminaries on Khovanov homology, where we rely heavily on Bar–Natan’s formulation in~\cite{BN05}, and describe Gaussian elimination for chain complexes following Bar–Natan~\cite{bar2007fast} (see Lemma~\ref{l.geliminate}). We assume that the reader has some familiarity with Khovanov homology. In Section~\ref{ss.tlmjnf}, we define the Jones Normal Form of a braid word and prove a result of Jones~\cite{Jones-index} (Proposition~\ref{prop:monotone-path}), which states that every word in the Temperley-Lieb monoid $\TL_n$ is equivalent to a word in Jones Normal Form.

In Section \ref{s.Gaussian_Elimination}, we describe certain Gaussian elimination isomorphisms in the complex $\CKh(\ftbraid_n^k)$. In Section \ref{ss.algorithm} we then introduce the Algorithm \ref{ss.algorithm} that chooses, for an enhanced Kauffman state in $\CKh(\ftbraid_n^k)$, at most one Gaussian elimination isomorphism, which we refer to as  a \textit{distinguished Gaussian elimination isomorphism}. We prove the correctness of this algorithm in Lemma \ref{lem:unique-distinguished-isom}. The set $\G$ of all such distinguished Gaussian elimination isomorphisms is defined in Definition \ref{d.setGaussianG}. 

In Section \ref{ss.whittling_graph} we define the graph $G$ constructed from $\G$, and in Section \ref{ss.linear_order} we show that if $G$ is acyclic, then there exists an order on $\G$ which we can use to apply Gaussian eliminations along all the elements in $\G$, reducing $\CKh(\ftbraid_n^k)$ to $\FT_n^k$ (Definition \ref{d.whittled_complex}). The most technical part of the paper is Section \ref{ss.linear_order}, where we prove that $G$ is acyclic. 

In Section \ref{s.reducetoJNF} we combine results from previous sections to prove Theorem \ref{t.deflate}, and in Section \ref{s.enumerate} we use Theorem \ref{t.deflate} to prove Theorem \ref{t.count}. Motivating examples for Theorem \ref{t.deflate} are discussed in Section \ref{ss.motivatingegs}. 

\section{Background}
\label{sec:background}

\subsection{Kauffman states and enhanced Kauffman states}

\begin{definition}[Kauffman state] \label{d.Kauffman_state}
Given a tangle or a link diagram $D$ with crossings, a \textit{Kauffman state $\sigma$} on $D$ is a choice of the $0$-resolution or the $1$-resolution at every crossing of $D$, as shown in Figure~\ref{figure:resolutions}. \textit{Applying a Kauffman state to $D$} means to replace all crossings in $D$ by the state's choice of the $0$-resolution or the $1$-resolution at each crossing, which results in a disjoint collection of arcs and closed components without crossings, which we call a \textit{complete tangle resolution}.
\end{definition} 

\begin{figure}[h]
\centering
\begin{tikzpicture}[scale=.5]
\draw (-1,-1) -- (1,1);
\draw (-1,1) -- (-.15,.15);
\draw (.15,-.15) -- (1,-1);
\draw (-6.5,-2.75) node[above, align=center]{$0\text{-resolution}$};
\draw (6.5,-2.75) node[above, align=center]{$1\text{-resolution}$};
\draw (0,-2.75) node[above, align=center]{Crossing};
\draw[->, line width=0.6pt] (-2.3,0) -- (-4.2,0);
\draw[->, line width=0.6pt] (2.3,0) -- (4.2,0);
\foreach \ox/\cx in {-7.5/-6.5, -5.5/-6.5} \draw[rounded corners = 6mm] (\ox,1) -- (\cx,0) -- (\ox,-1);
\foreach \ox in {1,-1} \draw[rounded corners = 6mm] (5.5,\ox) -- (6.5,0) -- (7.5,\ox);
\foreach \x in {-6.5, 0, 6.5} \draw[dash pattern=on 10pt off 10pt] (\x,0) circle (1.7cm);
\end{tikzpicture}
\caption{The $0$-and $1$-resolution of a Kauffman state.}
\label{figure:resolutions}
\end{figure}

\begin{definition} \label{d.enhanced-K-state}
Let $T$ be a tangle with crossings, $\sigma$ a Kauffman state on $T$, and $T_{\sigma}$ be the complete resolution of $T$ by $\sigma$. An \textit{enhanced Kauffman state} is a pair $(\sigma, \epsilon)$, where $\epsilon: T_{\sigma}\to \{\pm \}$ is a marking of $+$ or $-$ on each closed component of  $T_\sigma$.
\end{definition}

\subsection{Khovanov homology and the Khovanov chain complex}
Let $L$ be a link diagram with $t$ crossings and number the crossings $\{1, 2, \ldots, t\}$. The application of a Kauffman state $\sigma$ on $L$ can be encoded by a binary string $s_b$ of length $t$, where the  $i$th entry is $0$ if the Kauffman state chooses the $0$-resolution at the $i$th crossing, and $1$ if it chooses the $1$-resolution. For an enhanced Kauffman state $(\sigma, \epsilon)$ in $\CKh(L)$, we define the homological degree $h$ and the quantum degree $q$. 
\begin{definition}[Homological and Quantum grading \cite{Turner-5lectures}]
Let $(\sigma, \epsilon)$ be an enhanced Kauffman state on a link $L \subset S^3$ with diagram $D$, with corresponding binary string $s_b$. Let $r_s$ be the number of $1$'s in $s_b$, and let $c_-(D)$ be the number of negative crossings in $D$, while $c_+(D)$ is the number of positive crossings of $D$. Let $\deg(\sigma, \epsilon)$ be the number of circles marked with $+$ minus the number of circles marked with $-$.  Then 
\begin{itemize}
\item The homological grading $h = h(\sigma, \epsilon)$ of $(\sigma, \epsilon)$ is 
\[ h = r_s - c_-.\]
\item The quantum grading of $q = q(\sigma, \epsilon)$ is 
\[ q = \deg(\sigma, \epsilon) + h + c_+(D) - c_-(D). \]
\end{itemize} 
\end{definition}

To an oriented link diagram $D$ with $t$ crossings, Khovanov associates a chain complex $C^{*, *}(D)$. Let $V = \mathbb{Z}\{v_-, v_+\}$ denote the $\mathbb{Z}$-module with basis $v_-$ and $v_+$, graded so that $\deg(v_-) = -1$ and $\deg(v_+) = 1$. For a Kauffman state $\sigma$, let $k_s$ denote the number of closed components in the complete resolution of $D$ obtained by applying $\sigma$ to $D$ to resolve all crossings. 
To an element $s_b$ of the set of $2^t$ resolutions of $D$, we associate the graded vector space 
\[ V_{s} := V^{\otimes k_s} \{r_s + c_+(D)-2n_-(D)\},\]
and define
\[ C^{h, *}(D) := \underset{s_b\in \{0, 1\}^n, r_s = h + c_-(D) }{\oplus}  V_s.\]

We now describe the differential $d$ that turns $\{C^{*, *}(D)\}$ into a chain complex.  Given two Kauffman states such that their associated binary strings $s$ and $s'$ are related by changing a single $0$ in $s_b$ to a $1$ in $s'_b$, while all other entries are the same, we define a map  $d_s: V_s \rightarrow V_{s'}$ as follows. Define \[ m: V\otimes V \rightarrow V \] by 
\[ m(v_+ \otimes v_+) = v_+, \qquad m(v_+\otimes v_-) = m(v_- \otimes v_+) = v_-, \qquad m(v_- \otimes v_-) = 0.\]
Moreover, define \[ \triangle: V \rightarrow V\otimes V\] by 
\[ \triangle(v_+) = v_+ \otimes v_- + v_- \otimes v_+, \qquad \triangle(v_-) = v_-\otimes v_-. \]
For $s$ and $s'$ related by changing $0$ in $s$ to $1$ in $s'$ (note that either a closed component is split into two or two closed components are merged into 1), we define the following, for $v \in V_s$:
\[ d_{s, s'}(v) = 
\begin{cases} 
m(v) &\text{if from $s$ to $s'$, a pair of closed components $C, C'$ in $D_\sigma$} \\ 
    &\text{are merged into a single component in $D_{\sigma'}$} \\ 
\triangle(v) &\text{if from $s$ to $s'$, a closed component $C\in D_\sigma$ is split into two}  \\
       & \text{closed components $C', C''$ in $D_{\sigma'}$} 
\end{cases} \]
The maps $m$ and $\triangle$ are extended over $V_{s}$ by operating only on the component $V\otimes V$ of the tensor product corresponding to $C$ and $C'$ in $D_\sigma$, and respectively on the component $V$ corresponding to $C$ in $D_{\sigma}$. 

Finally for $s$ and $s'$ related by a single change of resolution from $0$ to $1$, where $s'$ is obtained from $s$ by changing a $0\in s$ to $1$, let $\star$ denote the digit in $s$ which is changed. Reading a binary string from left to right, let $\sgn(s, s')$ be 
\[ \sgn(s, s') = (-1)^{\text{\# of $1$'s to the left of $\star$ in $s$}}. \] 
Then, for $v\in C^{h, *}(D)$, the differential $d^h$ at homological grading $h$ is defined as 
\[ d^h (v)  = \underset{s\in C^{h, *}(D), \ \text{$s' \in C^{h, *}$ related to $s$}}{\sum} \sgn(s, s') d_{s, s'}(v). \]

\subsection{The Khovanov complex for tangles}

Bar-Natan extended the Khovanov homology to tangles in~\cite{BN05}, where he associated to a tangle a formal chain complex whose objects are formal direct sums of graded tangle resolutions  and whose morphisms are matrices of cobordisms between tangle resolutions.  Without passing to $\mathbb{Z}$-modules, the homotopy type of Bar-Natan's formal chain complex associated to a tangle $T$ is invariant under the Reidemeister moves  \cite[Theorem 1]{BN05}; thus it is an invariant of tangles.
We remark that Bar-Natan's construction applies to links as well, by regarding a link as a tangle with no boundary points. We assume the reader's familiarity with the paper~\cite{BN05}, and we only briefly describe below the settings from~\cite{BN05} that are most relevant to our work in this paper.

The category $\Cob^3(\emptyset)$ is the pre-additive category whose objects are complete tangle resolutions and whose morphisms are cobordisms between such resolutions (smoothings), considered up to boundary-preserving isotopies. If $B$ is a finite set of points on a circle, such as the boundary $\partial T$ of a tangle $T$, then $\Cob^3(B)$ is the category whose objects are (boundary-preserving) isotopy classes of complete tangle resolutions with boundary $B$ and whose morphisms are cobordisms between such objects. We follow Bar-Natan's lead in designating $\Cob^3$ as either $\Cob^3(\emptyset)$ and $\Cob^3(B)$ to reduce unnecessary notation. 

The pre-additive category $\Mat(\Cob^3)$ is constructed from $\Cob^3$ as follows: 
\begin{itemize}
    \item \textbf{Objects}: Possibly empty formal direct sums $\bigoplus_{i=1}^n \mathcal{O}_i$ of objects $\mathcal{O}_i$ of $\Cob^3$. 
    \item If $\mathcal{O} = \bigoplus_{i=1}^n \mathcal{O}_i$ and $\mathcal{O'} = \bigoplus_{j=1}^m \mathcal{O}_j'$, then a morphism $F: \mathcal{O'}\rightarrow \mathcal{O}$ in $\Mat(\Cob^3)$ is an $m\times n$ matrix $F = (F_{ij})$ of morphisms $F_{ij}: \mathcal{O}_j' \rightarrow \mathcal{O}_i$.
    \item \textbf{Morphisms} in $\Mat(\Cob^3)$ are added using matrix addition. 
    \item Compositions of morphisms in $\Mat(\Cob^3)$ are defined by a rule modeled on matrix multiplication: 
    \[ ((F_{ij})\circ (G_{jk}))_{ik}:= \sum_j F_{ij} \circ G_{jk}.\]
\end{itemize}

We quotient the morphisms of the category $\Cob^3$ by the local relations illustrated in Figure \ref{f.bn_relations}, and call the resulting quotient $\Cob^3_{/l}$. 

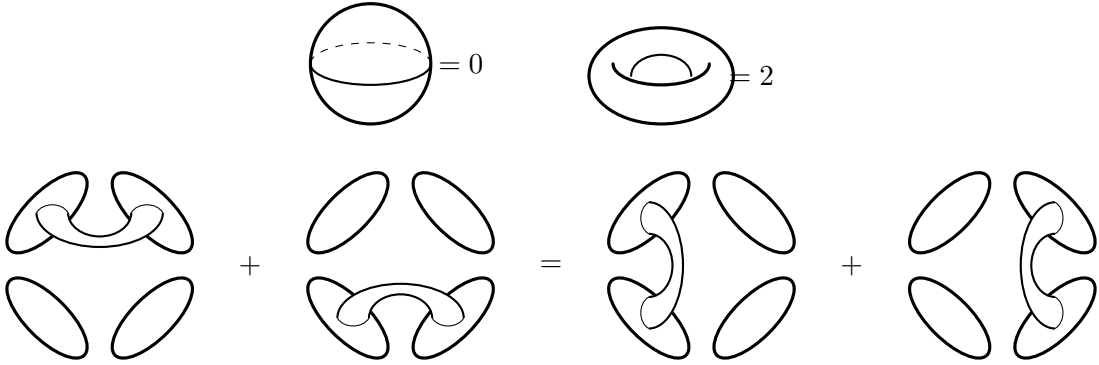
\begin{figure}[H] 
\begin{tikzpicture}[scale=0.8]
\draw[very thick] (0,0) circle (1);
\draw[thick] (-1,0) arc (180:360:1 and 0.35);
\draw[dashed] (1,0) arc (0:180:1 and 0.35);
\node at (1.5,0) {$=0$};
\end{tikzpicture} \hspace{1cm}
\begin{tikzpicture}[scale=0.8, line cap=round, line join=round]
\draw[very thick] (0,0) ellipse (1.2 and 0.8);
\draw[very thick] (-.8,.2) arc (180:360:0.8 and 0.35);
\draw[thick] (0.5,0) arc (0:180:0.5 and 0.35);
\node at (1.5,0) {$=2$};
\end{tikzpicture}
\vspace{0.5cm}

\begin{tikzpicture}
    \begin{scope}[scale=0.7]
   \draw[very thick, shift={(-1,1)}, rotate=45] (0,0) ellipse (1 and 0.4);
   \draw[very thick, shift={(1,-1)}, rotate=45] (0,0) ellipse (1 and 0.4);
  \draw[very thick, shift={(-1,-1)}, rotate=-45] (0,0) ellipse (1 and 0.4);
  \draw[very thick, shift={(1,1)}, rotate=-45] (0,0) ellipse (1 and 0.4);
\fill[white]
  (-0.6,1)
  arc (180:360:0.6 and 0.45)
  --
  (0.6,1)
  to[out=60, in=120] (1.2,1)
  --
  (1.2,1)
  arc (0:-180:1.2 and 0.65)
  --
  (-1.2,1)
  to[out=60, in=120] (-0.6,1)
  -- cycle;

  \draw[thick] (-0.6,1) arc (180:360:0.6 and 0.45);
 \draw[thick] (-1.2,1) arc (180:360:1.2 and 0.65); 
 \draw (-1.2,1) to[out=60, in=120] (-0.6,1);
 \draw (0.6,1) to[out=60, in=120] (1.2,1); 
    \end{scope}
    \node at (2,0) { $+$};
     \begin{scope}[shift={(4,0)},scale=0.7,rotate=180]  
    \draw[very thick, shift={(-1,1)}, rotate=45] (0,0) ellipse (1 and 0.4);
     \draw[very thick, shift={(1,-1)}, rotate=45] (0,0) ellipse (1 and 0.4);
    \draw[very thick, shift={(-1,-1)}, rotate=-45] (0,0) ellipse (1 and 0.4);
    \draw[very thick, shift={(1,1)}, rotate=-45] (0,0) ellipse (1 and 0.4);
    \fill[white]
  (-0.6,1)
  arc (180:360:0.6 and 0.45)
  --
  (0.6,1)
  to[out=60, in=120] (1.2,1)
  --
  (1.2,1)
  arc (0:-180:1.2 and 0.65)
  --
  (-1.2,1)
  to[out=60, in=120] (-0.6,1)
  -- cycle;
 \draw[thick] (-0.6,1) arc (180:360:0.6 and 0.45);
 \draw[thick] (-1.2,1) arc (180:360:1.2 and 0.65);
 \draw (-1.2,1) to[out=60, in=120] (-0.6,1);
 \draw (0.6,1) to[out=60, in=120] (1.2,1); 
    \end{scope}
 \node at (6,0) { $=$};
\begin{scope}[shift={(8,0)},scale=0.7,rotate=90]  
\draw[very thick, shift={(-1,1)}, rotate=45] (0,0) ellipse (1 and 0.4);
\draw[very thick, shift={(1,-1)}, rotate=45] (0,0) ellipse (1 and 0.4);
\draw[very thick, shift={(-1,-1)}, rotate=-45] (0,0) ellipse (1 and 0.4);
\draw[very thick, shift={(1,1)}, rotate=-45] (0,0) ellipse (1 and 0.4);

    \fill[white]
  (-0.6,1)
  arc (180:360:0.6 and 0.45)
  --
  (0.6,1)
  to[out=60, in=120] (1.2,1)
  --
  (1.2,1)
  arc (0:-180:1.2 and 0.65)
  --
  (-1.2,1)
  to[out=60, in=120] (-0.6,1)
  -- cycle;
 \draw[thick] (-0.6,1) arc (180:360:0.6 and 0.45);
 \draw[thick] (-1.2,1) arc (180:360:1.2 and 0.65);
 \draw (-1.2,1) to[out=60, in=120] (-0.6,1);
 \draw (0.6,1) to[out=60, in=120] (1.2,1); 
    \end{scope} 
    \node at (10,0) { $+$};
        \begin{scope}[shift={(12,0)},scale=0.7,rotate=270]
\draw[very thick, shift={(-1,1)}, rotate=45] (0,0) ellipse (1 and 0.4);
\draw[very thick, shift={(1,-1)}, rotate=45] (0,0) ellipse (1 and 0.4);
\draw[very thick, shift={(-1,-1)}, rotate=-45] (0,0) ellipse (1 and 0.4);
\draw[very thick, shift={(1,1)}, rotate=-45] (0,0) ellipse (1 and 0.4);
    \fill[white]
  (-0.6,1)
  arc (180:360:0.6 and 0.45)
  --
  (0.6,1)
  to[out=60, in=120] (1.2,1)
  --
  (1.2,1)
  arc (0:-180:1.2 and 0.65)
  --
  (-1.2,1)
  to[out=60, in=120] (-0.6,1)
  -- cycle;
 \draw[thick] (-0.6,1) arc (180:360:0.6 and 0.45);
 \draw[thick] (-1.2,1) arc (180:360:1.2 and 0.65);
 \draw (-1.2,1) to[out=60, in=120] (-0.6,1);
 \draw (0.6,1) to[out=60, in=120] (1.2,1); 
    \end{scope}
\end{tikzpicture}

\caption{ Bar-Natan relations.}\label{f.bn_relations}
\end{figure}

Let $Kob$ be the category of chain complexes over $\Mat(\Cob^3_{/l})$, whose objects are chains of finite length
\[ \cdots \longrightarrow \Omega^{r-1} \stackrel{d_{r-1}}{\longrightarrow}\Omega^r \stackrel{d_r}{\longrightarrow} \Omega^{r+1} \longrightarrow  \cdots, \]
for which the composition $d_r\circ d_{r-1}$ is 0 for all $r$. 

Given a tangle $T$, Bar-Natan's formal chain complex $[T]$, when regarded as an object in $Kob$, is an up-to-homotopy invariant of 
$T$ (see~\cite[Theorem 1]{BN05}). 

For the scope of this paper, it is convenient to work with the category $Cob^3_{\bullet/l'}$, where $Cob^3_{\bullet}$ has the same objects as $Cob^3$, but now cobordisms are allowed to have dots, where a dot is associated a formal degree of $-2$. The category $Cob^3_{\bullet/l'}$ is then obtained by reducing the category $Cob^3_{\bullet}$ modulo local relations involving dotted cobordisms. The new relations $l'$ are the initial relations $l$, now rewritten in terms of cobordisms with dots. For details, we refer the reader to~\cite[Section 11.2]{BN05}.

\subsection{Abstract Gaussian elimination for chain complexes}

\begin{lemma}[Gaussian elimination \cite{bar2007fast}] \label{l.geliminate}
If $\phi: b_1 \rightarrow b_2$ is an isomorphism, then the four term complex segment 
\[ \cdots \rightarrow [C]  \stackrel{\left( \begin{array}{c} \alpha \\ \beta \end{array} \right)}{\rightarrow} \left[  \begin{array}{c} b_1 \\ D \end{array} \right]   \stackrel{\left( \begin{array}{cc} \ \phi & \delta \\ \gamma & \epsilon \end{array} \right) }{\rightarrow} \left[  \begin{array}{c} b_2 \\ E \end{array} \right]  \stackrel{\left( \begin{array}{cc} \mu & v\end{array}\right)}{\rightarrow}  [F] \rightarrow \cdots   \] 
is homotopy equivalent to the complex segment 
\[ \cdots \rightarrow  [C] \stackrel{(\beta)}{\rightarrow}  [D]   \stackrel{(\epsilon-\gamma\phi^{-1}\delta)}{\rightarrow} [E]  \stackrel{(\nu)}{\rightarrow} [F] \rightarrow \cdots   \] 
\end{lemma}

Let $\alpha$ be a component of the differential in a chain complex and suppose that $\alpha$ is an isomorphism. By \textit{applying Gaussian elimination along $\alpha$} we will mean creating a new complex by replacing the component corresponding to $\alpha$ using Lemma \ref{l.geliminate}. 

Delooping a closed component of a complete resolution $T_{\sigma}$, we obtain two copies of $T_{\sigma}$ in the complex graded by $q$ and  $q^{-1}$. The following lemma (with slightly different notation) is well-known \cite{bar2007fast}. 

\begin{lemma}[Delooping] \label{l.deloop}
Refer to Figure \ref{fig:BN-deloop-square}.
If an object $T$ in $Cob^3_{\bullet/l'}$ is such that $T = T' \cup \ell$, where $\ell$ is a closed loop, then $T$ is isomorphic to the direct sum of two copies $qT'$ and $q^{-1}T'$ of $T'$, one taken with a degree shift of $+1$ and one with degree shift of $-1$. 
\end{lemma}

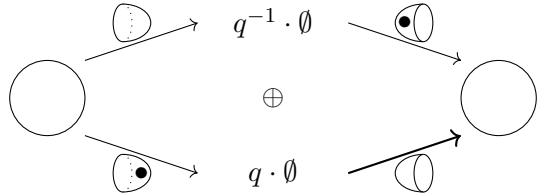
\begin{figure}[H]
    \centering
 \begin{tikzpicture}[scale=.5]
\draw (0,0) circle (1cm);
\draw[->] (1,1) -- (4,2);
\begin{scope}[xshift=2cm, yshift=2cm, scale=.5]
    \kap
\end{scope}
\begin{scope}[yscale=-1]
    \draw[->] (1,1) -- (4,2);
    \begin{scope}[xshift=2cm, yshift=2cm, scale=.5]
        \kapdot
    \end{scope}
\end{scope}
\node (empty1) at (6,2) {$q^{-1} \cdot \emptyset$};
\node (empty2) at (6,-2) {$q \cdot \emptyset$};
\node (oplus) at (6,0) {$\oplus$};
\draw[->] (8,2) -- (11,1);
\begin{scope}[xshift=10cm, yshift=2cm, scale=.5]
    \kupdot
\end{scope}
\begin{scope}[yscale=-1]
    \draw[thick, ->] (8,2) -- (11,1);
    \begin{scope}[xshift=10cm, yshift=2cm, scale=.5]
        \kup
    \end{scope}
\end{scope}
\draw (12,0) circle (1cm);
\end{tikzpicture}
    \caption{Delooping diagram from \cite{bar2007fast}.}
    \label{fig:BN-deloop-square}

\end{figure}

We denote the map in Lemma \ref{l.deloop} from an object $T$ containing a closed loop to a copy $q^{\pm 1}T'$ from delooping that closed loop by $D_{\pm}$. The corresponding maps from $q^{\pm 1}T'$ to $T$ are denoted by $L_{\pm}$. 
\begin{lemma} \label{l.mergesplit}
Suppose we have a map $m:T_1 \cup \ell \rightarrow T_2$, where $\ell$ is a closed loop and $m$ merges the closed loop with a saddle. Then the composition $m \circ L_+: qT_1 \rightarrow T_2$ is the identity map. Similarly, suppose we have a map $\Delta: T_1 \rightarrow T_2 \cup \ell$, where $\ell$ is a closed loop and $\Delta$ splits an arc in $T_1$ to produce the closed loop $\ell$. Then the composition $D_-\circ \Delta: T_1 \to q^{-1}T_2$ is the identity map. See Figure \ref{f.splitiso}, where the components of the isomorphisms are dashed. \\

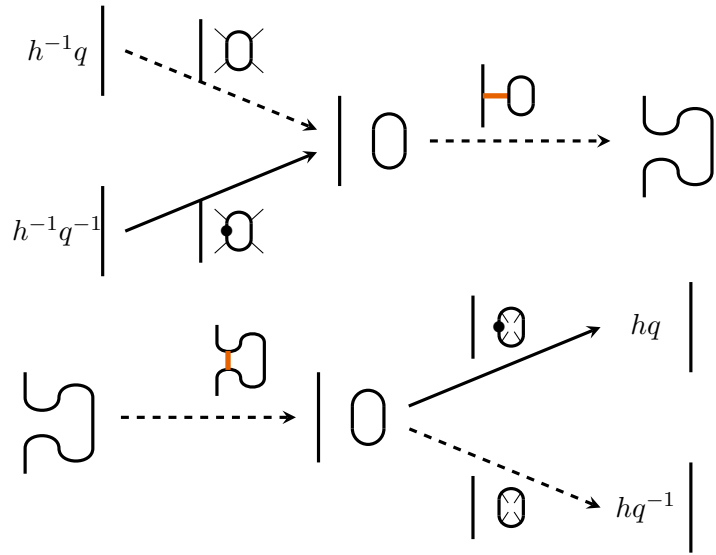
\begin{figure}[H]
 \begin{tikzpicture}[scale=.6]
\node (A) at (-6,2){};
 \node (B) at (-6,-2){};
 \node (C) at (0,0){};
 \node (D) at (6,-0){};  

 \begin{scope}[shift={(A)}]
 \node at (-1,0){ $h^{-1}q$};
 \draw[very thick] (0,-1) to (0,1);
 \end{scope}

\begin{scope}[shift={(-3.25,2)}, scale=.7,xscale=1.1]
 \draw[ very thick] (-.75,-1) to (-.75,1);
\draw [very thick](0,-0.25) arc (180:360:0.35cm);
   \draw[very thick] (0,-.25) to (0,.25); 
      \draw[very thick] (.7,-.25) to (.7,.25); 
    \draw [very thick](0,0.25) arc (180:0:0.35cm);
     %
    \draw (-0.35,-.7) -- (0,-0.35);
     \draw (-0.35,0.7) -- (0,0.35);
      \draw (1.05,-.7) -- (0.7,-0.35);
     \draw (1.05,.7) -- (0.7,0.35);
\end{scope} 
 \draw[-stealth, very thick, dashed] (-5.5,2) to (-1.25,0.25);
 
 \begin{scope}[shift={(B)}]
  \node at (-1,0){ $h^{-1}q^{-1}$};
 \draw[ very thick] (0,-1) to (0,1);
 \end{scope}
 
 \begin{scope}[shift={(-3.25,-2)}, scale=.7,xscale=1.1]
 \draw[ very thick] (-.75,-1) to (-.75,1);
\draw [very thick](0,-0.25) arc (180:360:0.35cm);
   \draw[very thick] (0,-.25) to (0,.25); 
      \draw[very thick] (.7,-.25) to (.7,.25); 
    \draw [very thick](0,0.25) arc (180:0:0.35cm);
    \node at (0,0)[scale=1] {$\bullet$};
    \draw (-0.35,-.7) -- (0,-0.35);
     \draw (-0.35,0.7) -- (0,0.35);
      \draw (1.05,-.7) -- (0.7,-0.35);
     \draw (1.05,.7) -- (0.7,0.35);
\end{scope} 
  \draw[-stealth, very thick] (-5.5,-2) to (-1.25,-.25);
 
\begin{scope}[shift={(C)}]
 \draw[ very thick] (-.75,-1) to (-.75,1);
\draw [very thick](0,-0.25) arc (180:360:0.35cm);
   \draw[very thick] (0,-.25) to (0,.25); 
      \draw[very thick] (.7,-.25) to (.7,.25); 
    \draw [very thick](0,0.25) arc (180:0:0.35cm);
\end{scope} 

\begin{scope}[shift={(3,1)}, scale=.7, xscale=1.1]
 \draw[ very thick] (-.75,-1) to (-.75,1);
 \draw[line width=2pt, Orange]  (-.75,0) to (0,0);
\draw [very thick](0,-0.25) arc (180:360:0.35cm);
   \draw[very thick] (0,-.25) to (0,.25); 
      \draw[very thick] (.7,-.25) to (.7,.25); 
    \draw [very thick](0,0.25) arc (180:0:0.35cm);
\end{scope} 

  \draw[-stealth, very thick, dashed] (1.25,0) to (5.25,0);
  
\begin{scope}[shift={(D)}, shift={(0,-.5)}, yscale=.75]
  \draw[very thick] (0,-1)  to (0,-.25) .. controls ++(0,.5) and ++(0,.5) .. (.75,-.25).. controls ++(0,-.5) and ++(0,-.5)..(1.5,-.25) to (1.5,.5)to (1.5,1.25).. controls ++(0,.5) and ++(0,.5)..(.75,1.25) .. controls ++(0,-.5) and ++(0,-.5) .. (0,1.25) to (0,2);
  \end{scope}
 
\end{tikzpicture}

\begin{tikzpicture}[xscale=-1, scale=.6]
\node (A) at (-7.5,2){};
 \node (B) at (-7.5,-2){};
 \node (C) at (0,0){};
 \node (D) at (7.25,-0){};  
 
 \begin{scope}[shift={(A)},xscale=-1]
 \node at (-1,0){ $hq$};
 \draw[ very thick] (0,-1) to (0,1);
 \end{scope}

\begin{scope}[shift={(-3.25,2)}, scale=.7,xscale=-1.1]
 \draw[ very thick] (-.75,-1) to (-.75,1);
\draw [very thick](0,-0.25) arc (180:360:0.35cm);
   \draw[very thick] (0,-.25) to (0,.25); 
      \draw[very thick] (.7,-.25) to (.7,.25); 
    \draw [very thick](0,0.25) arc (180:0:0.35cm);
 \node at (0,0)[scale=1] {$\bullet$};
    \draw (0.65,-0.5) --(0.35,0) -- (0.05,-0.5);
     \draw (0.65,0.5)--(0.35,0) -- (0.05,0.5);
      \draw [fill=white, white](0.5,0) arc (0:360:0.15cm);
\end{scope} 
 \draw[stealth-, very thick] (-5.5,2) to (-1.25,0.25);
 \begin{scope}[shift={(B)},xscale=-1]
  \node at (-1,0){ $hq^{-1}$};
 \draw[ very thick] (0,-1) to (0,1);
 \end{scope}

 \begin{scope}[shift={(-3.25,-2)}, scale=.7,xscale=-1.1]
 \draw[ very thick] (-.75,-1) to (-.75,1);
\draw [very thick](0,-0.25) arc (180:360:0.35cm);
   \draw[very thick] (0,-.25) to (0,.25); 
      \draw[very thick] (.7,-.25) to (.7,.25); 
    \draw [very thick](0,0.25) arc (180:0:0.35cm);
 \draw (0.65,-0.5) --(0.35,0) -- (0.05,-0.5);
     \draw (0.65,0.5)--(0.35,0) -- (0.05,0.5);
      \draw [fill=white, white](0.5,0) arc (0:360:0.15cm);
\end{scope} 

  \draw[stealth-, very thick, dashed] (-5.5,-2) to (-1.25,-.25);
 
\begin{scope}[shift={(C)},xscale=-1]
 \draw[ very thick] (-.75,-1) to (-.75,1);
\draw [very thick](0,-0.25) arc (180:360:0.35cm);
   \draw[very thick] (0,-.25) to (0,.25); 
      \draw[very thick] (.7,-.25) to (.7,.25); 
    \draw [very thick](0,0.25) arc (180:0:0.35cm);
\end{scope}

\begin{scope}[shift={(3,1)}, scale=.7,xscale=-1,yscale=.75]
 \draw[very thick] (0,-1)  to (0,-.25) .. controls ++(0,.5) and ++(0,.5) .. (.75,-.25).. controls ++(0,-.5) and ++(0,-.5)..(1.5,-.25) to (1.5,.5)to (1.5,1.25).. controls ++(0,.5) and ++(0,.5)..(.75,1.25) .. controls ++(0,-.5) and ++(0,-.5) .. (0,1.25) to (0,2);
 \draw [line width=2pt, Orange] (0.35,0.1)--(0.35,.9);
\end{scope} 

  \draw[stealth-, very thick, dashed] (1.25,0) to (5.25,0);
  
\begin{scope}[shift={(D)}, shift={(0,-.5)}, yscale=.75,xscale=-1]
  \draw[very thick] (0,-1)  to (0,-.25) .. controls ++(0,.5) and ++(0,.5) .. (.75,-.25).. controls ++(0,-.5) and ++(0,-.5)..(1.5,-.25) to (1.5,.5)to (1.5,1.25).. controls ++(0,.5) and ++(0,.5)..(.75,1.25) .. controls ++(0,-.5) and ++(0,-.5) .. (0,1.25) to (0,2);
  \end{scope}
 
\end{tikzpicture}

\caption{Merging and splitting a circle.}\label{f.splitiso} 
\end{figure}

We note that multiplication by  $q^{\pm 1}$ changes the quantum grading by $\pm 1$. 
\end{lemma} 

\subsection{The Temperley-Lieb monoid and the Jones Normal Form} \label{ss.tlmjnf}

\begin{definition}[$\TL$-moves]  \label{n.tlmoves}
Let $n \in \mathbb{N}$, where $n\geq 2$. Let $\TL_n$ be the Temperley-Lieb monoid on $n$ strands consisting of words generated by the set
$E = \{ e_1, \ldots, e_{n-1}\} $
quotiented by the following relations: 
\begin{enumerate}
    \item[\namedlabel{item:TLA}{[TL-A]}] $e_i^2 \sim e_i$
    \item[\namedlabel{item:TLB}{[TL-B]}] $e_i \sim e_i e_j e_i$ for $|i-j| =1$
    \item[\namedlabel{item:TLC}{[TL-C]}] $e_i e_j \sim e_j e_i$ for $|i-j|\geq 2$.
\end{enumerate}
\end{definition}

\begin{remark}
The Temperley-Lieb algebra is usually defined over a commutative ring $R$ with a distinguished element $\delta \in R$ such that relation \ref{item:TLA} is given as $e_i^2 \sim \delta e_i$. For our purposes, the reader may assume $\delta=1$, though we honestly work with the underlying words, treating the relations as ``moves'' rather than equivalences.
\end{remark}

\begin{notation}
For two words $W, W'$ in $\TL_n$, we write $W = W'$ when they are the same word, and we write $W\sim W'$ when $W$ and $W'$ are in the same equivalence class in $\TL_n$. 
\end{notation}

\begin{notation}
\label{notation:seven-moves}
We further split up the Temperley-Lieb relations into \emph{moves} which replace a word with an equivalent word. We define the seven types of moves $a_\pm, b_\pm^{\pm1}, c$ as follows.
\begin{itemize}
    \item From relation \ref{item:TLA}, we have two moves:
    \[ e_i \map{a_+} e_i^2 \map{a_-} e_i. \]
    \item From relation \ref{item:TLB}, we have four moves:
    \[ e_i \map{b_+^1} e_i e_{i+1} e_i \map{b_-^1} e_i \]
    and 
    \[ e_i \map{b_+\inv} e_i e_{i-1}e_i \map{b\inv_-} e_i.\]
    \item Relation \ref{item:TLC} gives just one type of move:
    \[ e_ie_j \map{c} e_je_i\]
    where $|i-j|\geq 2$. 
\end{itemize}
\end{notation} 

\begin{definition}
Let $W, W'$ be equivalent words in $\TL_n$. 
A \emph{path} $P$ from $W$ to $W'$ is a sequence $(W_i)_{i=0}^k$ of words $W_i \in \TL_n$ such that $W = W_0$, $W' = W_k$, and for all $i < k$, $W_{i+1}$ is obtained from $W_i$ via a single move from Notation \ref{notation:seven-moves}. 
Note that for all $i$, $W_i \sim W \sim W'$. 

A \emph{subpath} of $P$ is a subsequence $(W_i)_{i = j}^l$ of $P$, where $0 \leq j \leq l \leq k$. 

We will often denote paths in $\TL_n$ in the following form, by including the moves used to obtain each $W_{i+1}$ from the former word $W_i$:
\[
W_0 \map{M_1} W_1 \map{M_2} \cdots \map{M_k} W_k,
\]
where the data of $(W_i)_{i=1}^k$ can be determined from the data of the path $P$. As a shorthand, we will also write $W \leadsto W'$ to specify there is a path from $W$ to $W'$. 

\end{definition}

\begin{definition}
    Let $P$ be a path $(W_i)_{i=0}^k$. If for all $i<k$, $\len(W_i) \leq \len(W_{i+1})$, we say that $P$ is \emph{monotone increasing (in length)}. 
On the other hand, if for all $i<k$, $\len(W_i)\geq \len(W_{i+1})$, then we say that $P$ is \emph{monotone decreasing (in length)}.
\end{definition}

\begin{definition}[Jones Normal Form \cite{Kassel-Turaev}] \label{d.jnf}
    For $1\leq k \leq n-1$, let $E_{n, k}$ be the set of $2k$-tuples $(i_1,\ldots, i_k, j_1, \ldots, j_k)$ of integers such that 
    \[0 < i_1 < i_2 < \ldots <i_k < n, \qquad 0 < j_1 < j_2 < \ldots < j_k < n, \]
    and 
    \[j_1 \leq i_1, j_2 \leq i_2, \ldots, j_k \leq i_k.\]
\end{definition}
For such a tuple $t = (i_1, \ldots, i_k, j_1, \ldots, j_k)$, set
\[W_t: = (e_{i_1}e_{i_1-1}\cdots e_{j_1})(e_{i_2}e_{i_2-1} \cdots e_{j_2})\cdots(e_{i_k}e_{i_k-1} \ldots e_{j_k}).  \]

If a word $W \in TL_n$ is such that $W = W_t$ for some tuple $t$, then $W$ is said to be in \textit{Jones normal form}. It is known that every word in $\TL_n$ is equivalent to a word in Jones normal form. In fact, given any word $W \in \TL_n$, there exists a monotone decreasing path from $W$ to its Jones normal form $\JNF(W)$.

\begin{proposition}{\cite[Lemma 4.1.2]{Jones-index}}
\label{prop:monotone-path}

Given $W \in \TL_n$, there exists a monotone decreasing path $P = (W_i)_{i=0}^k$, where $W_0 = W$ and $W_k = \JNF(W)$. That is, there is a monotone decreasing path that transforms $W$ to its Jones normal form $\JNF(W)$.

\begin{proof}

We provide a proof by induction on the index $n$.
    In the following, we say \emph{reduce greedily} to mean applying the following moves as much as one can:
    \begin{itemize}
        \item $e_ie_i \to e_i$
        \item $e_ie_{i\pm 1}e_i \to e_i$
    \end{itemize}
 The result of a greedy reduction is a new word such that $e_i^2$ and $e_ie_{i\pm 1}e_i$ do not appear. 
    Throughout the algorithm below, we reduce greedily whenever possible.

    \paragraph{\textbf{Base case $n=2$:}} 
    Any $W \in \TL_2$ is of the form $e_1^l$, for some $l \geq 1$. If $l>1$, we apply the greedy reduction; in fact we apply the move $e_ie_i \to e_i$, until the resulting word is $e_1 = \JNF(W)$. 

    \paragraph{\textbf{Induction step}:} Assume that the statement holds for any word of index $n$ and let $W \in \TL_{n+1}$, where the maximum index $e_n$ appears at least once.  (Otherwise, we are done by the induction hypothesis, since if $e_n$ does not appear in $W$, then $W \in \TL_{n}$). After greedy reduction, our modified word will be of the form
    \[
        W_0 \cdot e_n \cdot W_1 \cdot e_n \cdot W_2 \cdot ... \cdot e_n \cdot W_{k},
    \]
    where $W_i \in \TL_n$, for all $0\leq i \leq k$. By the induction hypothesis, we may modify this word until all the $W_i$ are in Jones normal form, i.e. we have the following word:
    \[
        \JNF(W_0) \cdot e_n \cdot \JNF(W_1) \cdot e_n \cdot \JNF(W_2) \cdot ... \cdot e_n \cdot \JNF(W_{k}).
    \]

    Now push all $k$ instances of $e_n$ as far right as possible using the commuting relations from $\TL_n$, and then reduce the word by using the relation $e_n^2 \to e_n$ whenever possible. Then, each instance of $e_n$ will either be (1) followed by a decreasing run beginning with $e_{n-1}$, or (2) is at the very end of the word. Note also that $e_{n-1}$ appears exactly once for $\JNF(W_i)$, for all $0 \leq i \leq k$.

Writing $V_i = \JNF(W_i)$, our modified word now takes the form 
    \[
        V_0 \cdot (e_n e_{n-1}) \cdot V_1 \cdot (e_ne_{n-1}) \cdot V_2 \cdot \ldots \cdot V_{k-1} (e_n) V_{k},
    \]
    where $V_{k}$ is either the identity or is a decreasing run beginning with $e_{n-1}$. 
    Observe that for all $1 \leq i \leq k$, $V_i \in \TL_{n-1}$, and thus it commutes with $e_n$. We may thus commute the second copy of $e_n$ to the left until it abuts against the $e_{n-1}$ directly following the first instance of $e_n$, commute the fourth copy of $e_n$ to the left so that it abuts against $e_{n-1}$ directly following the third instance of $e_n$, and so on. The result of commuting the second copy of $e_n$ as explained above is a word of the form
    \[
        V_0 \cdot (e_n e_{n-1} e_n) \cdot V_1 \cdot (e_{n-1}) \cdot V_2 \cdot \ldots \cdot V_{k-1} (e_n) V_{k}
    \]
    or
    \[
        V_0 \cdot (e_n e_{n-1} e_n) \cdot V_1 \cdot (e_{n-1}) \cdot V_2 \cdot \ldots \cdot V_{k-1}(e_n),
    \]
    depending on the ending of the word in $e_n$ or $V_{k-1}$. We replace now the instance $e_ne_{n-1}e_n$ with $e_n$. This reduces the number of instances of $e_n$ in the word.

    We repeat this algorithm (of commuting $e_n$ to the left until it abuts an $e_{n-1}$ followed by the replacement of $e_ne_{n-1}e_n$ with $e_n$) $\lfloor \frac{k}{2} \rfloor$ times until the number of instances of $e_n$ is at most one. If $e_n$ appears in the resulting word, then it must be the beginning of a decreasing run at the very end of the word; that is, our modified word is of the form
    \[
        U \cdot (e_n e_{n-1} \ldots e_m),
    \]
    where $m \leq n$ and $U \in \TL_n$; note the decreasing run may just be the word $e_n$. By the induction hypothesis, there is a monotone decreasing path taking $U$ to $\JNF(U)$. After applying this path, we arrive at 
    \[
        \JNF(U) \cdot (e_n e_{n-1} \ldots e_m),
    \]
    which is in Jones Normal Form.
\end{proof}
\end{proposition}

Equivalently, given any word $W$, there is a monotone increasing path $\JNF(W) \leadsto W$. This can be obtained by taking the reverse path from the proof of Proposition \ref{prop:monotone-path}.

\section{Gaussian elimination} \label{s.Gaussian_Elimination}
\subsection{Gaussian elimination isomorphisms}
\label{ss.Distinguishediso}

\begin{notation}[braid generators] \label{n.braidgen}
We use  $\sigma_i$ to denote the braid generator involving  the $i$th and the $(i+1)$th strand of the braid, and $\sigma_i^{-1}$ to represent the inverse of the braid generator $\sigma_i$;  see Figure \ref{f.bgenerator} for our convention.

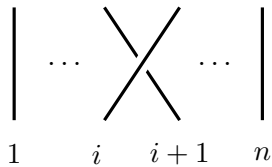
\begin{figure}[H] 
\begin{tikzpicture}[xscale=1,yscale=1.5]
\draw[very thick] (-1.2,0)--(-1.2,1);
\node at (-1.2,-.3){$1$};
\node at (-0.5,.5){$\dots$};
\draw[very thick] (1,0) to (0,1);
\node at (-0.1,-.3){$i$};
\draw[line width = 6pt,white,->] (0,0) to (1,1);
\draw[very thick] (0,0) to (1,1);
\node at (1,-.3){$i+1$};
\node at (1.5,.5){$\dots$};
\draw[very thick] (2.1,0)--(2.1,1);
\node at (2.1,-.3){$n$};
\end{tikzpicture}
\caption{The braid generator $\sigma_i.$} \label{f.bgenerator} 
\end{figure} 
\end{notation}

\begin{notation}[barred braid word] \label{n.barredbraidw}
Let $\beta$ be a braid on $n$ strands.
A word $w$ of (barred or unbarred) numbers $1, 2, \ldots, n-1$ represents a Kauffman state on $\beta$ as follows. We replace each $\sigma_i$ in $\beta$ by $i$ or $\overline{i}$ to obtain the word of numbers corresponding to the Kauffman state.  If the Kauffman state chooses the $0$-resolution on $\sigma_i$, then we replace $\sigma_i$ by $i$, otherwise we replace $\sigma_i$ by $\bar{i}$. 

For example, a barred braid word $1233 \bar{2}\bar{1}$ represents a Kauffman state which chooses the $0$-resolution on the first four crossings in the braid $\sigma_1\sigma_2\sigma_3\sigma_3 \sigma_2\sigma_1$, and the $1$-resolution on the last two crossings. We will use the subscript on a digit of a barred braid word to further indicate its position in the braid word, reading from left to right. For example, $1_12_23_33_4 \bar{2}_5\bar{1}_6$. 

We will use the verb ``to bar'' to denote changing the resolution from $0$ to $1$ on a crossing corresponding to a digit in the barred braid word. We will also use a lowercase letter to denote a barred word $w = w(\sigma)$ representing a Kauffman state $\sigma$, and use a barred digit to refer to the respective crossing. See Figure \ref{f.barring} for an illustrated example. 
\end{notation}

\begin{figure}[H]
\begin{center} 
\begin{tikzpicture}
\begin{scope}[rotate = 90]
\begin{scope}
\draw[very thick] (0, -.5) to (0, .5); 
\draw[very thick] (.5, -.5) to (.5, .5);
\draw[very thick] (0, .5) .. controls +(0, .25) and +(0,.25) .. (.5, .5); 
\end{scope}

\begin{scope}
\draw[very thick] (0, 1.25) to (0, 2.25); 
\draw[very thick] (.5, 1.25) to (.5, 2.25);
\draw[very thick] (0, 1.25) .. controls +(0, -.25) and +(0, -.25) .. (.5, 1.25); 
\end{scope}
\end{scope}

\draw[very thick] (-2.25, -.5) to (0.25, -.5);

\draw[very thick] (.5, .5) to (2.25, .5);

\draw[very thick] (-4, -.5) to (-1.5, -.5);

\draw[very thick] (-4, 0) to (-1.5, 0);

\draw[very thick] (-4, .5) to (-1.5, .5);

\draw[very thick] (-4, 1) to (2.25, 1);

\draw[very thick, dashed] (-3.5, -.5) to (-3.5, 0);

\draw[very thick, dashed] (-3, 0) to (-3, 0.5);

\draw[very thick, dashed] (-2.5, .5) to (-2.5, 1);

\draw[very thick, dashed] (-2, .5) to (-2,  1);
 
\begin{scope}[xshift=51, yshift= -14.25, rotate = 90]
\begin{scope}
\draw[very thick] (0, -.5) to (0, .5); 
\draw[very thick] (.5, -.5) to (.5, .5);
\draw[very thick] (0, .5) .. controls +(0, .25) and +(0,.25) .. (.5, .5); 
\end{scope}

\begin{scope}
\draw[very thick] (0, 1.25) to (0, 2.25); 
\draw[very thick] (.5, 1.25) to (.5, 2.25);
\draw[very thick] (0, 1.25) .. controls +(0, -.25) and +(0, -.25) .. (.5, 1.25); 
\end{scope}
\end{scope}
\end{tikzpicture} 
\end{center} 
\caption{\label{f.barring} The Kauffman state $1233 \bar{2}\bar{1}$.}
\end{figure}
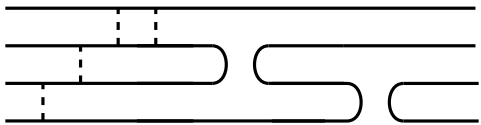 

Let $\TL_n$ be the Temperley-Lieb monoid on $n$ strands defined in Definition \ref{n.tlmoves}. The complete tangle resolution resulting from applying a Kauffman state $\sigma$ represented by a barred word $w$ to a braid $\beta$ can be described as a word of $ W \in\TL_n$, by sending each barred digit $\bar{i} \in w$ to $e_i$, following the order on $w$, reading from left to right. We will follow the convention of using uppercase for the $\TL_n$ word corresponding to the barred braid word for the Kauffman state.

For a Kauffman state represented by $w$, every closed component $C$ in $W$ has a \textit{leftmost crossing} $\ell(C)$ representing the crossing in $\beta$, whose replacement by $e_{i}$ for some $1\leq i \leq n-1$ by the Kauffman state borders the closed component, and whose position in $w$, reading the word from left to right, is the smallest. Similarly, every closed component $C$ in $W$ has a \textit{rightmost crossing} $r(C)$, which borders the closed component, and whose position in $w$ is the largest (see Figure \ref{f.leftmostrightmost}).

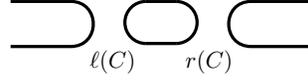
\begin{figure}[H]
\begin{center} 
  \begin{tikzpicture}[every node/.style={scale=.8}]
\begin{scope}[
rotate=-90, scale=.8]  
  \draw[very thick] (.75,2.5) to (.75,1.5) .. controls ++(0,-.5) and ++(0,-.5) .. (0,1.5) to (0,2.5);
\draw [very thick](0,-0.25) arc (180:360:0.35cm);
   \draw[very thick] (0,-.25) to (0,.25); 
      \draw[very thick] (.7,-.25) to (.7,.25); 
    \draw [very thick](0,0.25) arc (180:0:0.35cm);    
\draw[very thick] (0,-2.5)  to (0,-1.5) .. controls ++(0,.5) and ++(0,.5) .. (.75,-1.5)--(.75,-2.5);
\node (l) at (1, -0.8) {$\ell(C)$};
\node (r) at (1,0.8) {$r(C)$};
 \end{scope}

\end{tikzpicture}

\end{center}
\caption{\label{f.leftmostrightmost} Leftmost and rightmost crossings of a closed component $C$.}
\end{figure} 

\begin{definition}[GE pair] \label{d.gepair} In the set of enhanced Kauffman states of the braid $\beta$, a \textit{Gaussian Elimination pair} (\textit{GE pair} for short) is a pair of enhanced Kauffman states $(s, \epsilon)$, $(t, \delta)$, where 
\begin{itemize} 
\item[(G1)] $s = w\bar{i}_j x i_{j'}z $ and $t =  w\bar{i}_j x \bar{i}_{j'}z$, for $j'-j = n-1$, or 
\item[(G2)] $s = w\bar{i}_j x \bar{i}_{j'}z $ and $t =  w\bar{i}_j x_L \overline{(i+1)} x_R\bar{i}_{j'}z$, for $j'-j = n-1$, 
\end{itemize} 

such that their enhancements $\epsilon, \delta$ differ as follows: 
\begin{itemize} 
\item[(G1)] Suppose $s$ and $t$ are as in case (G1) above. Then $T$ has a closed component $C$, marked by $\delta$ with a $-$, such that $\ell(C) = \bar{i}_j$ and $r(C) = \bar{i}_{j'}$.  
\item[(G2)] Suppose $s$ and $t$ are as in  case (G2) above. Then $S$ has a closed component $C$ that $\delta$ marks with a $+$, and $\ell(C) = \bar{i}_{j}$, $r(C) = \bar{i}_{j'}$. 
\end{itemize} 
See Figure \ref{f.gepairs} for an illustration of both cases. 

\begin{figure}[H] 
\centering

 \begin{subfigure}{.75\textwidth}
 \centering
 \begin{tikzpicture}[every node/.style={scale=.7}, scale=.8]
\begin{scope}[shift={(-3,0)}, rotate=-90
] 
\draw[very thick] (.75,2.5) to (.75,1.5) .. controls ++(0,-.5) and ++(0,-.5) .. (0,1.5) to (0,2.5);
 \draw[line width = 1pt, dashed] (0,2)--(.75,2);
\begin{scope}[shift={(0, 2)}]
\draw[very thick] (0,-2.5)  to (0,-1.5) .. controls ++(0,.5) and ++(0,.5) .. (.75,-1.5)--(.75,-2.5);
\end{scope}
 \end{scope}

 \draw[very thick, purple, -stealth] (1,-0.25)--(3,-0.25);

 \node at (2,0){$G1$};
 

\begin{scope}[shift={(8,0)},rotate=-90
]

\begin{scope}[shift={(0,-1)}]
  \draw[very thick] (.75,2.5) to (.75,1.5) .. controls ++(0,-.5) and ++(0,-.5) .. (0,1.5) to (0,2.5);
  \draw [very thick](0,-0.25) arc (180:360:0.35cm);
   \draw[very thick] (0,-.25) to (0,.25); 
      \draw[very thick] (.7,-.25) to (.7,.25); 
    \draw [very thick](0,0.25) arc (180:0:0.35cm);    
    \end{scope}
    
 \begin{scope}[shift={(0,-1)}]
\draw[very thick] (0,-2.5)  to (0,-1.5) .. controls ++(0,.5) and ++(0,.5) .. (.75,-1.5)--(.75,-2.5);
\end{scope}
  \end{scope}
  
\end{tikzpicture}\caption{}
\end{subfigure}

 \begin{subfigure}{.75\textwidth}
  \centering
  \begin{tikzpicture}[every node/.style={scale=.8}]
\begin{scope}[shift={(-5,0)}, rotate=-90, scale=.8] 
 \draw[very thick] (-.75,-2.5)--(-.75,2.5); 
  \draw[very thick] (.75,2.5) to (.75,1.5) .. controls ++(0,-.5) and ++(0,-.5) .. (0,1.5) to (0,2.5);
\draw [very thick](0,-0.25) arc (180:360:0.35cm);
   \draw[very thick] (0,-.25) to (0,.25); 
      \draw[very thick] (.7,-.25) to (.7,.25); 
    \draw [very thick](0,0.25) arc (180:0:0.35cm);    
\draw[very thick] (0,-2.5)  to (0,-1.5) .. controls ++(0,.5) and ++(0,.5) .. (.75,-1.5)--(.75,-2.5);
\node at (0.4,0.2){$+\,\,\,\,\,\,$};
 \end{scope}
 \draw[very thick, purple, -stealth] (-2,-0.25)--(0,-0.25);
 \node at (-1,0){$G2$};
\begin{scope}[shift={(3,0)},rotate=-90, scale=.8]
\draw[very thick] (.75,2.5) to (.75,1.5) .. controls ++(0,-.5) and ++(0,-.5) .. (0,1.5) to (0,2.5);
\begin{scope}[shift={(-.75,-.5)}]
  \draw[very thick] (0,-2)  to (0,0) .. controls ++(0,.5) and ++(0,.5) .. (.75,0).. controls ++(0,-.5) and ++(0,-.5)..(1.5,0) to (1.5,.5)to (1.5,1).. controls ++(0,.5) and ++(0,.5)..(.75,1) .. controls ++(0,-.5) and ++(0,-.5) .. (0,1) to (0,3);
  \end{scope} 
\draw[very thick] (0,-2.5)  to (0,-1.5) .. controls ++(0,.5) and ++(0,.5) .. (.75,-1.5)--(.75,-2.5);
  \end{scope}
\end{tikzpicture}
\caption{}
\end{subfigure}
\caption{ G1 (A) and G2 (B) Gaussian elimination pairs. }\label{f.gepairs}
\end{figure}
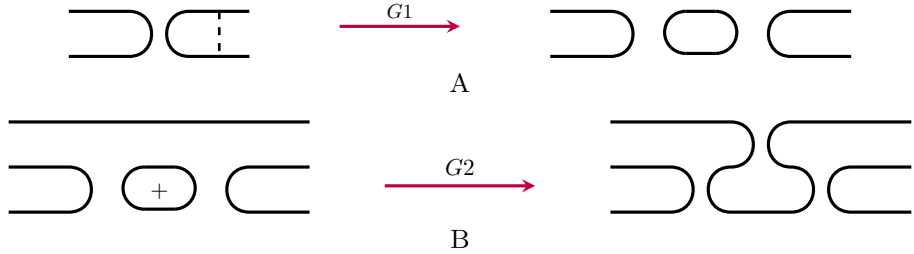 

We also say that the respective (unenhanced) Kauffman states $s$ and $t$ are a GE pair.

\end{definition}

We define Gaussian eliminations G1 and G2 corresponding to the two cases in Definition \ref{d.gepair}.

\begin{definition}[Gaussian elimination isomorphisms] \label{d.gelimiso} \ \\
\underline{Gaussian elimination G1:}
 Suppose we have a GE pair $(s, \epsilon)$, $(t, \delta)$ of type G1. By definition, the Kauffman state $\sigma(s)$ is represented by a word $$s = w\bar{i}_jxi_{j'}z,$$ and the state $\sigma(t)$ is represented by a word $$t = w\bar{i}_jx\bar{i}_{j'}z.$$ 
 
   By Lemma \ref{l.mergesplit}, there is an isomorphism $D_{-} \circ \Delta: (s, \epsilon) \to q^{-1} (t, \delta)$, that is a component of the Khovanov chain complex.
   
  It is possible to apply Gaussian elimination (Lemma \ref{l.geliminate}) to this isomorphism to eliminate both the enhanced Kauffman state $(s, \epsilon)$ and the target $(t, \delta)$.  We refer to such an isomorphism $g = D_{-} \circ \Delta$ as a G1 \emph{isomorphism}. See Figure \ref{f.g1} for such a map $\Delta$.
 
\begin{figure}[H]
\begin{tikzpicture}[every node/.style={scale=.7}, scale=.6]
\begin{scope}[shift={(-6,0)}, rotate=-90, scale=.8] 
\begin{scope}[shift={(.5, 0.9)}];
\node at (-1.5, 1.2){$i_{j'}$};
\node at (-1.5,1.95){$\dots$};
\begin{scope}[shift={(0,-.15)}]
\node at (-1.5,-0.5){$\dots$};
\node at (-1.5,0.1){$\bar{i}_{j}$};
\node at (-1.5,0.75){$\dots$};
\end{scope}
\end{scope}

\draw[very thick] (.75,2.5) to (.75,1.5) .. controls ++(0,-.5) and ++(0,-.5) .. (0,1.5) to (0,2.5);
 \draw[line width = 1pt, dashed] (0,2)--(.75,2);

\begin{scope}[shift={(0, 2)}]
\draw[very thick] (0,-2.5)  to (0,-1.5) .. controls ++(0,.5) and ++(0,.5) .. (.75,-1.5)--(.75,-2.5);
\end{scope}
 \end{scope}

\begin{scope}[shift={(-0.8,1)}, rotate=-90, scale=.6]
 
\begin{scope}[shift={(0.5,0)}];
\node at (-1.5, 1){$\bar{i}_{j}$};
\node at (-1.5, 2){$\bar{i}_{j'}$};
\node at (-1.5,0){$\dots$};
\node at (-1.5,3){$\dots$};
\end{scope}

\draw[very thick] (.75,2.5) to (.75,1.5) .. controls ++(0,-.5) and ++(0,-.5) .. (0,1.5) to (0,2.5);
\begin{scope}[shift={(0, 2)}]
\draw[very thick] (0,-2.5)  to (0,-1.5) .. controls ++(0,.5) and ++(0,.5) .. (.75,-1.5)--(.75,-2.5);
\end{scope}
\draw[line width = 2pt, orange] (0,2)--(.75,2);
  \end{scope}

 \draw[very thick, purple, -stealth] (-2,0)--(2,0);
\begin{scope}[shift={(-2,0)}]
  
\begin{scope}[shift={(8.5,0)},rotate=-90, scale=.8]

\begin{scope}[shift={(.5,1)}];
\node at (-1.5,-1){$\bar{i}_{j'}$};
\begin{scope}[shift={(0,-2.25)}]
\node at (-1.5,-1.75){$\dots$};
\node at (-1.5,-1){$\bar{i}_{j}$};
\node at (-1.5,0){$\dots$};
\node at (-1.5,2){$\dots$};
\end{scope}
\end{scope}

\begin{scope}[shift={(0,-1)}]
  \draw[very thick] (.75,2.5) to (.75,1.5) .. controls ++(0,-.5) and ++(0,-.5) .. (0,1.5) to (0,2.5);
  \draw [very thick](0,-0.25) arc (180:360:0.35cm);
   \draw[very thick] (0,-.25) to (0,.25); 
      \draw[very thick] (.7,-.25) to (.7,.25); 
    \draw [very thick](0,0.25) arc (180:0:0.35cm);    
    \end{scope}
    
 \begin{scope}[shift={(0,-1)}]
\draw[very thick] (0,-2.5)  to (0,-1.5) .. controls ++(0,.5) and ++(0,.5) .. (.75,-1.5)--(.75,-2.5);
\end{scope}
  \end{scope}
  \end{scope}
  
\end{tikzpicture}
\caption{\label{f.g1} The map $\Delta$ for a G1 Gaussian elimination isomorphism.}
\end{figure}
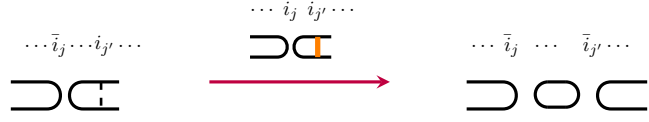 

\underline{Gaussian elimination G2:}
Suppose we have a GE pair $(s, \epsilon)$, $(t, \delta)$ of type G2. By definition, the Kauffman state $\sigma(s)$ is represented by a word $$s = w\bar{i}_jx\bar{i}_{j'}z,$$ and the state $\sigma(t)$ is represented by a word $$t = w\bar{i}_jx_L\overline{(i+1)}x_R\bar{i}_{j'}z.$$

By Lemma~\ref{l.mergesplit}, we may
apply Gaussian elimination (Lemma \ref{l.geliminate}) to the isomorphism $m \circ L_+$ to remove both $(s, \epsilon)$ and $(t, \delta)$ from the chain complex.   We refer to such an isomorphism  $g = m \circ L_+$ as a G2 \emph{isomorphism}. See Figure \ref{f.g2} for such a map $m$.

  \begin{figure}[H]
  
  \begin{tikzpicture}[every node/.style={scale=.8}]

\begin{scope}[shift={(-5,0)}, rotate=-90, scale=.8] 

\node at (-1.5,-1.75){$\dots$};
\node at (-1.5,-.75){$\bar{i}_j$};
\node at (-1.5,0){$\dots$};
\node at (-1.5,.75){$\bar{i}_{j'}$};
\node at (-1.5,1.75){$\dots$};

 \draw[very thick] (-.75,-2.5)--(-.75,2.5); 
  \draw[very thick] (.75,2.5) to (.75,1.5) .. controls ++(0,-.5) and ++(0,-.5) .. (0,1.5) to (0,2.5);
\draw [very thick](0,-0.25) arc (180:360:0.35cm);
   \draw[very thick] (0,-.25) to (0,.25); 
      \draw[very thick] (.7,-.25) to (.7,.25); 
    \draw [very thick](0,0.25) arc (180:0:0.35cm);    
\draw[very thick] (0,-2.5)  to (0,-1.5) .. controls ++(0,.5) and ++(0,.5) .. (.75,-1.5)--(.75,-2.5);
 \end{scope}

 \begin{scope}[shift={(0,1)}, rotate=-90, scale=.6]
 
 \node at (-1.5,-1.75){$\dots$};
\node at (-1.5,0){$\overline{(i+1)}$};
\node at (-1.5,1.75){$\dots$};

\draw[very thick] (-.75,-2.5)--(-.75,2.5); 
  \draw[very thick] (.75,2.5) to (.75,1.5) .. controls ++(0,-.5) and ++(0,-.5) .. (0,1.5) to (0,2.5);
\draw [very thick](0,-0.25) arc (180:360:0.35cm);
   \draw[very thick] (0,-.25) to (0,.25); 
      \draw[very thick] (.7,-.25) to (.7,.25); 
    \draw [very thick](0,0.25) arc (180:0:0.35cm);    
\draw[very thick] (0,-2.5)  to (0,-1.5) .. controls ++(0,.5) and ++(0,.5) .. (.75,-1.5)--(.75,-2.5);
\draw[line width = 2pt, orange] (-.75, 0) -- (0,0);
  \end{scope}
 \draw[very thick, purple, -stealth] (-2,0)--(2,0);
\begin{scope}[shift={(5,0)},rotate=-90, scale=.8]
\node at (-1.5, -2){$\dots$};
\node at (-1.5,-1.4){$\bar{i}_j$};
\node at (-1.5,-.9){$\dots$};
\node at (-1.5,0){$\overline{(i+1)}$};
\node at (-1.5,.9){$\dots$};
\node at (-1.5,1.4){$\bar{i}_{j'}$};
\node at (-1.5, 2){$\dots$};
\draw[very thick] (.75,2.5) to (.75,1.5) .. controls ++(0,-.5) and ++(0,-.5) .. (0,1.5) to (0,2.5);
\begin{scope}[shift={(-.75,-.5)}]
  \draw[very thick] (0,-2)  to (0,0) .. controls ++(0,.5) and ++(0,.5) .. (.75,0).. controls ++(0,-.5) and ++(0,-.5)..(1.5,0) to (1.5,.5)to (1.5,1).. controls ++(0,.5) and ++(0,.5)..(.75,1) .. controls ++(0,-.5) and ++(0,-.5) .. (0,1) to (0,3);
  \end{scope} 
   \draw[very thick] (0,-2.5)  to (0,-1.5) .. controls ++(0,.5) and ++(0,.5) .. (.75,-1.5)--(.75,-2.5);
  \end{scope}
\end{tikzpicture}
 \caption{The map $m$ for a G2 Gaussian elimination isomorphism. }\label{f.g2} 
 \end{figure}
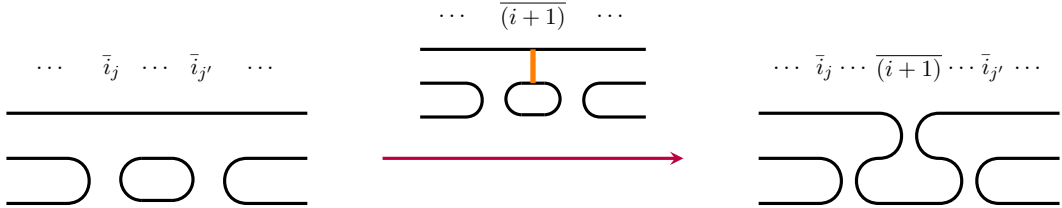 
 
Given a Gaussian elimination isomorphism $g_i: (s, \epsilon) \rightarrow (t, \delta)$, where $i = 1$ indicates a G1 Gaussian elimination isomorphism, and $i = 2$ indicates a G2 Gaussian elimination isomorphism, we call $(s, \epsilon)$ the \textit{source} of $g_i$, and $(t, \delta)$ the \textit{target} of $g_i$. 
 
\subsection{Algorithm: choosing a distinguished Gaussian elimination isomorphism}
\label{ss.algorithm}
Typically, given a pair of enhanced Kauffman states $(s, \epsilon)$ and $(t, \delta)$ that forms a GE pair following Definition \ref{d.gepair}, there could be different Gaussian eliminations isomorphisms supported on different pairs of subwords of $s$ and $t$. 
Given an enhanced Kauffman state $(w, \varepsilon)$, we use the following steps to choose (at most one) isomorphism G1 or G2 whose source is $(w, \varepsilon)$. We refer to such an isomorphism as a \textit{distinguished isomorphism}. 
\end{definition} 

\begin{definition} \label{d.setGaussianG}
We build the set $\G$  of distinguished Gaussian elimination isomorphisms inductively, where we set $\G_0 = \emptyset$ and $$\G_i := \cup_{k \leq i} \G_{k-1}.$$ Let $(w, \epsilon)$ be an enhanced Kauffman state in the Khovanov chain complex $\CKh(\beta)$ with homological and quantum gradings $i$ and $j$, respectively. Let $m(\beta)$ be the lowest homological degree of the chain $\CKh(\beta)$, and let $M(\beta)$ be the highest homological degree. We begin by examining $(w, \epsilon) \in \CKh_{\underline{i}, *}(\beta)$. 
\end{definition}

We start by looking at a generator $(w, \epsilon)$ of the Khovanov chain complex $\CKh_{i, *}(\beta)$ of minimum homological degree $i(w) = m(\beta)$. To determine if there exists a distinguished isomorphism with source $(w, \varepsilon)$, we analyze subwords $s$ of $w$, to determine if $(w, \varepsilon)$ is the source or target of a Gaussian elimination isomorphism by applying Definitions \ref{d.gepair} and \ref{d.gelimiso}. 
A subword $s$ of $w$ that supports a Gaussian elimination isomorphism must start with a barred letter but may end with a barred or unbarred letter. We define $\init(\alpha)$ to be the first letter of the subword $s$ and $\fin(\alpha)$ to be the last letter of the subword $s$. 

The first subword $s$ of $w$ is chosen as follows: We start at the leftmost barred letter in $w$; call that the `start' of $s$. The `stop' of $s$ is a letter to the right of the `start' of $s$. Initially, the `stop' of $s$ is the letter in $w$ immediately to the right of the `start' of $s$.

Fix the `start' location, and consider the subword $s$ with the first letter at the `start' location and the last letter at the `stop' location. 

\underline{Step 1:} If the subword $s$ tells us that $(w, \varepsilon)$ is the target of a G1 or G2 isomorphism in $\mathcal{G}_{i-1}$, which we can check by unbarring barred digits in $s$ from left to right, looking at the pre-images, and applying Definitions \ref{d.gepair} and \ref{d.gelimiso}, then we do not choose a distinguished isomorphism with source $(w, \varepsilon)$ and we move to the next generator $(w, \epsilon)$ of $\CKh_{i, *}(\beta)$ and start from Step 1. If there are no more generators remaining in $\CKh_{i, *}(\beta)$, then we move to a generator in $\CKh_{i+1, *}(\beta)$. Since $M(i) < \infty$, we terminate when there are no more generators in the complex to consider.   Otherwise, we continue with Step 2.

\underline{Step 2:} If $s$ corresponds to the source of a G1 or G2 Gaussian elimination isomorphism, which we can check by barring unbarred digits of $s$ from left to right and examining the images using Definitions \ref{d.gepair} and \ref{d.gelimiso}, then we declare that this G1 (or G2) isomorphism $g_1$ is a distinguished isomorphism with source $(w, \varepsilon)$ and add $g_1$ (or $g_2$) to $\mathcal{G}_i$ (by `adding' we also record the information of the source and target enhanced states of $g_1$ (or $g_2$)). We move to the next generator $(w, \epsilon)$ of $\CKh_{i, *}(\beta)$ and start from Step 1. If there are no more generators remaining in $\CKh_{i, *}(\beta)$, then we move to a generator in $\CKh_{i+1, *}(\beta)$. Since $M(i) < \infty$, we terminate when there are no more generators in the complex to consider. Otherwise, we continue with Step 3.

\begin{remark} 
We remark that it is not possible for identical enhanced Kauffman states corresponding to a subword $s$ to support both a G1 \textit{and} G2 isomorphism (see Lemma \ref{lem:unique-distinguished-isom}) so this procedure is well-defined. 
\end{remark}

\underline{Step 3:} Check if the `stop' of the subword $s$ is the last letter of $w$. If it is, then move the `start' one digit to the right, and move `stop' immediately to the right of the `start'. Go to Step 1 and repeat the rest of the steps with the new subword $s$ whose initial digit is `start,' and whose final digit is `stop.' 

If the `stop' of the subword $s$ is not the last letter of $w$, then move `stop' one digit to the right to obtain a new subword $s$ and repeat Steps 1 and 2 with $s$. 

The procedure terminates since the length of the word is finite.

The method described above for choosing distinguished isomorphisms implies that the following lemma holds.

\begin{lemma}
\label{lem:unique-distinguished-isom}
   The distinguished isomorphisms are chosen so that for each enhanced Kauffman state $(w, \varepsilon)$, there is at most one associated distinguished isomorphism $g(w)$, such that its source is $s_g = (w, \varepsilon)$ or its target $t_g = (w, \varepsilon)$. 
\end{lemma}
\begin{proof} By design, Step 1 in
Algorithm \ref{ss.algorithm} will not choose a distinguished Gaussian elimination isomorphism if the word $w$ contains a target of a distinguished Gaussian elimination isomorphism. If the word $w$ does not contain any subword that is the target of a Gaussian elimination isomorphism, we consider the ambiguity that may result through each subsequent step of the algorithm. In Step 2 of the algorithm, the subword $s$ is being compared to the source of possible GE pairs. The Kauffman state $\bar{i}_jxi_{j'}$ supporting a G1 distinguished Gaussian elimination isomorphism is distinct from the Kauffman state $\bar{i}_jx\bar{i}_{j'}$ supporting a G2 distinguished Gaussian elimination isomorphism, so the assignment cannot coincide. In addition, for $s$ to be the source of a G2 Gaussian elimination isomorphism, the closed component corresponding to the adjacent barred $i$'s in the Kauffman state is marked with a $-$, which gives a different enhanced Kauffman state compared to one where the closed component is marked with a $+$.
\end{proof}

\subsection{Whittling the complex via Gaussian elimination isomorphisms}
\label{ss.whittling_graph}

\begin{definition} \label{d.isocollection}
 Denote by $\mathcal{G}$ the collection of distinguished isomorphisms chosen on the enhanced Kauffman states of $\CKh(ft_n^k)$, according to the procedure described in Section \ref{ss.Distinguishediso}. 
\end{definition} 

We define a directed graph $G$ whose vertices are elements of $\G$ and whose directed edges come from components of differentials of the Khovanov chain complex $\CKh(ft_n^k)$. 

\begin{definition} \label{d.graphisocollection}
Let the vertex set of the graph $G$ be $V(G) = \G$. Let $\alpha: (s, \epsilon) \rightarrow (t, \delta)$ and $\alpha': (s', \epsilon') \rightarrow (t', \delta')$ be distinguished isomorphisms in $\G$ with corresponding vertices $v$ and $v'$, respectively. There is a directed edge $E$ in $G$ from $v$ to $v'$ if there is a map 
\[ \varphi: (s, \epsilon) \rightarrow (t', \delta') \]  such that it satisfies the following two conditions: 
\begin{itemize}
\item $\varphi$ is a single component of a differential in $\CKh(ft^k)$. Here a single component is the result of composing a map that bars a letter (representing the change in the Kauffman state from the $0$-resolution to the $1$-resolution) with a component of the delooping map, and
\item  $\varphi \not= 0$. 
\end{itemize} 

We call $\varphi$ a \textit{connecting map} between $\alpha$ and $\alpha'$. 
\end{definition}

The purpose of this graph is to understand the maps in the new complex obtained by applying Gaussian elimination along a distinguished isomorphism $g\in \G$.

\begin{lemma} \label{l.gacyclic} 
For a braid $\beta$, let $\mathcal{F}$ be a collection of isomorphisms in $\CKh(\beta)$. Suppose $f\in \mathcal{F}$ and denote the new chain complex $\CKh_f(\beta)$ obtained by applying Gaussian elimination (see Lemma \ref{l.geliminate}) along $f$ in $\CKh(\beta)$. Suppose the directed graph $G$ defined by Definition \ref{d.graphisocollection} has no 2-cycles and let $f'\not= f$ be another element in $\mathcal{F}$. Then the component of the differential $f'_f$ corresponding to $f'$ in the Gaussian reduced complex $\CKh_f(\beta)$ is equal to $f'$. 
\end{lemma}

\begin{proof} By Lemma \ref{l.geliminate}, the Gaussian reduced complex $\CKh_f(\beta)$ is the original complex without the terms $b_1$ and $b_2$ that are the source and target of the isomorphism $f$. The new component of the differential in $\CKh_f(\beta)$ with source $c_1$ and $c_2$ is given by $\epsilon  - \gamma \phi^{-1} \delta$ (see Figure \ref{f.2-cycle}).  

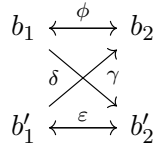
\begin{figure}[H]
\begin{center}
\begin{tikzcd}
   b_1 
        \arrow{r}{\phi} \arrow{dr}[swap, near start]{\delta}
        & b_2 \arrow{l} \\
   b'_1 
        \arrow{r}{\epsilon}  \arrow{ur}[swap, near end]{\gamma}
        & b'_2 \arrow{l}
\end{tikzcd} 
\end{center}
\caption{\label{f.2-cycle} A 2-cycle involving $\phi, \epsilon$.} 
\end{figure} 

Here $\epsilon = f'$, and if $\gamma$ and $\delta$ are both nonzero, then they define a 2-cycle in the graph $G$. Since our graph $G$ has no 2-cycles, one of $\gamma$ and $\delta$ is zero as $\phi^{-1} = f^{-1}$ is an isomorphism. If one of $\gamma$ and $\delta$ is  zero, the components of the differentials in the new complex remains unchanged. 
\end{proof} 

In the context of our paper, with $\G$ denoting the collection of distinguished Gaussian elimination isomorphisms, we aim to iterate this construction over the elements of $\alpha \in \G$. Given $\alpha \in \G$, let $\CKh_\alpha(\beta)$ denote the new chain complex obtained by applying Gaussian elimination to $\CKh(\beta)$ along $\alpha$. We would like to iterate over the remaining Gaussian elimination isomorphisms in $\G \setminus \{\alpha\}$. 

It turns out that if the graph $G$ is acyclic, then there is a topological ordering on its vertices, which induces an ordering on $\mathcal{G}$ with respect to which Gaussian elimination can be applied sequentially along the elements of $\mathcal{G}$, following that order.

\begin{definition}
A directed graph is \textit{acyclic} if it has no directed cycles.  
\end{definition} 

\begin{definition}
A topological ordering on a directed graph is a linear ordering $<$ of vertices such that $u\to v$ implies $u < v$. 
\end{definition}

It is well known that the existence of a topological order on a directed graph is equivalent to it being acyclic. 
\begin{theorem}{\cite{topological-sort}} \label{t.acyclic-ordering}
A directed graph is acyclic if and only if there is a topological ordering on its vertices. 
\end{theorem}

Now suppose that $G$ is acyclic, and therefore, by Theorem \ref{t.acyclic-ordering}, the set of vertices $V(G)$ admits a linear ordering $<$ . Number the vertices $\alpha_1, \alpha_2, \ldots,$ with respect to this linear ordering, so that $\alpha_{i} < \alpha_{i+1}$ for $1\leq i \leq |V(G)|$. We show that we can iterate Gaussian elimination on the elements of $\G$.

\begin{lemma} \label{l.whencycle}
Let $\CKh$ be a chain complex. Let $\G$ be a set of isomorphisms in $\CKh$ and let $G$ be the corresponding graph associated to $\G$, as constructed in Definition \ref{d.graphisocollection}. Let $\alpha\in \G$ and denote by $\CKh_\alpha$ the new chain complex obtained from $\CKh$ by applying Gaussian elimination along $\alpha$. Consider the new graph $G_\alpha$ constructed for $\CKh_\alpha$. If $G_\alpha$ has a cycle, then $G$ has a cycle. 
\end{lemma} 
\begin{proof}
From $G$ to $G_\alpha$, we delete a vertex $v$ corresponding to $\alpha$, then we add a new edge between two vertices $v', v'' \in V(G_\alpha) = V(G) \setminus v$, if the following maps exist as in Figure \ref{f.edgeinGa}. The edge $(v', v'')$ in $G_\alpha$ corresponds to the composition $ \varphi\circ\alpha^{-1} \circ \varphi'$. 
\begin{figure}[H]
\begin{center}
\begin{tikzcd}
   (s', \epsilon') 
        \arrow{r}{\alpha'} \arrow{dr}[swap, near start]{\varphi'}
        & (t', \delta') \\
   (s, \epsilon) 
        \arrow{r}{\alpha} \arrow{dr}[swap, near start]{\varphi}
        & (t, \delta) \\
   (s'', \epsilon'')  
        \arrow{r}{\alpha''}  
        & (t'', \delta'')
\end{tikzcd} 
\end{center}
\caption{\label{f.edgeinGa} A new edge $(v', v'')$ in $G_\alpha$ from the edges in $G$.} 
\end{figure} 
Consider a cycle $C: v_1, \ldots, v_k = v_1$ in $G_\alpha$. If the edges in $G_\alpha$ are the same as in $G$, meaning the component of the differentials labeling each edge is the same, then the sequence of vertices $v_1, \ldots, v_k$ is a cycle in $G$. See Figure \ref{f.cycle} for a generic illustration of the cycle in $G_\alpha$. 
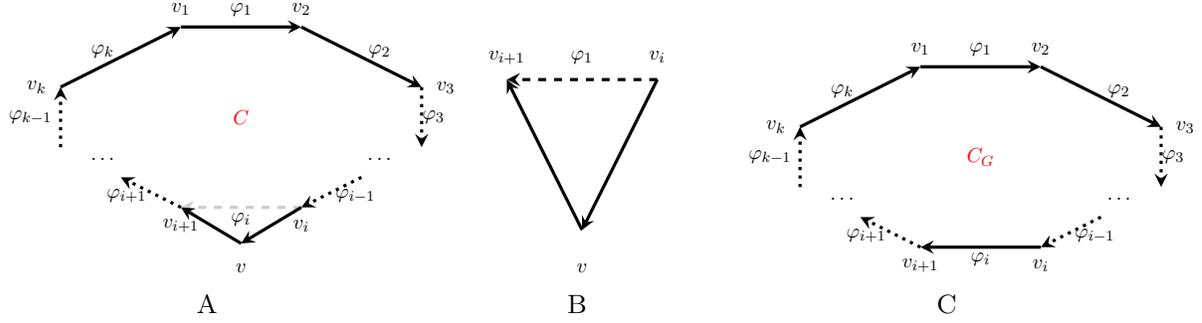
\begin{figure}[H]
 \centering
 \begin{subfigure}{0.33\textwidth}
\centering 
\begin{tikzpicture}[every node/.style={scale=.7}, scale=.8]]

\node at (-1,4.3){$v_1$};
\node at (1,4.3){$v_2$};
\node at (3.4,3){$v_3$};
\node at (0,4.3){$\varphi_1$};
\node at (2.3,3.6){$\varphi_2$};
\node at (-2.3,3.6){$\varphi_k$};
 \draw[very thick, -stealth] (-1,4)--(1,4);
\draw[very thick, -stealth] (1, 4) -- (3, 3);
\node at (-3.4,3){$v_k$};
\draw[very thick, -stealth](-3, 3)--(-1, 4);
\draw[very thick,dotted, -stealth](-3, 2)--(-3, 3);
\node at (-1.9, 1.2){$\varphi_{i+1}$};
\node at (0, 0.8){$\varphi_{i}$};
\node at (1.9, 1.2){$\varphi_{i-1}$};
\draw[very thick,dotted, -stealth](3, 3)--(3, 2);
\node at (3.2, 2.5){$\varphi_{3}$};
\node at (-3.5, 2.5){$\varphi_{k-1}$};

\draw[very thick, -stealth,dashed, gray!40](1,1)--(-1,1);

\node at (-1,0.7){$v_{i+1}$};
\node at (1,0.7){$v_{i}$};
\node at (0,0){$v$};
\draw[very thick, -stealth](1,1)--(0,0.4);
\draw[very thick, -stealth](0,0.4)--(-1,1);
\draw[very thick,dotted, -stealth](2, 1.5)--(1,1);
\draw[very thick,dotted, -stealth](-1, 1)--(-2, 1.5);
\node at (2.3, 1.8){$\ldots$};
\node at (-2.3, 1.8){$\ldots$};
\node[red] at (0,2.5){$C$};
\end{tikzpicture}\caption{}\label{f.cycle}
 \end{subfigure}
 \begin{subfigure}{0.25\textwidth}
 \centering 
\begin{tikzpicture}[every node/.style={scale=.7}, scale=1]]

\node at (-1,4.3){$v_{i+1}$};
\node at (1,4.3){$v_i$};
\node at (0,1.5){$v$};
\node at (0,4.3){$\varphi_1$};
 \draw[very thick, -stealth, dashed] (1,4)--(-1,4);
\draw[very thick, -stealth] (1, 4) -- (0, 2);
\draw[very thick, -stealth] (0, 2)--(-1,4);
\end{tikzpicture}
\caption{}\label{f.seqvertices}
 \end{subfigure}
  \begin{subfigure}{0.33\textwidth}
\centering 
\begin{tikzpicture}[every node/.style={scale=.7}, scale=.8]]

\node at (-1,4.3){$v_1$};
\node at (1,4.3){$v_2$};
\node at (3.4,3){$v_3$};
\node at (0,4.3){$\varphi_1$};
\node at (2.3,3.6){$\varphi_2$};
\node at (-2.3,3.6){$\varphi_k$};
 \draw[very thick, -stealth] (-1,4)--(1,4);
\draw[very thick, -stealth] (1, 4) -- (3, 3);
\node at (-3.4,3){$v_k$};
\draw[very thick, -stealth](-3, 3)--(-1, 4);
\draw[very thick,dotted, -stealth](-3, 2)--(-3, 3);
\node at (-1.9, 1.2){$\varphi_{i+1}$};
\node at (0, 0.8){$\varphi_{i}$};
\node at (1.9, 1.2){$\varphi_{i-1}$};
\draw[very thick,dotted, -stealth](3, 3)--(3, 2);
\node at (3.2, 2.5){$\varphi_{3}$};
\node at (-3.5, 2.5){$\varphi_{k-1}$};

\node at (-1,0.7){$v_{i+1}$};
\node at (1,0.7){$v_{i}$};
\draw[very thick, -stealth](1,1)--(-1,1);
\draw[very thick,dotted, -stealth](2, 1.5)--(1,1);
\draw[very thick,dotted, -stealth](-1, 1)--(-2, 1.5);
\node at (2.3, 1.8){$\ldots$};
\node at (-2.3, 1.8){$\ldots$};
\node[red] at (0,2.5){$C_{G}$};
\end{tikzpicture}\caption{}\label{f.oancycle}
 \end{subfigure}
\caption{ (A): A cycle $C$ in $G_\alpha$ ,  (B): a sequence of vertices $v_i, v, v_{i+1}$ in $G$ , and (C): the cycle $C_G$ in $G$ corresponding to $C$ in $G_\alpha$. }
\end{figure}
Otherwise, suppose there is an edge between $v_i$ and  $v_{i+1}$ whose corresponding map $\varphi$ in $G$ is zero, but $\varphi_\alpha \in G_\alpha$ is nonzero. Then it must have come from applying a Gaussian elimination isomorphism along $\alpha$, as in Figure \ref{f.edgeinGa}. This implies that there is a vertex $v$ in $G$ and two edges $(v_i, v)$ and $(v, v_{i+1})$. See Figure \ref{f.seqvertices} for a local picture involving just $v_i, v_{i+1}, v$ and Figure \ref{f.oancycle} for an illustration involving a cycle in $G$. Replacing such edge $(v_i, v_{i+1})$ in $C$ with $(v_i, v)$ and $(v, v_{i+1})$, recovers a cycle $C_G$ in the graph $G$. 
\end{proof} 

\begin{proposition}

Let $1\leq k \leq |V(G)|$ and 
$\CKh^k(\beta)$ be defined as the chain complex obtained from $CKh(\beta)$ by applying Gaussian eliminations along the first $k$ isomorphisms in $\G$ with respect to the linear order $<$. That is,
\[ \CKh^k(\beta):= (((\CKh_{\alpha_1})_{\alpha_2})\cdots) _{\alpha_k}(\beta).  \]  
Suppose $G$ is acyclic and let $f\in \G \setminus \{\alpha_1, \alpha_2, \ldots, \alpha_k\}.$ Then the component of the differential $f^{k}$ in $\CKh^k(\beta)$ with the same source and target as $f$ is equal to $f$. More precisely, 
\[ f^{k} = f. \] 
\end{proposition} 

\begin{proof}
The proof is by induction on $k$. The base case $k =1$ follows from Lemma \ref{l.whencycle}. Consider $k + 1$. By Lemma \ref{l.gacyclic}, if $G_{k-1}$ has no 2-cycles, then $f = f^k$. Then, Lemma \ref{l.whencycle} tells us that if $G$ is acyclic, then $G_{k-1}$ has no 2-cycles. 
\end{proof}

\subsection{Existence of a linear order on the set $\mathcal{G}$} \label{ss.linear_order}
In this section we fix $n$ and specialize $\beta = \ftbraid^k_n$ for the rest of the paper. 

From the previous section, it suffices to show that $G$ is acyclic in order to apply Gaussian elimination along all the distinguished isomorphisms in $\mathcal{G}$. This is the goal of the present section. We begin by understanding a \textit{connecting map} between two vertices in $G$ corresponding to Gaussian eliminations.

Recall that for an enhanced Kauffman state represented by $(w, \epsilon)$, the word $w$ is a barred braid word that encodes the choice of a Kauffman state $\sigma$ of the 1-resolution on a crossing $\sigma_i$ by barring the digit $i$, and the 0-resolution by an unbarred digit. We write $w= w'$ for two Kauffman states on the same braid when all the barrings and their locations are the same. 

\begin{definition}[Active crossing]
Let $\varphi: (s, \epsilon) \rightarrow (t, \delta)$ be a component of the differential $d$ in the chain complex $\CKh(\beta)$. The \textit{active crossing} $\AC(\varphi)$ of $\varphi$ is the crossing on which $d$ changes the $0$-resolution on the corresponding digit in $s$ to the $1$-resolution on the corresponding digit in $t$ via a saddle.
\end{definition} 

We keep track of the locations of subwords in $s$ and $t$ through the following definitions/notations: 

\begin{itemize}
\item \textbf{Support of $\alpha\in \G$}: If $(s, \epsilon)$ be the source of a distinguished isomorphism $\alpha \in \G$ chosen by Algorithm \ref{ss.algorithm}, then it contains a subword $w = \bar{i}xi$ or $\bar{i}x\bar{i}$ on which $\alpha$ is supported, depending on whether $\G$ is of type G1 or G2. That is, $\AC(\alpha) = i$ or $\AC(\alpha) = i+1$, respectively. The shortest such subword $w = \overline{i}xi$ is called the \textit{support of $\alpha$}.
\item $\mathbf{\init(w)}$: For $w = \bar{i}xi$ or $w = \bar{i}x\bar{i}$, let $\init(w)$ be the position of the first $\bar{i}$ in $s$ and $\fin(w)$ be the position of the second (barred or unbarred) $i$ in $s$ immediately to the right of $x$. For $\alpha\in \G$, we define $\init(\alpha)$ to be $\init(w)$, where $w$ is the support of $\alpha$ and $\fin(\alpha)$  is $\fin(w)$.
\item $\mathbf{a<b}$: Let $a, b$ be two digits in a barred braid word $w$, we write $a < b$ to mean that the position of $a$ in the braid word is to the left of the position of $b$.
\item $\mathbf{w[m, n]}$: Let $w$ be a barred braid word representing a Kauffman state. The subword $w[m, n]$ is the subword from the $m$th position to the $n$th position, including the $m$th digit and the $n$th digit, reading from left to right. We will write $w[m]$ to mean $w[m, m]$.
\item $\mathbf{s\subseteq w}$: We use $s\subseteq w$ to denote a subword of a barred braid word, and $s\subset w$ to denote a proper subword. 
\end{itemize}

Our first result of this section characterizes a connecting map $\varphi:\alpha\rightarrow \alpha'$ in relation to the supports of $\alpha$ and $\alpha'\in \G$. 

\begin{definition}[Obstructing a source and making a target] Consider a connecting map $\varphi: (s, \epsilon) \rightarrow (t', \delta')$ as shown in Figure \ref{f.generic_connecting}.
\begin{figure}[H]
\begin{center}
\begin{tikzcd}
    (s, \varepsilon) 
        \arrow{r}{\alpha} \arrow{dr}{\varphi}
        & (t, \delta) \\
    (s', \varepsilon') 
        \arrow{r}{\alpha'} 
        & (t', \delta') 
\end{tikzcd} 
\end{center}
\caption{Generic connecting map $\varphi$.}
\label{f.generic_connecting}
\end{figure} 
\begin{itemize} 
\item Let $w$ be the support of the Gaussian elimination isomorphism $\alpha$. We say that \textit{$\varphi$ obstructs $\alpha$} if $\AC(\varphi) \in w$ and $t'[\init(w), \fin(w)]$ is not the source of $\alpha$ 

\item  Let $w'$ be the support of the Gaussian elimination isomorphism $\alpha'$. We say that \textit{$\varphi$ makes $\alpha'$} if $\AC(\varphi) \in t'[\init(w'), \fin(w')]$, and $s[\init(w'), \fin(w')]$ is not the target of $\alpha'$. 
\end{itemize} 
\end{definition} 

\begin{lemma} \label{l.connecting} 
Suppose we are given a connecting map $\varphi:(s, \epsilon) \mapsto (t', \delta')$ as shown in Figure \ref{f.generic_connecting}.
Then $\AC(\varphi) \not= \AC(\alpha')$ and at least one of the following holds: 
\begin{enumerate}
\item $\varphi$ obstructs $\alpha$. In this case $\init(\alpha) < \init(\alpha')$. 
\item $\varphi$ makes $\alpha'$. In this case $\init(\alpha') < \init(\alpha)$.
\end{enumerate} 

\begin{proof} 
First, we have $\AC(\varphi) \not= \AC(\alpha')$; otherwise it would imply $(s, \epsilon) = (s', \epsilon')$, contradicting the well-definedness of Algorithm \ref{ss.algorithm}. 

Suppose that $\varphi$ obstructs $\alpha$. Then, we need to show $\init(\alpha) < \init(\alpha')$. Suppose on the contrary that $\init(\alpha') < \init(\alpha)$. Consider $w = s[\init(\alpha), \fin(\alpha)]$ and $w' = s[\init(\alpha'), \fin(\alpha')]$.  The subword $w'$ is the target of a distinguished Gaussian elimination isomorphism $\alpha'$. This means that $w' = \bar{i}x\bar{i}$ or $w' = \bar{i}\overline{i+1}x\bar{i}$. In the first case, the digits $i+1, i \in x$ are unbarred, and in the second case, the digit $i-1 \in x$ is unbarred. Since $\init(\alpha') < \init(\alpha)$, Algorithm \ref{ss.algorithm} would not choose $\alpha$ as the distinguished Gaussian elimination isomorphism, leading to a contradiction.

Next, assume $\varphi$ does not obstruct $\alpha$. We need to prove that $\init(\alpha') < \init(\alpha)$ and that $\varphi$ makes $\alpha'$. By definition, one possibility is that $\AC(\varphi) \notin w$. Here $w$ is the support of $\alpha$. This means that $t'[\init(w), \fin(w)]$ supports a Gaussian elimination isomorphism, but it is somehow not chosen to be in $\G$. Let $w' \subseteq t'$ be the support of $\alpha'$. The position of $\init(w')$ in the braid $\beta$ must thus be to the left of $\init(w)$ in $\beta$. That is, $\init(\alpha') < \init(\alpha)$. Otherwise, Algorithm \ref{ss.algorithm} would not choose $\alpha'$. 

Let $w'(s) = s[\init(w'), \fin(w')]$. If $w'(s) = t'[\init(w'), \fin(w')]$, then $w'(s)$ is the target in $s$ of the distinguished Gaussian elimination isomorphism $\alpha'$. Since $\init(w') < \init(w)$, we have $\init(w'(s)) < \init(w)$. This is a contradiction to the assumption that $\alpha$ is a distinguished Gaussian elimination isomorphism.  Thus, $w'(s) \not= t'[\init(w'), \fin(w')]$ and we can conclude that $\AC(\varphi) \in t'[\init(w'), \fin(w')]$. Hence $\varphi$ makes $\alpha'$. 
\end{proof} 
\end{lemma} 

By Lemma \ref{l.connecting}, we have the following possibilities for the local shapes of the Kauffman states for a connecting map $\varphi$, depending on the type, G1 or G2, of the distinguished Gaussian elimination isomorphisms $\alpha, \alpha'$ that $\varphi$ obstructs or makes. 
\subsubsection{Possible shapes of a connecting map.}
With the same notation as stated in Lemma \ref{l.connecting}, we illustrate the local pictures of $\varphi$ corresponding to the cases of the lemma. In Figures \ref{f.config1ab} - \ref{f.config4}, the red pair of saddles corresponds to $\varphi$'s barring of a crossing in $s$. 
\begin{itemize}
\item \textbf{The connecting map $\varphi$ obstructs $\alpha$, where $\alpha$ is a G1 Gaussian elimination isomorphism (Figure \ref{f.config1ab})}. Suppose $w\subseteq s$ is the subword that supports $\alpha$. There are two ways $\varphi$ can modify $w$ so that it is no longer the support, or source, of a G1 Gaussian elimination isomorphism. The two possibilities (A) and (B) are shown in Figure \ref{f.config1ab}, where we show the subword that $\varphi$ modifies for each scenario. 

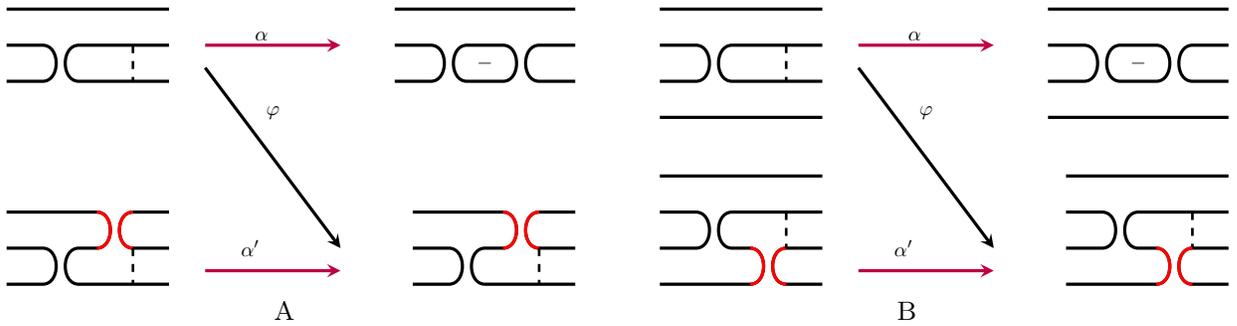
\begin{figure}[H]
 \begin{subfigure}{.45\textwidth}
 \centering
\begin{tikzpicture}[every node/.style={scale=.7}, scale=.6]

\begin{scope}[shift={(-6,0)}, rotate=-90, scale=.8] 

\draw[very thick] (-1, -.5) to (-1, 4);
\draw[very thick] (1,4) to (1,1.5) .. controls ++(0,-.5) and ++(0,-.5) .. (0,1.5) to (0,4);
\draw[line width = 1pt, dashed] (0,3)--(1,3);
\draw[very thick] (0,-0.5)  to (0,0.5) .. controls ++(0,.5) and ++(0,.5) .. (1,0.5)--(1,-0.5);
 \end{scope}

\node at (-0.75, 0.2){$\alpha$};
 \draw[very thick, purple, -stealth] (-2,0)--(1,0);
\draw[very thick, -stealth] (-2, -.5) -- (1, -4.5);
\node at (-0.5, -1.5){$\varphi$};
 \draw[very thick, purple, -stealth] (-2,-5)--(1, -5);
\node at (-1, -4.5){$\alpha'$};

\begin{scope}[shift={(5,0)},rotate=-90, scale=.8]
\draw[very thick] (-1, -3.5) to (-1, 1.5);

\draw[very thick] (1,1.5) to (1,0.5) .. controls ++(0,-.5) and ++(0,-.5) .. (0,0.5) to (0,1.5);
\draw[very thick] (1,-1.5) to (1,-0.5) .. controls ++(0,.5) and ++(0,.5) .. (0,-0.5)--(0,-1.5) .. controls ++(0,-0.5) and ++(0,-.5) .. (1,-1.5);
\node at (0.5, -1){$-$}; 
\draw[very thick] (0,-3.5)  to (0,-2.5) .. controls ++(0,.5) and ++(0,.5) .. (1,-2.5)--(1,-3.5);
\end{scope}
  
\begin{scope}[shift={(-6,-4.5)}, rotate=-90, scale=.8]  

\draw[very thick] (-1,4) to (-1,3) .. controls ++(0,-.5) and ++(0,-.5) .. (0,3) to (0,4);
\draw[very thick] (-1,-0.5)  to (-1,2) .. controls ++(0,.5) and ++(0,.5) .. (0,2)--(0,1.5) .. controls ++(0,-0.5) and ++(0,-.5) .. (1,1.5)--(1,4);

\draw[very thick, red] (-1,3) .. controls ++(0,-.5) and ++(0,-.5) .. (0,3);
\draw[very thick, red] (-1,2) .. controls ++(0,.5) and ++(0,.5) .. (0,2);

\draw[very thick] (0,-0.5)  to (0,0.5) .. controls ++(0,.5) and ++(0,.5) .. (1,0.5)--(1,-0.5);

\draw[line width = 1pt, dashed] (0,3)--(1,3);

 \end{scope}

\begin{scope}[shift={(3,-4.5)}, rotate=-90, scale=.8]  

\draw[very thick] (-1,4) to (-1,3) .. controls ++(0,-.5) and ++(0,-.5) .. (0,3) to (0,4);

\draw[very thick] (-1,-0.5)  to (-1,2) .. controls ++(0,.5) and ++(0,.5) .. (0,2)--(0,1.5) .. controls ++(0,-0.5) and ++(0,-.5) .. (1,1.5)--(1,4);

\draw[very thick] (0,-0.5)  to (0,0.5) .. controls ++(0,.5) and ++(0,.5) .. (1,0.5)--(1,-0.5);
\draw[very thick, red] (-1,3) .. controls ++(0,-.5) and ++(0,-.5) .. (0,3);
\draw[very thick, red] (-1,2) .. controls ++(0,.5) and ++(0,.5) .. (0,2);

\draw[line width = 1pt, dashed] (0,3)--(1,3);

 \end{scope}
   \end{tikzpicture}  \subcaption{}
   \end{subfigure} \hspace{1cm}
  \begin{subfigure}{.4\textwidth}
\begin{tikzpicture}[every node/.style={scale=.7}, scale=.6]

\begin{scope}[shift={(-6,0)}, rotate=-90, scale=.8] 
\draw[very thick] (-1, -.5) to (-1, 4);
\draw[very thick] (2, -.5) to (2, 4);
\draw[very thick] (1,4) to (1,1.5) .. controls ++(0,-.5) and ++(0,-.5) .. (0,1.5) to (0,4);
\draw[line width = 1pt, dashed] (0,3)--(1,3);
\draw[very thick] (0,-0.5)  to (0,0.5) .. controls ++(0,.5) and ++(0,.5) .. (1,0.5)--(1,-0.5);
 \end{scope}
\node at (-0.75, 0.2){$\alpha$};
 \draw[very thick, purple, -stealth] (-2,0)--(1,0);
\draw[very thick, -stealth] (-2, -.5) -- (1, -4.5);
\node at (-0.5, -1.5){$\varphi$};
 \draw[very thick, purple, -stealth] (-2,-5)--(1, -5);
\node at (-1, -4.5){$\alpha'$};

\begin{scope}[shift={(5,0)},rotate=-90, scale=.8]
\draw[very thick] (-1, -3.5) to (-1, 1.5);
\draw[very thick] (2, -3.5) to (2, 1.5);

\draw[very thick] (1,1.5) to (1,0.5) .. controls ++(0,-.5) and ++(0,-.5) .. (0,0.5) to (0,1.5);
\draw[very thick] (1,-1.5) to (1,-0.5) .. controls ++(0,.5) and ++(0,.5) .. (0,-0.5)--(0,-1.5) .. controls ++(0,-0.5) and ++(0,-.5) .. (1,-1.5);
\node at (0.5, -1){$-$}; 
\draw[very thick] (0,-3.5)  to (0,-2.5) .. controls ++(0,.5) and ++(0,.5) .. (1,-2.5)--(1,-3.5);
\end{scope}
  
\begin{scope}[shift={(-6,-4.5)}, rotate=-90, scale=.8]  

\draw[very thick] (-2, -.5) to (-2, 4);
\draw[very thick] (0,4) to (0,3) .. controls ++(0,-.5) and ++(0,-.5) .. (1,3) to (1,4);
\draw[very thick] (1,-0.5)  to (1,2) .. controls ++(0,.5) and ++(0,.5) .. (0,2)--(0,1.5) .. controls ++(0,-0.5) and ++(0,-.5) .. (-1,1.5)--(-1,4);
 \draw[very thick, red] (1,2) .. controls ++(0,.5) and ++(0,.5) .. (0,2);
\draw[very thick, red] (0,3) .. controls ++(0,-.5) and ++(0,-.5) .. (1,3);

\draw[very thick] (-1,-0.5)  to (-1,0.5) .. controls ++(0,.5) and ++(0,.5) .. (0,0.5)--(0,-0.5);
 \draw[line width = 1pt, dashed] (-1,3)--(0,3);
\end{scope}

\begin{scope}[shift={(3,-4.5)}, rotate=-90, scale=.8]  
\draw[very thick] (-2, -.5) to (-2, 4);
\draw[very thick] (0,4) to (0,3) .. controls ++(0,-.5) and ++(0,-.5) .. (1,3) to (1,4);
\draw[very thick] (1,-0.5)  to (1,2) .. controls ++(0,.5) and ++(0,.5) .. (0,2)--(0,1.5) .. controls ++(0,-0.5) and ++(0,-.5) .. (-1,1.5)--(-1,4);
\draw[very thick] (-1,-0.5)  to (-1,0.5) .. controls ++(0,.5) and ++(0,.5) .. (0,0.5)--(0,-0.5);
 \draw[line width = 1pt, dashed] (-1,3)--(0,3);
 \draw[very thick, red] (1,2) .. controls ++(0,.5) and ++(0,.5) .. (0,2);
\draw[very thick, red] (0,3) .. controls ++(0,-.5) and ++(0,-.5) .. (1,3);
 \end{scope}
   \end{tikzpicture}
    \subcaption{}
   \end{subfigure}
\caption{(A): The support of a G1 Gaussian elimination isomorphism is modified by $\varphi$ by the red saddle placed on top to no longer be the source of a G1 Gaussian elimination isomorphism in $t'$ . (B): The support of a G1 Gaussian elimination isomorphism is modified by $\varphi$ by the red saddle placed on the bottom to no longer be the source of a G1 Gaussian elimination isomorphism in $t'$.} \label{f.config1ab}
\end{figure}

\item \textbf{The connecting map $\varphi$ obstructs $\alpha$, where $\alpha$ is a G2 Gaussian elimination isomorphism  (Figure \ref{f.connecting_obstruct_2})}.
We illustrate the case where the barring by the connecting map obstructs a G2 Gaussian elimination isomorphism.

\begin{figure}[H] 
\begin{subfigure}{0.4\textwidth}
\centering
\begin{tikzpicture}[every node/.style={scale=.7}, scale=.6]

\begin{scope}[shift={(-4,0)}, rotate=-90, scale=.8] 

\draw[very thick] (-1, -3.5) to (-1, 1.5);
\draw[very thick] (2, -3.5) to (2, 1.5);

\draw[very thick] (1,1.5) to (1,0.5) .. controls ++(0,-.5) and ++(0,-.5) .. (0,0.5) to (0,1.5);
\draw[very thick] (1,-1.5) to (1,-0.5) .. controls ++(0,.5) and ++(0,.5) .. (0,-0.5)--(0,-1.5) .. controls ++(0,-0.5) and ++(0,-.5) .. (1,-1.5);
\node at (0.5, -1){$+$}; 
\draw[very thick] (0,-3.5)  to (0,-2.5) .. controls ++(0,.5) and ++(0,.5) .. (1,-2.5)--(1,-3.5);
\end{scope}

\node at (-0.75, 0.2){$\alpha$};
 \draw[very thick, purple, -stealth] (-2,0)--(1,0);
\draw[very thick, -stealth] (-2, -.5) -- (1, -4.5);
\node at (-0.5, -1.5){$\varphi$};
 \draw[very thick, purple, -stealth] (-2,-5)--(1, -5);
\node at (-1, -4.5){$\alpha'$};

\begin{scope}[shift={(2,0)},rotate=-90, scale=.8]
\draw[very thick] (2, -.5) to (2, 4.5);
\draw[very thick] (0,4.5) to (0,4) .. controls ++(0,-.5) and ++(0,-.5) .. (1,4) to (1,4.5);
\draw[very thick] 
(-1,-0.5)  to (-1,1.25) .. controls ++(0,.5) and ++(0,.5) .. 
(0,1.25) to (0,1)   .. controls ++(0,-0.5) and ++(0,-.5) .. 
(1,1) to (1,2.75) .. controls ++(0,.5) and ++(0,.5) .. 
(0,2.75) to (0,2.5) .. controls ++(0,-0.5) and ++(0,-.5) .. (-1,2.5) to (-1, 4.5);
\draw[very thick] (0,-0.5)  to (0,0) .. controls ++(0,.5) and ++(0,.5) .. (1,0)--(1,-0.5);

\end{scope}
  
\begin{scope}[shift={(-6,-4)}, rotate=-90, scale=.8]  

\draw[very thick] (-1, -.5) to (-1, 4.5);
\draw[very thick] (0,4.5) to (0,4) .. controls ++(0,-.5) and ++(0,-.5) .. (1,4) to (1,4.5);
\draw[very thick] 
(2,-0.5)  to (2,1.25) .. controls ++(0,.5) and ++(0,.5) .. 
(1,1.25) to (1,1)   .. controls ++(0,-0.5) and ++(0,-.5) .. 
(0,1) to (0,2.75) .. controls ++(0,.5) and ++(0,.5) .. 
(1,2.75) to (1,2.5) .. controls ++(0,-0.5) and ++(0,-.5) .. (2,2.5) to (2, 4.5);
\draw[very thick, red] (2,1.25) .. controls ++(0,.5) and ++(0,.5) .. 
(1,1.25);
\draw[very thick, red] (1,2.5) .. controls ++(0,-0.5) and ++(0,-.5) .. (2,2.5);
\draw[very thick] (0,-0.5)  to (0,0) .. controls ++(0,.5) and ++(0,.5) .. (1,0)--(1,-0.5);

\end{scope}

\begin{scope}[shift={(2,-4)}, rotate=-90, scale=.8]  
\draw[very thick] (-1, -.5) to (-1, 4.5);
\draw[very thick] (0,4.5) to (0,4) .. controls ++(0,-.5) and ++(0,-.5) .. (1,4) to (1,4.5);
\draw[very thick] 
(2,-0.5)  to (2,1.25) .. controls ++(0,.5) and ++(0,.5) .. 
(1,1.25) to (1,1)   .. controls ++(0,-0.5) and ++(0,-.5) .. 
(0,1) to (0,2.75) .. controls ++(0,.5) and ++(0,.5) .. 
(1,2.75) to (1,2.5) .. controls ++(0,-0.5) and ++(0,-.5) .. (2,2.5) to (2, 4.5);
\draw[very thick, red] (2,1.25) .. controls ++(0,.5) and ++(0,.5) .. 
(1,1.25);
\draw[very thick, red] (1,2.5) .. controls ++(0,-0.5) and ++(0,-.5) .. (2,2.5);
\draw[very thick] (0,-0.5)  to (0,0) .. controls ++(0,.5) and ++(0,.5) .. (1,0)--(1,-0.5);
 \end{scope}
 
   \end{tikzpicture}
\caption{\label{f.connecting_obstruct_2} }
   \end{subfigure}\hspace{1cm}
   \begin{subfigure}{0.4\textwidth}
   \centering 
\begin{tikzpicture}[every node/.style={scale=.7}, scale=.6]

\begin{scope}[shift={(-6,0)}, rotate=-90, scale=.8] 

\draw[very thick] (-1, -.5) to (-1, 4);
\draw[very thick] (1,4) to (1,3) .. controls ++(0,-.5) and ++(0,-.5) .. (0,3) to (0,4);
\draw[very thick] (0,-0.5)  to (0,2) .. controls ++(0,.5) and ++(0,.5) .. (1,2)--(1,-0.5);
 \end{scope}

\node at (-0.75, 0.2){$\alpha$};
 \draw[very thick, purple, -stealth] (-2,0)--(1,0);
\draw[very thick, -stealth] (-2, -.5) -- (1, -4.5);
\node at (-0.5, -1.5){$\varphi$};
 \draw[very thick, purple, -stealth] (-2,-5)--(1, -5);
\node at (-1, -4.5){$\alpha'$};

\begin{scope}[shift={(2,0)},rotate=-90, scale=.8]
\draw[very thick] (-1, -.5) to (-1, 4);
\draw[very thick] (1,4) to (1,3) .. controls ++(0,-.5) and ++(0,-.5) .. (0,3) to (0,4);
\draw[very thick] (0,-0.5)  to (0,2) .. controls ++(0,.5) and ++(0,.5) .. (1,2)--(1,-0.5);
 \end{scope}

\begin{scope}[shift={(-6,-4.5)}, rotate=-90, scale=.8]  

\draw[very thick] (-1, -.5) to (-1, 4);
\draw[very thick] (0,4)--(0,1.5) .. controls ++(0,-0.5) and ++(0,-.5) .. (1,1.5)--(1,4);
\draw[very thick] (0,-0.5)  to (0,0.5) .. controls ++(0,.5) and ++(0,.5) .. (1,0.5)--(1,-0.5);
\draw[very thick, red] (0,1.5) .. controls ++(0,-0.5) and ++(0,-.5) .. (1,1.5);
\draw[very thick, red] (0,0.5) .. controls ++(0,.5) and ++(0,.5) .. (1,0.5);

\draw[line width = 1pt, dashed] (0,2.5)--(1,2.5);

 \end{scope}

\begin{scope}[shift={(2,-4.5)}, rotate=-90, scale=.8]  
\draw[very thick] (-1, -.5) to (-1, 4);
\draw[very thick] (0,2) to (0,1.5) .. controls ++(0,-0.5) and ++(0,-.5) .. (1,1.5) to (1,2).. controls ++(0,.5) and ++(0,.5) .. (0,2);

\draw[very thick] (0,4)--(0,3) .. controls ++(0,-0.5) and ++(0,-.5) .. (1,3)--(1,4);
\draw[very thick] (0,-0.5)  to (0,0.5) .. controls ++(0,.5) and ++(0,.5) .. (1,0.5)--(1,-0.5);
\draw[very thick, red] (0,1.5) .. controls ++(0,-0.5) and ++(0,-.5) .. (1,1.5);
\draw[very thick, red] (0,0.5) .. controls ++(0,.5) and ++(0,.5) .. (1,0.5);
\node at (0.5, 1.75){$-$}; 
 \end{scope}
   \end{tikzpicture}  
  \caption{}
   \label{fig:connectingKauff}

   \end{subfigure}
\caption{(A): The connecting map $\varphi$ changes the subword $w\subseteq s$ on which $\alpha$ is supported, so 
that the $t'[\init(w), \fin(w)]$ is no longer the source of a G2 Gaussian elimination isomorphism.  (B): The connecting map $\varphi$ barrs the letter corresponding to the red saddle, which is the target of a Gausian elimination isomorphism in $t'$. }

\end{figure}
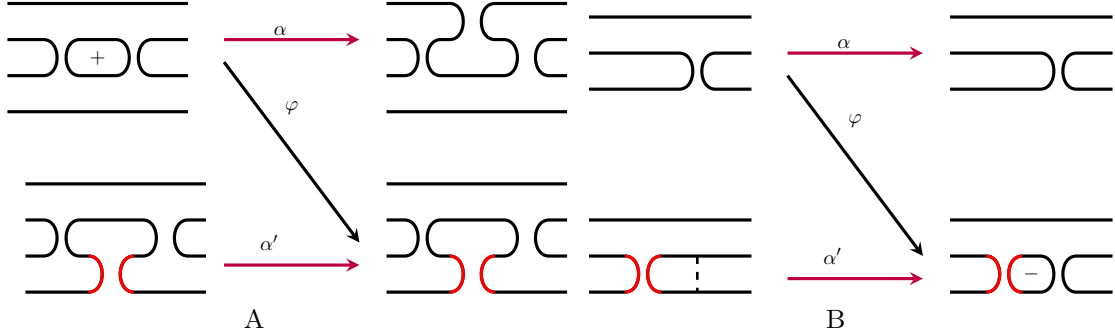 

\item  \textbf{The connecting map $\varphi$ makes $\alpha'$, where $\alpha'$ is a G1 Gaussian elimination isomorphism (Figure \ref{fig:connectingKauff})}.

\item \textbf{The connecting map $\varphi$ makes $\alpha'$, where $\alpha'$ is a G2 Gaussian elimination isomorphism (Figure \ref{f.config4})}.


\begin{figure}[H]
 \begin{subfigure}{.45\textwidth}
\begin{tikzpicture}[every node/.style={scale=.7}, scale=.6]

\begin{scope}[shift={(-6,0)}, rotate=-90, scale=.8] 

\draw[very thick] (2, -.5) to (2, 4.5);
\draw[very thick, dashed] (0, 4.2) to (1, 4.2);
\draw[very thick] (0,4.5) to (0,4) .. controls ++(0,-.5) and ++(0,-.5) .. (1,4) to (1,4.5);
\draw[very thick] 
(-1,-0.5)  to (-1,1.25) .. controls ++(0,.5) and ++(0,.5) .. 
(0,1.25) to (0,1)   .. controls ++(0,-0.5) and ++(0,-.5) .. 
(1,1) to (1,2.75) .. controls ++(0,.5) and ++(0,.5) .. 
(0,2.75) to (0,2.5) .. controls ++(0,-0.5) and ++(0,-.5) .. (-1,2.5) to (-1, 4.5);
\draw[very thick] (0,-0.5)  to (0,0) .. controls ++(0,.5) and ++(0,.5) .. (1,0)--(1,-0.5);

\end{scope}

\node at (-0.75, 0.2){$\alpha$};
 \draw[very thick, purple, -stealth] (-2,0)--(1,0);
\draw[very thick, -stealth] (-2, -.5) -- (1, -4.5);
\node at (-0.5, -1.5){$\varphi$};
 \draw[very thick, purple, -stealth] (-2,-5)--(1, -5);
\node at (-1, -4.5){$\alpha'$};

\begin{scope}[shift={(2,0)},rotate=-90, scale=.8]
\draw[very thick] (2, -.5) to (2, 4.5);
\draw[very thick, dashed] (0, 4.2) to (1, 4.2);
\draw[very thick] (0,4.5) to (0,4) .. controls ++(0,-.5) and ++(0,-.5) .. (1,4) to (1,4.5);
\draw[very thick] 
(-1,-0.5)  to (-1,1.25) .. controls ++(0,.5) and ++(0,.5) .. 
(0,1.25) to (0,1)   .. controls ++(0,-0.5) and ++(0,-.5) .. 
(1,1) to (1,2.75) .. controls ++(0,.5) and ++(0,.5) .. 
(0,2.75) to (0,2.5) .. controls ++(0,-0.5) and ++(0,-.5) .. (-1,2.5) to (-1, 4.5);
\draw[very thick] (0,-0.5)  to (0,0) .. controls ++(0,.5) and ++(0,.5) .. (1,0)--(1,-0.5);

\end{scope}
  
\begin{scope}[shift={(-4,-5)}, rotate=-90, scale=.8]  
\draw[very thick] (-1, -3.5) to (-1, 1.5);
\draw[very thick] (2, -3.5) to (2, 1.5);

\draw[very thick] (1,1.5) to (1,0.5) .. controls ++(0,-.5) and ++(0,-.5) .. (0,0.5) to (0,1.5);
\draw[very thick] (1,-1.5) to (1,-0.5) .. controls ++(0,.5) and ++(0,.5) .. (0,-0.5)--(0,-1.5) .. controls ++(0,-0.5) and ++(0,-.5) .. (1,-1.5);
\node at (0.5, -1){$+$}; 
\draw[very thick] (0,-3.5)  to (0,-2.5) .. controls ++(0,.5) and ++(0,.5) .. (1,-2.5)--(1,-3.5);

 \end{scope}

\begin{scope}[shift={(2,-5)}, rotate=-90, scale=.8]  

\draw[very thick] (2, -.5) to (2, 4.5);
\draw[very thick] (0,4.5) to (0,4) .. controls ++(0,-.5) and ++(0,-.5) .. (1,4) to (1,4.5);
\draw[very thick] 
(-1,-0.5)  to (-1,1.25) .. controls ++(0,.5) and ++(0,.5) .. 
(0,1.25) to (0,1)   .. controls ++(0,-0.5) and ++(0,-.5) .. 
(1,1) to (1,2.75) .. controls ++(0,.5) and ++(0,.5) .. 
(0,2.75) to (0,2.5) .. controls ++(0,-0.5) and ++(0,-.5) .. (-1,2.5) to (-1, 4.5);

\draw[very thick] (0,-0.5)  to (0,0) .. controls ++(0,.5) and ++(0,.5) .. (1,0)--(1,-0.5);
\draw[very thick, red] (0,2.75) .. controls ++(0,.5) and ++(0,.5) .. 
(1,2.75);
\draw[very thick, red] (0,4) .. controls ++(0,-0.5) and ++(0,-.5) .. (1,4);

 \end{scope}
   \end{tikzpicture}  \subcaption{}
   \end{subfigure} \hspace{0.3cm}
  \begin{subfigure}{.45\textwidth}
\begin{tikzpicture}[every node/.style={scale=.7}, scale=.6]

\begin{scope}[shift={(-6,0)}, rotate=-90, scale=.8] 
\draw[very thick] (2, -.5) to (2, 4.5);
\draw[very thick, dashed] (0, -0.2) to (1, -0.2);
\draw[very thick] (0,4.5) to (0,4) .. controls ++(0,-.5) and ++(0,-.5) .. (1,4) to (1,4.5);
\draw[very thick] 
(-1,-0.5)  to (-1,1.25) .. controls ++(0,.5) and ++(0,.5) .. 
(0,1.25) to (0,1)   .. controls ++(0,-0.5) and ++(0,-.5) .. 
(1,1) to (1,2.75) .. controls ++(0,.5) and ++(0,.5) .. 
(0,2.75) to (0,2.5) .. controls ++(0,-0.5) and ++(0,-.5) .. (-1,2.5) to (-1, 4.5);
\draw[very thick] (0,-0.5)  to (0,0) .. controls ++(0,.5) and ++(0,.5) .. (1,0)--(1,-0.5);

\end{scope}
\node at (-0.75, 0.2){$\alpha$};
 \draw[very thick, purple, -stealth] (-2,0)--(1,0);
\draw[very thick, -stealth] (-2, -.5) -- (1, -4.5);
\node at (-0.5, -1.5){$\varphi$};
 \draw[very thick, purple, -stealth] (-2,-5)--(1, -5);
\node at (-1, -4.5){$\alpha'$};

\begin{scope}[shift={(2,0)},rotate=-90, scale=.8]
\draw[very thick] (2, -.5) to (2, 4.5);
\draw[very thick, dashed] (0, -0.2) to (1, -0.2);
\draw[very thick] (0,4.5) to (0,4) .. controls ++(0,-.5) and ++(0,-.5) .. (1,4) to (1,4.5);
\draw[very thick] 
(-1,-0.5)  to (-1,1.25) .. controls ++(0,.5) and ++(0,.5) .. 
(0,1.25) to (0,1)   .. controls ++(0,-0.5) and ++(0,-.5) .. 
(1,1) to (1,2.75) .. controls ++(0,.5) and ++(0,.5) .. 
(0,2.75) to (0,2.5) .. controls ++(0,-0.5) and ++(0,-.5) .. (-1,2.5) to (-1, 4.5);
\draw[very thick] (0,-0.5)  to (0,0) .. controls ++(0,.5) and ++(0,.5) .. (1,0)--(1,-0.5);

\end{scope}
  
\begin{scope}[shift={(-4,-5)}, rotate=-90, scale=.8]  
\draw[very thick] (-1, -3.5) to (-1, 1.5);
\draw[very thick] (2, -3.5) to (2, 1.5);

\draw[very thick] (1,1.5) to (1,0.5) .. controls ++(0,-.5) and ++(0,-.5) .. (0,0.5) to (0,1.5);
\draw[very thick] (1,-1.5) to (1,-0.5) .. controls ++(0,.5) and ++(0,.5) .. (0,-0.5)--(0,-1.5) .. controls ++(0,-0.5) and ++(0,-.5) .. (1,-1.5);
\node at (0.5, -1){$+$}; 
\draw[very thick] (0,-3.5)  to (0,-2.5) .. controls ++(0,.5) and ++(0,.5) .. (1,-2.5)--(1,-3.5);

 \end{scope}

\begin{scope}[shift={(2,-5)}, rotate=-90, scale=.8]  

\draw[very thick] (2, -.5) to (2, 4.5);
\draw[very thick] (0,4.5) to (0,4) .. controls ++(0,-.5) and ++(0,-.5) .. (1,4) to (1,4.5);
\draw[very thick] 
(-1,-0.5)  to (-1,1.25) .. controls ++(0,.5) and ++(0,.5) .. 
(0,1.25) to (0,1)   .. controls ++(0,-0.5) and ++(0,-.5) .. 
(1,1) to (1,2.75) .. controls ++(0,.5) and ++(0,.5) .. 
(0,2.75) to (0,2.5) .. controls ++(0,-0.5) and ++(0,-.5) .. (-1,2.5) to (-1, 4.5);

\draw[very thick] (0,-0.5)  to (0,0) .. controls ++(0,.5) and ++(0,.5) .. (1,0)--(1,-0.5);
\draw[very thick, red] (0,0) .. controls ++(0,.5) and ++(0,.5) .. 
(1,0);
\draw[very thick, red] (0,1) .. controls ++(0,-0.5) and ++(0,-.5) .. (1,1);

 \end{scope}
   \end{tikzpicture}
    \subcaption{}
   \end{subfigure}
\caption{\label{f.config4} \label{f.config4}  (A): The barring of $\varphi$ corresponding to the red saddle creates the subword in $t'$ that is the target of a $G2$ Gaussian elimination isomorphism.  (B): The barring of $\varphi$ corresponding to the red saddle in the other possible location creates the subword in $t'$ that is the target of a $G2$ Gaussian elimination isomorphism.  
}
\end{figure}
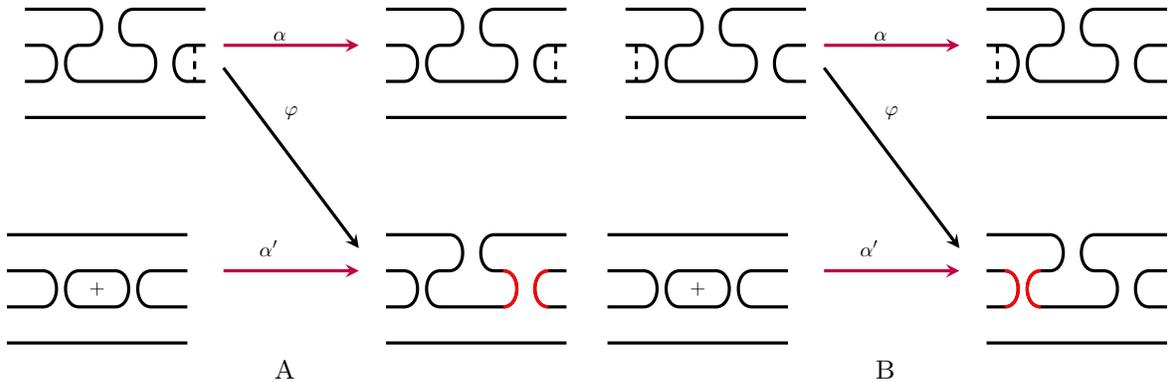 
\end{itemize}

Now we show that the graph $G$ with vertices $V = V(G)$, the set of distinguished Gaussian elimination isomorphisms in $\G$, and edges $E = E(G)$, the set of connecting maps, is acyclic. This is not in general true for an arbitrary subset of component maps of the differential in $\CKh(\beta)$, but we are in the special case of a subgraph of the Khovanov complex, where Algorithm \ref{ss.algorithm} chooses a relatively small subset of the edges in the full Khovanov cube in the construction of $G$. Our proof depends on the characterization of a connecting map $\varphi$ as explained in Lemma \ref{l.connecting}, which corresponds to an edge of $G$. The main strategy is a proof by contradiction, where we assume a cycle exists in $G$ and derive a contradiction.

\begin{remark}
At the time of this writing, we believe it is possible to give an alternative proof using the fact that a cycle corresponds to a cobordism obtained by composing components of the differentials between the same pair of enhanced Kauffman states. The cycle in $G$ exists if and only if the cobordism is not the zero map, which translates to the cobordism having low genus. Then, one could obstruct the existence of such low-genus cobordisms. Since this alternative approach still involves analyzing cases for the components of the differential for short cycles, we think that it is more illustrative to stick with our original approach. 
\end{remark} 

Let $G$ be the graph corresponding to the set of Gaussian elimination isomomorphisms $\mathcal{G}$. We will label a vertex by their corresponding Gaussian elimination isomorphism $\alpha: (s, \epsilon) \rightarrow (t, \delta)$. A path $P$ in $G$ is denoted as follows:  $\alpha_1 \stackrel{\varphi_{1}}{\longrightarrow}  \alpha_2 \stackrel{\varphi_{2}}{\longrightarrow} \cdots \stackrel{\varphi_{p-1}}{\longrightarrow} \alpha_{p}. $ 

We define $\triangle(P)$ to keep track of the changes in the enhanced Kauffman states that are the sources and the targets of the vertices through a path $P$ in the graph $G$. 

\begin{definition}{($\triangle(P)$)} Let
$ P: \alpha_1  \stackrel{\varphi_{1}}{\longrightarrow}  \alpha_2 \stackrel{\varphi_{2}}{\longrightarrow} \cdots \stackrel{\varphi_{p-1}}{\longrightarrow} \alpha_{p} $
be a path in $G$, where $\alpha_i: (s_i, \epsilon_i) \rightarrow (t_i, \delta_i)$ are distinguished isomorphisms in $\G$, for all $1\leq i \leq p$. The set $\triangle(P)$ is the set of crossings on which the Kauffman state of $t_{p}$ chooses the $1$-resolution, but the Kauffman state  of $s_1$ chooses the $0$-resolution.  
\end{definition}

Next for $b\in \triangle(P)$, we define $Z(b)$ as the set of crossings which keeps track of other crossings that are not in $\triangle(P)$ relative to the \textit{position} of $b$ in the braid $\beta$, as described below. 

\paragraph{\textbf{Notation.}} We refer to the horizontal position, reading from left to right, of a letter $c$ corresponding to a crossing in the braid word as $h(c)$. We refer to its vertical position as $v(c)$. The vertical position is the strand number of the lower strand that the crossing abuts; see Figure \ref{f:horvet}.
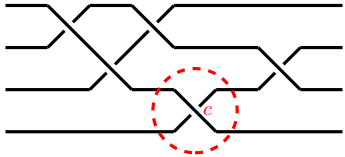
\begin{figure}[H]
\begin{center} 
\begin{tikzpicture}[every node/.style={scale=.7}, scale=.7]
\begin{scope}[shift={(3,-4.5)}, rotate=-90, scale=.8]  
\draw[very thick] (-2,0)  to (-2,1);
\draw[very thick] (-2,2)  to (-2,3);
\draw[very thick] (-2,4)  to (-2,8);

\draw[very thick] (-1,0)  to (-1,1);
\draw[very thick] (-1,4)  to (-1,6);
\draw[very thick] (-1,7)  to (-1,8);

\draw[very thick] (0,0)  to (0,2);
\draw[very thick] (0,3)  to (0,4);
\draw[very thick] (0,5)  to (0,6);
\draw[very thick] (0,7)  to (0,8);

\draw[very thick] (1,0)  to (1,4);
\draw[very thick] (1,5)  to (1,8); 
\draw[very thick, red, dashed ] (0.5,4.5) circle (1);

\draw[very thick] (-1,1)  to (-2,2);
\draw[ultra thick, white] (-1.4,1.4)  to (-1.6,1.6);
\draw[very thick] (0,2)  to (-2,4);
\draw[ultra thick, white] (-0.4,2.4)  to (-0.6,2.6);
\draw[ultra thick, white] (-1.4,3.4)  to (-1.6,3.6);
\draw[very thick] (1,4)  to (0,5);
\draw[ultra thick, white] (0.6,4.4)  to (0.4,4.6);
\draw[very thick] (0,6)  to (-1,7);
\draw[ultra thick, white] (-0.4,6.4)  to (-0.6,6.6);
\draw[very thick] (-2,1)  to (0,3);
\draw[very thick] (-2,3)  to (-1,4);
\draw[very thick] (0,4)  to (1,5);
\draw[very thick] (-1,6)  to (0,7);

\node[red] at (0.5,4.8){$c$};
 \end{scope}
\end{tikzpicture}\end{center} 
\caption{Horizontal and vertical positions of a crossing $c$ such that $h(c)=4$ and $v(C)=1.$}\label{f:horvet} 
\end{figure}

\begin{definition}[$Z(b)$]
Let $P$ be a path in $G$ and let $b\in \triangle(P)$. For $j \leq v(b)$ and $c$ such that $v(c) = j$ and $h(c) = \min \{ c \in \beta \ | \ h(c) \geq h(b) \}$, define 
\[Z_{j}(b):= \{ a \in \beta \ | \ h(a) \leq h(c) \}.    \]  Then 
$ Z(b) := \cup_{j\leq v(b)} Z_j(b). $
We will refer to  such crossings as \textit{trapezoidal}, due to the shape of the enclosing area; see Figure  \ref{f.trapez} for an example. 

\begin{figure}[H]
\begin{center} 
\begin{tikzpicture}[every node/.style={scale=.7}, scale=.7]
\begin{scope}[shift={(3,-4.5)}, rotate=-90, scale=.8]  
\draw[very thick] (-2,-0.5)  to (-2,3.5) .. controls ++(0,.5) and ++(0,.5) .. (-1,3.5)--(-1,3) .. controls ++(0,-0.5) and ++(0,-.5) .. (0,3)--(0,6);
\draw[very thick] (-2,6) to (-2,4.5) .. controls ++(0,-.5) and ++(0,-.5) .. (-1,4.5) to (-1,6);
\draw[very thick] (-1,-0.5)  to (-1,2) .. controls ++(0,.5) and ++(0,.5) .. (0,2)--(0,1.5) .. controls ++(0,-0.5) and ++(0,-.5) .. (1,1.5)--(1,6);
\draw[very thick] (0,-0.5)  to (0,0.5) .. controls ++(0,.5) and ++(0,.5) .. (1,0.5)--(1,-0.5);
\node[red] at (-0.5, 2){$b$};
\draw[very thick, red] (-1,3) .. controls ++(0,-.5) and ++(0,-.5) .. (0,3);
\draw[very thick, red] (-1,2) .. controls ++(0,.5) and ++(0,.5) .. (0,2);
\draw[very thick] (2,-0.5)  to (2,6);
\draw[line width = 1pt, dashed] (-2,5.5)--(-1,5.5);
\draw[line width = 1pt, dashed] (-2,-0.2)--(-1,-0.2);
\draw[line width = 1pt, dashed] (-1,5)--(0,5);
\draw[line width = 1pt, dashed] (0,4.3)--(1,4.3);
\draw[line width = 1pt, dashed] (1,4)--(2,4);
\draw[line width = 1pt, dashed] (1,0)--(2,0);
\draw[very thick, cyan] (2.5,-0.7) to (-1.5,-0.7) to (-1.5, 3.3) to (2.5, 6.5) to  (2.5, -0.7) ;
 \end{scope}
\end{tikzpicture}\end{center} 
\caption{The set of crossings $Z(b)$ is all the crossings in the trapezoidal region as shown. } 
\label{f.trapez}
\end{figure}
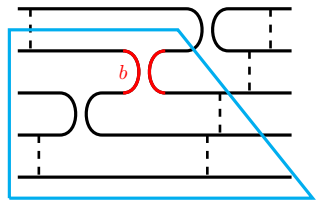 
\end{definition}

We introduce some more notation and definitions for keeping track of the changes to the barring of barred braid words that occur as Kauffman states of sources corresponding to vertices in a path of $G$. 

Let $P: \alpha_1 \rightarrow \cdots \rightarrow \alpha_p$ be a path in $G$. For $1< i \leq p$ and vertex $\alpha_i:(s_i,\epsilon_i) \rightarrow (t_i, \delta_i)$, let $\triangle^i(P) := \triangle(P_i)$ be the set of crossings on which the Kauffman state $s_i$  bars a letter $a$, but the letter $a$ is unbarred in $s_1$ for the subpath 
$P_i : \alpha_1 \rightarrow \cdots \rightarrow \alpha_i$ ending at $\alpha_i$. Let $b\in \triangle^i(P)$. Then $Z_i(b)$ is the barred braid word corresponding to the subword of $s_i$, consisting of the set of crossings in $Z(b)$ with the barring information from $s_i$. 

\begin{lemma} \label{l.unbarred_diff} 
Suppose we have a path $P: \alpha_1 \rightarrow \cdots \rightarrow \alpha_p$ in $G$. Let $b$ be a crossing that is barred in $s_1$. Suppose that $b$ becomes unbarred at the end of the path. That is, it is unbarred in  $s_p$ and this is the first instance in the path $P$ where this happens. Furthermore, $Z_p(b) = Z_1(p)$. 

Then there is an $i$, $ 1 < i < p$, and a crossing $b'\in \beta$, such that
\begin{itemize} 
\item[(1)] $b'\in \triangle^i(P)$, and
\item[(2)] $Z_i(b') =  Z_1(b')$, and at least one of the following is true:
\begin{itemize}
\item[(a)] $b' < b$, or
\item[(b)] $b' \in Z(b)$.
    \end{itemize}
\end{itemize}

\begin{proof}
We provide a proof by induction on the number of crossings in $Z(b)$.

\paragraph{\textbf{Base case: $|Z(b)| = 0$.}}
If the set $Z(b)$ is empty, there is nothing to prove, as the barring on $b$ is not removable by a distinguished Gaussian elimination isomorphism. The path $P$ does not exist. We rule out this case; see Figure \ref{f:nonremovable}. 

\begin{figure}[H]
\centering
\begin{tikzpicture}[every node/.style={scale=.7}, scale=.6]

\begin{scope}[shift={(-5,-5)},rotate=-90, scale=.8] 
\draw[very thick] (-2, -.5) to (-2, 4);
\draw[very thick] (-1, -.5) to (-1, 4);
\draw[very thick] (1,4) to (1,1.5) .. controls ++(0,-.5) and ++(0,-.5) .. (0,1.5) to (0,4);
\draw[line width = 1pt, dashed] (-1,3)--(-2,3);
\draw[line width = 1pt, dashed] (0,2)--(-1,2);
\draw[very thick] (0,-0.5)  to (0,0.5) .. controls ++(0,.5) and ++(0,.5) .. (1,0.5)--(1,-0.5);
\draw[very thick, red] (0,0.5) .. controls ++(0,.5) and ++(0,.5) .. 
(1,0.5);
\draw[very thick, red] (0,1.5) .. controls ++(0,-0.5) and ++(0,-.5) .. (1,1.5);
 \end{scope}
 \end{tikzpicture}
 \caption{Example for a nonremovable barring on $b$. The letter $b$ is the first letter in the braid word and the barring on $b$ cannot be removed by any distinguished Gaussian elimination isomorphism, as it cannot be the active crossing of the target of any Gaussian elimination isomorphism.}\label{f:nonremovable}
\end{figure}
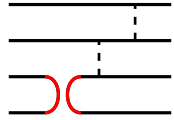

\paragraph{\textbf{Inductive hypothesis.}}
We assume $Z(b)$ with $|Z(b)|=n$ is nonempty, and that for all $b$,  such that $Z(b) < n $ there is an $i$, $ 1 < i < p$, and a crossing $b'\in \beta$, such that conditions (1) and (2) of the lemma are satisfied.

For the crossing $b$ to be unbarred in $s_p$, this must occur as the result of going in the reverse direction of the distinguished Gaussian elimination isomorphism $\alpha_p$. Therefore, the crossing $b$ must be part of a subword of $t_p$ that is the target of $\alpha_p$. We consider the two cases depending on whether the distinguished isomorphism $\alpha_p$ is of type G1 or G2.  
\begin{itemize} 
\item \textbf{The distinguished Gaussian elimination isomorphism $\alpha_p$ is of type G1.} This means that $b = \bar{j}$, the second barred $j$ in the subword $w = \bar{j}x\bar{j} \in t_p$; see Figure \ref{f.config1}.

\begin{figure}[H]
    \centering
    \begin{tikzpicture}[every node/.style={scale=.7}, scale=.6]

\begin{scope}[rotate=-90]  

\draw[very thick] (1,1.5) to (1,0.5) .. controls ++(0,-.5) and ++(0,-.5) .. (0,0.5) to (0,1.5);
\draw[very thick] (1,-1.5) to (1,-0.5) .. controls ++(0,.5) and ++(0,.5) .. (0,-0.5)--(0,-1.5) .. controls ++(0,-0.5) and ++(0,-.5) .. (1,-1.5);
\node at (-0.5, -2){$\overline{j}$}; 
\node at (-0.5, 0){$\overline{j}$}; 
\node at (0.5, -1){$-$}; 
\draw[very thick] (0,-3.5)  to (0,-2.5) .. controls ++(0,.5) and ++(0,.5) .. (1,-2.5)--(1,-3.5);

 \end{scope}
 \end{tikzpicture}
\caption{Local picture of $t_p$.} \label{f.config1}
\end{figure}
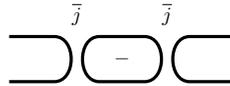

We examine the possible barrings of crossings around the subword $t_p$ as shown in Figure $\ref{f.config1}$; see Figure \ref{f21:2} for an illustration.

Denote $a = \init(\alpha_p)$. If $a$ is not barred in $s_1$, then we choose $b' = a$.  Otherwise we have a subword $w(1) = \bar{j}x_1\bar{j} = s_1[\init(w), \fin(w)]$ in $s_1$. Here there are two cases for this subword in $s_1$: 
\begin{enumerate}
\item $w(1)$ is a target of a G1 or G2 Gaussian elimination isomorphism. \\ 
Compare $w(1)$ with $w_1\subseteq s_1$, where $w_1$ is the support of the distinguished Gaussian elimination isomorphism $\alpha_1$. We must have $w_1 < w(1)$. That is, $\init(\alpha_1) < \init(\alpha_p)$. Otherwise, $\alpha_1$ would not be chosen as the distinguished Gaussian elimination isomorphism for $s_1$ following Algorithm \ref{ss.algorithm}. By Lemma \ref{l.connecting}, $\varphi_1$ must obstruct $\alpha_1$ or make $\alpha_2$. This means there exists a $b' \in s_2[1, \fin(\alpha_1)]$, which is barred in $s_2$ but not barred in $s_1$. We can directly check that $b'$ satisfies the conditions of Lemma \ref{l.unbarred_diff}, in particular, $b'< b$. See Figure \ref{f21:1} for the case where $w(1)$ is the target of a G2 Gaussian elimination isomorphism. 

\begin{figure}[htbp!]
 \centering
\begin{subfigure}{.45\textwidth}
\centering
\begin{tikzpicture}[every node/.style={scale=.7}, scale=.6]

\begin{scope}[rotate=-90]

\draw[very thick] (0,4.5) to (0,4) .. controls ++(0,-.5) and ++(0,-.5) .. (1,4) to (1,4.5);

\draw[very thick] 
(2, 4.5) to (2,2.75)
.. controls ++(0,-.5) and ++(0,-.5) .. 
(1,2.75) .. controls ++(0,.5) and ++(0,.5) .. 
(0,2.75)  .. controls ++(0,-0.5) and ++(0,-.5) .. (-1,2.75) to (-1, 4.5);

\draw[very thick] 
(-1,-0.5)  to (-1,1.25) .. controls ++(0,.5) and ++(0,.5) .. 
(0,1.25) 
.. controls ++(0,-0.5) and ++(0,-.5) .. 
(1,1.25) .. controls ++(0,.5) and ++(0,.5) .. 
(2,1.25) to (2,-0.5);

\draw[very thick] (0,-0.5)  to (0,0) .. controls ++(0,.5) and ++(0,.5) .. (1,0)--(1,-0.5);

\draw[very thick] (0,-0.5)  to (0,0) .. controls ++(0,.5) and ++(0,.5) .. (1,0)--(1,-0.5);
\node at (-0.5, 2.6){$c$}; 
\node at (1.5, 2.6){$d$}; 
\node at (1.2, 0.5){$a$}; 
\node at (1.2, 3.5){$b$}; 
\node at (-0.2, 0.5){$\bar{j}$}; 
\node at (-0.2, 3.5){$\bar{j}$}; 
\end{scope}
    \end{tikzpicture} \caption{}\label{f21:2}
    \end{subfigure}   \begin{subfigure}{.45\textwidth}
    \centering
\begin{tikzpicture}[every node/.style={scale=.7}, scale=.6]

\begin{scope}[rotate=-90] 
\draw[very thick] (0,4.5) to (0,4) .. controls ++(0,-.5) and ++(0,-.5) .. (1,4) to (1,4.5);
\draw[very thick] 
(-1,-0.5)  to (-1,1.25) .. controls ++(0,.5) and ++(0,.5) .. 
(0,1.25) to (0,1)   .. controls ++(0,-0.5) and ++(0,-.5) .. 
(1,1) to (1,2.75) .. controls ++(0,.5) and ++(0,.5) .. 
(0,2.75) to (0,2.5) .. controls ++(0,-0.5) and ++(0,-.5) .. (-1,2.5) to (-1, 4.5);
\draw[very thick] (0,-0.5)  to (0,0) .. controls ++(0,.5) and ++(0,.5) .. (1,0)--(1,-0.5);
\node at (-0.5, 2.6){$c$}; 
\node at (1.2, 0.5){$a$}; 
\node at (1.2, 3.5){$b$}; 
\node at (-0.2, 0.5){$\bar{j}$}; 
\node at (-0.2, 3.5){$\bar{j}$}; 
\end{scope}
\end{tikzpicture} \caption{}\label{f21:1}
\end{subfigure} 
\caption{(A): Possible barrings around $w(1)$  and the case (B): where $w(1)$ is the target of a G2 Gaussian elimination isomorphism.}\label{f.21}
\end{figure}
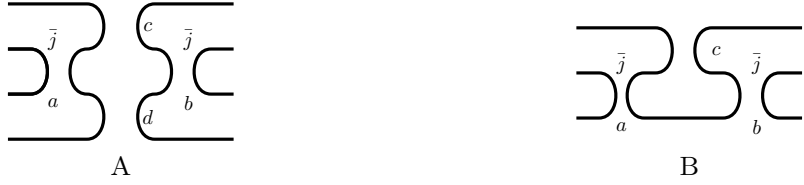

\item $w(1)$ is not a target of a G1 of G2 Gaussian elimination isomorphism. \\ 
The possibilities for the barring on the subword $w(1)$ are as shown in Figures \ref{f21:1}.

Note in particular that we do not need to consider the case when $c$ is barred but $d$ is not, since we assume that $w(1)$ is not a target of a G1 or G2 Gaussian elimination isomorphism. Thus we may always assume that there is a crossing $d$, where $a = s_1[\init(w)] < d < b = s_1[\fin(w)]$, which is barred between $a$ and $b$ in $s_1$. 
In this case we have a crossing $d$, $a<d<b$, which is barred between $a$ and $b$ in $s_1$: 

For the barring on $b$ to be removed as the active crossing of $\alpha_p$, the crossing $d$ must be removed at some $s_i$, with $i < p$. We can apply the induction hypothesis to $d$ to find a $b'$, noting that $|Z(d)| < |Z(b)|$. 
\end{enumerate} 
\item \textbf{The distinguished Gaussian elimination isomorphism $\alpha_2$ is of type G2. }
This means that $b = \overline{(j+1)}$ in the subword $w = \bar{j} \overline{(j+1)} x_p \bar{j} \subseteq  s_p$ that is the target of $\alpha_p$ (see Figure \ref{f.abc}).

\begin{figure}
    \centering
    \begin{tikzpicture}
\begin{scope}[shift={(2,-5)}, rotate=-90, scale=.6]  

\draw[very thick] (0,4.5) to (0,4) .. controls ++(0,-.5) and ++(0,-.5) .. (1,4) to (1,4.5);
\draw[very thick] 
(-1,-0.5)  to (-1,1.25) .. controls ++(0,.5) and ++(0,.5) .. 
(0,1.25) to (0,1)   .. controls ++(0,-0.5) and ++(0,-.5) .. 
(1,1) to (1,2.75) .. controls ++(0,.5) and ++(0,.5) .. 
(0,2.75) to (0,2.5) .. controls ++(0,-0.5) and ++(0,-.5) .. (-1,2.5) to (-1, 4.5);

\draw[very thick] (0,-0.5)  to (0,0) .. controls ++(0,.5) and ++(0,.5) .. (1,0)--(1,-0.5);
\draw[very thick, red] (-1,1.25) .. controls ++(0,.5) and ++(0,.5) .. 
(0,1.25);
\draw[very thick, red] (-1,2.5) .. controls ++(0,-0.5) and ++(0,-.5) .. (0,2.5);
\node at (-0.5, 3){$b$}; 
\node at (1.5, 0.5){$a$}; 
\node at (1.5, 3.5){$c$}; 
 \end{scope}
   \end{tikzpicture} 
\caption{\label{f.abc} The crossing $b$ corresponding to $\overline{(j+1)}$ is shown in red. }
\end{figure}
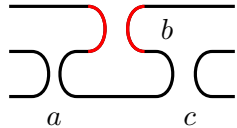
\end{itemize}
Let $a= \init(\alpha_p)$ and $c = \fin(\alpha_p)$. If any one of $a, b$ is not barred in $s_1$, then we can choose $b'$ as either $a$ or $c$. 
Otherwise, we consider $w(1) = s_1[\init(w), \fin(w)]$. There are two cases: 
\begin{enumerate}
\item $w(1)$ is the target of a G1 or G2 Gaussian elimination isomorphism. \\
In this case, $w(1)$ is necessarily the target of a G2 Gaussian elimination isomorphism, because otherwise, if $w(1)$ is the target of a G1 Gaussian elimination isomorphism, the corresponding cobordism would go from an enhanced Kauffman state with a closed component  marked with a $-$ to an enhanced Kauffman state with the same circle marked with a $+$. For grading reasons, the cobordism vanishes.  Comparing with the support $w_1$ of $\alpha_1$, it must be the case that $\init(w_1) < \init(w)$, otherwise $\alpha_1$ would not be chosen. Now $\varphi_1$ must obstruct $\alpha_1$ or make $\alpha_2$, which implies that there exists a $b' \in s_2[1, \init(\alpha_1)]$ which is barred in $s_2$ but not barred in $s_1$. We can directly check that $b'$ satisfies all the conditions of the lemma.

\item $w(1)$ is not the target of a G1 or G2 Gaussian elimination isomorphism.  \\ 
In this case, we have a crossing $d$, $a< d< c$, as shown in Figure \ref{f.abcd}.  

\begin{figure}
\centering
\begin{tikzpicture}[every node/.style={scale=.7}, scale=.6]
\begin{scope}[rotate=-90, scale=.8]
\draw[very thick] (0,6.5) to (0,6).. controls ++(0,-.5) and ++(0,-.5) .. (1,6) to (1,6.5);
\draw[very thick] 
(2, 6.5) to (2,4)
.. controls ++(0,-.5) and ++(0,-.5) .. 
(1,4) to (1,5)
.. controls ++(0,.5) and ++(0,.5) .. 
(0,5) to (0,2.75) 
.. controls ++(0,-0.5) and ++(0,-.5) .. (-1,2.75) to (-1, 6.5);

\draw[very thick] 
(-1,-0.5)  to (-1,1.25) .. controls ++(0,.5) and ++(0,.5) .. 
(0,1.25) 
.. controls ++(0,-0.5) and ++(0,-.5) .. 
(1,1.25) to (1,2.75)
.. controls ++(0,.5) and ++(0,.5) .. 
(2,2.75) to (2,-0.5);

\draw[very thick] (0,-0.5)  to (0,0) .. controls ++(0,.5) and ++(0,.5) .. (1,0)--(1,-0.5);

\draw[very thick] (0,-0.5)  to (0,0) .. controls ++(0,.5) and ++(0,.5) .. (1,0)--(1,-0.5);

\draw[very thick] (2,2.75) .. controls ++(0,.5) and ++(0,.5) .. 
(1,2.75);
\draw[very thick] (1,4) .. controls ++(0,-0.5) and ++(0,-.5) .. (2,4);

\node at (-0.5, 1){$b$}; 
\node at (1.25, 0.6){$a$}; 
\node at (1.5, 2.75){$c$}; 
\node at (1.5, 4.25){$d$}; 
\end{scope}
  
   \end{tikzpicture}
\caption{Possibilities for configuration of barring when the distinguished Gaussian elimination isomorphism is of type G2.}\label{f.abcd}
\end{figure}
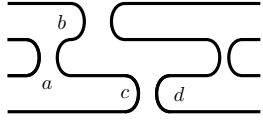
\end{enumerate}
If $d$ is not barred in $s_1$, then we can choose $b' = d$. Otherwise, we assume that $d$ is barred in $s_1$. In order for the barring on $b$ to be removed as the active crossing of a G2 Gaussian elimination isomorphism $\alpha_p$, the barring on $d$ must be removed at $s_k$, where $k < p$. We apply the induction hypothesis to $d$ to find a $b'$ satisfying the conditions of the lemma. 
\end{proof} 
\end{lemma}

With the machinery now in place, we prove the main theorem of this section. Let $\mathcal{G}$ be the set of distinguished Gaussian elimination isomorphisms for the braid $\ftbraid_n^k$. Recall $G$ is the graph constructed from $\mathcal{G}$ whose vertices are elements of $\mathcal{G}$ and edges correspond to nonzero components of the differentials in $\CKh(\ftbraid_n^k)$. 

\begin{theorem} \label{t.G-acyclic}
The graph $G$ is acyclic. 
\begin{proof}
    Suppose on the contrary that there is a path in $G$ that is a cycle: 
    \[P: \alpha_1 \rightarrow \cdots \rightarrow \alpha_p = \alpha_1. \]
For every $1< i \leq p$, any crossing in $\triangle^i(P)$ must be unbarred before it reaches the final destination $\alpha_p = \alpha_1$. This is because if $\triangle^{(p-1)}(P)$ is nonempty, then every component of the differential from $s_{p-1} \rightarrow t_1$ is zero. 

Our main strategy is to show that the unbarring of a crossing in $\triangle^i(P)$ is either impossible, or cannot be done without creating a new element in $\triangle^j(P)$ for some $j>i$. Therefore, $\triangle^i(P)$ is never empty, contradicting the existence of the cycle $P$. 

It is not hard to check that $\triangle^2(P) \not= \emptyset$ for any path $P$. Let $b\in \triangle^2(P)$. 

We provide a proof by induction on $|Z(b)|$. 

\paragraph{\textbf{Base case $|Z(b)| = 0$.}} For the base case when $Z(b) = \emptyset$ and so $|Z(b)| = 0$, the barring on $b$ is not removable, and so $\triangle^i(P)$ is never empty for any $1<i<p$ since it will always contain $b$. 

\paragraph{\textbf{Inductive step $|Z(b)| > 0$.}} In general, we let $j$ be the index when $b$ becomes unbarred in the cycle $P$. That is, $b = \AC(\alpha_j)$. We will always choose the smallest possible $j< p$ in the path. We also assume $b$ is the leftmost crossing in $\cup_i \triangle^i(P)$, meaning that $h(b)$ is the minimum over all $b\in \cup_i \triangle^i(P)$.

As in the proof of Lemma \ref{l.unbarred_diff},  we organize cases based on the type of Gaussian elimination isomorphism for $\alpha_j$ (G1 or G2). 

\begin{itemize}
\item \textbf{The distinguished isomorphism $\alpha_j$ is of type G1.} \\
Let $w_j \subseteq s_j$ be the word on which $\alpha_j$ is supported, and $w_i \subseteq s_i$ be the word on which $\alpha_i$ is supported.
Since $b$ is the leftmost crossing in $\cup_i \triangle^i(P)$ and $h(a) < h(b)$, we have that $a = \init(w_j)$ is barred in $s_i$ and $s_1$. 

Consider $w = t_j[a, b]$. This is the target of the G1 Gaussian elimination isomorphism $\alpha_j$.  The subword $w' = s_{j+1}[a, b]$ is the support of the Gaussian elimination isomorphism $\alpha_{j+1}$. We compare $w$ and $w'$ with the subword $w_1'=s_1[\init(w), \fin(w)]$. 
\begin{enumerate}
    \item \textbf{$w_1'$ is the support of a G1 or G2 Gaussian elimination isomorphism. } \\
    In this case, either $w_1' = w_1$, where $w_1$ is the support of $\alpha_1$, in which case $b$ is barred in $s_1$, or $w_1'$ is not the support of $\alpha_1$. The existence of the subword $w_1'$ in $s_1$ and the fact that $\alpha_1$ is the chosen Gaussian elimination isomorphism for $s_1$ imply $w_1 < w_1'$. The obstruction by $\varphi_1$ of $\alpha_1$, or to make $\alpha_2$ by Lemma \ref{l.connecting} would contradict the minimality of $b$. 
    \item \textbf{$w_1'$ is not the support of a G1 or G2 Gaussian elimination isomorphism. } \\
    The possible crossings around $w_1'$ that are all barred are shown in Figure \ref{f.configabcd}. 
\begin{figure}[H]
\begin{center}
\begin{tikzpicture}[every node/.style={scale=.7}, scale=.6]
\begin{scope}[rotate=-90, scale=.8]
\draw[very thick] (1,6.5) to (1,4).. controls ++(0,-.5) and ++(0,-.5) .. (2,4) to (2,6.5);
\draw[very thick]  
(0,6.5) to (0,2.75) 
.. controls ++(0,-0.5) and ++(0,-.5) .. (-1,2.75) to (-1, 6.5);

\draw[very thick] 
(-1,-0.5)  to (-1,1.25) .. controls ++(0,.5) and ++(0,.5) .. 
(0,1.25) 
.. controls ++(0,-0.5) and ++(0,-.5) .. 
(1,1.25) to (1,2.75)
.. controls ++(0,.5) and ++(0,.5) .. 
(2,2.75) to (2,-0.5);

\draw[very thick] (0,-0.5)  to (0,0) .. controls ++(0,.5) and ++(0,.5) .. (1,0)--(1,-0.5);

\draw[very thick] (0,-0.5)  to (0,0) .. controls ++(0,.5) and ++(0,.5) .. (1,0)--(1,-0.5);

\draw[very thick] (2,2.75) .. controls ++(0,.5) and ++(0,.5) .. 
(1,2.75);
\draw[very thick] (1,4) .. controls ++(0,-0.5) and ++(0,-.5) .. (2,4);

\node at (-0.5, 1){$c$}; 
\node at (1.25, 0.6){$a$}; 
\node at (1.5, 2.75){$d$}; 
\node at (0.5, 4.25){$b$};
\draw[very thick, dashed, red] (0,4.5) to (1,4.5);
\end{scope}  
   \end{tikzpicture}

    \caption{\label{f.configabcd} Possible barrings of crossings around the subword $w_1'$.} 
    \end{center} 
    \end{figure}
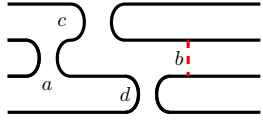

There are three distinct possibilities, as shown in Figure \ref{f.w113possibilities}. 

    \begin{figure}[h]
        \centering
\begin{subfigure}{.25\textwidth}
\centering 
\begin{tikzpicture}[scale=0.6]
\begin{scope}[rotate=-90]  
\draw[very thick] (-1,4) to (-1,3) .. controls ++(0,-.5) and ++(0,-.5) .. (0,3) to (0,4);
\draw[very thick] (-1,-0.5)  to (-1,2) .. controls ++(0,.5) and ++(0,.5) .. (0,2)--(0,1.5) .. controls ++(0,-0.5) and ++(0,-.5) .. (1,1.5)--(1,4);
\node at (0.5, 3.3){$b$}; 
\node at (0.5, 0.6){$a$}; 
\node at (-0.5, 2){$c$}; 
\draw[very thick] (0,-0.5)  to (0,0.5) .. controls ++(0,.5) and ++(0,.5) .. (1,0.5)--(1,-0.5);
\draw[line width = 1pt, dashed, red] (0,3)--(1,3);
 \end{scope}
 \end{tikzpicture} 
\caption{}          
 \end{subfigure}
\begin{subfigure}{.25\textwidth}
\centering 
\begin{tikzpicture}[scale=.6]
\begin{scope}[ rotate=-90]  
\draw[very thick] (0,4) to (0,3) .. controls ++(0,-.5) and ++(0,-.5) .. (1,3) to (1,4);
\draw[very thick] (1,-0.5)  to (1,2) .. controls ++(0,.5) and ++(0,.5) .. (0,2)--(0,1.5) .. controls ++(0,-0.5) and ++(0,-.5) .. (-1,1.5)--(-1,4);
\draw[very thick] (-1,-0.5)  to (-1,0.5) .. controls ++(0,.5) and ++(0,.5) .. (0,0.5)--(0,-0.5);
\node at (-0.5, 3.3){$b$}; 
\node at (-0.5, 0.5){$a$}; 
\node at (0.5, 2){$d$};
 \draw[line width = 1pt, dashed, red] (-1,3)--(0,3);
\end{scope}
 \end{tikzpicture} 
\caption{}          
 \end{subfigure}
\begin{subfigure}{.25\textwidth}
\centering
\begin{tikzpicture}[every node/.style={scale=.7}, scale=.6]
\begin{scope}[rotate=-90, scale=.8]
\draw[very thick] (1,6.5) to (1,4).. controls ++(0,-.5) and ++(0,-.5) .. (2,4) to (2,6.5);
\draw[very thick]  
(0,6.5) to (0,2.75) 
.. controls ++(0,-0.5) and ++(0,-.5) .. (-1,2.75) to (-1, 6.5);

\draw[very thick] 
(-1,-0.5)  to (-1,1.25) .. controls ++(0,.5) and ++(0,.5) .. 
(0,1.25) 
.. controls ++(0,-0.5) and ++(0,-.5) .. 
(1,1.25) to (1,2.75)
.. controls ++(0,.5) and ++(0,.5) .. 
(2,2.75) to (2,-0.5);

\draw[very thick] (0,-0.5)  to (0,0) .. controls ++(0,.5) and ++(0,.5) .. (1,0)--(1,-0.5);

\draw[very thick] (0,-0.5)  to (0,0) .. controls ++(0,.5) and ++(0,.5) .. (1,0)--(1,-0.5);

\draw[very thick] (2,2.75) .. controls ++(0,.5) and ++(0,.5) .. 
(1,2.75);
\draw[very thick] (1,4) .. controls ++(0,-0.5) and ++(0,-.5) .. (2,4);

\node at (-0.5, 1){$c$}; 
\node at (1.25, 0.6){$a$}; 
\node at (1.5, 2.75){$c$}; 
\node at (0.5, 4.25){$b$};
\draw[very thick, dashed, red] (0,4.5) to (1,4.5);
\end{scope}  
   \end{tikzpicture}
   \caption{}    
\end{subfigure}

  \caption{All three possibilities for the barring of $c$, $d$ when $w_1'$ is not the support of a G1 or G2 Gaussian elimination isomorphism: (A)--only $c$ is barred, (B)--only $d$ is barred, and (C)--both $c$ and $d$ are barred.}
        \label{f.w113possibilities}
    \end{figure}
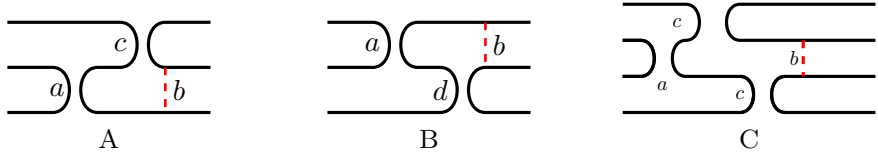

For each of these cases the barring on crossings $c, d$ would have to be removed at $k < j$ in the path $P$, in order for the barring on $c$ to be removed. 

Hence by Lemma \ref{l.unbarred_diff}, there exists some crossing $b' \in \triangle^k(P)$, where $b'<b$ or $b'\in Z(b)$. The first case $b' \in \triangle^k(P)$ contradicts the minimality of $b$. The second case fits into the induction hypothesis of decreasing $|Z(b')|$. 
\end{enumerate}

\item \textbf{The distinguished isomorphism $\alpha_j$ is of type G2.}  \\ 
Let $w_j \subseteq s_j$ be the word on which $\alpha_j$ is supported, and $w_2 \subseteq s_2$ be the word on which $\alpha_2$ is supported. The subword $w_j$ is shown in Figure \ref{f.config2abc}. 


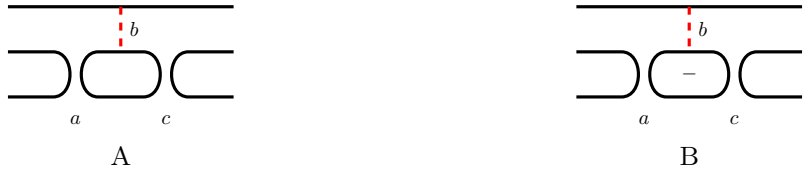
\begin{figure}[H]
    \centering
 \begin{subfigure}{.45\textwidth}
 \centering    
\begin{tikzpicture}[every node/.style={scale=.7}, scale=.6]
\centering
\begin{scope}[rotate=-90]  
\draw[very thick] (-1,-3.5) to (-1,1.5);
\draw[very thick, dashed, red] (-1, -1) to (0,-1);
\node at (-0.5, -0.7){$b$}; 
\draw[very thick] (1,1.5) to (1,0.5) .. controls ++(0,-.5) and ++(0,-.5) .. (0,0.5) to (0,1.5);
\draw[very thick] (1,-1.5) to (1,-0.5) .. controls ++(0,.5) and ++(0,.5) .. (0,-0.5)--(0,-1.5) .. controls ++(0,-0.5) and ++(0,-.5) .. (1,-1.5);
\draw[very thick] (0,-3.5)  to (0,-2.5) .. controls ++(0,.5) and ++(0,.5) .. (1,-2.5)--(1,-3.5);
\node at (1.5, -2){$a$}; 
\node at (1.5, 0){$c$}; 
 \end{scope}
 \end{tikzpicture}
 \caption{}\label{f.config2abc}
 \end{subfigure}
  \begin{subfigure}{.45\textwidth}
 \centering    
\begin{tikzpicture}[every node/.style={scale=.7}, scale=.6]
\centering
\begin{scope}[rotate=-90]  
\draw[very thick] (-1,-3.5) to (-1,1.5);
\draw[very thick, dashed, red] (-1, -1) to (0,-1);
\node at (-0.5, -0.7){$b$}; 
\draw[very thick] (1,1.5) to (1,0.5) .. controls ++(0,-.5) and ++(0,-.5) .. (0,0.5) to (0,1.5);
\draw[very thick] (1,-1.5) to (1,-0.5) .. controls ++(0,.5) and ++(0,.5) .. (0,-0.5)--(0,-1.5) .. controls ++(0,-0.5) and ++(0,-.5) .. (1,-1.5);
\node at (0.5, -1){$-$};
\draw[very thick] (0,-3.5)  to (0,-2.5) .. controls ++(0,.5) and ++(0,.5) .. (1,-2.5)--(1,-3.5);
\node at (1.5, -2){$a$}; 
\node at (1.5, 0){$c$}; 
 \end{scope}
 \end{tikzpicture}
 \caption{}\label{f.config2ab} 
 \end{subfigure}
\caption{(A): The subword $w_j$ supporting a G2 Gaussian elimination isomorphism  and  (B): the enhanced Kauffman state with a circle marked with a $-$.}
\end{figure}
Since $b$ is the leftmost crossing in $\cup_i \triangle^i(P)$, we have that $ a= \init(w_j)$ is barred in $s_1$.

Consider $c = \fin(w_j)$. If the crossing $c$ is unbarred in $s_1$ then we are done, as we can choose $b' = c$. Otherwise, we may assume that $c$ is barred in $s_1$.

We consider $w_1' = s_1[\init(w_j), \fin(w_j)]$. There are two possibilities. 
\begin{enumerate}
\item \textbf{$w_1'$ is the target of a G1 or G2 Gaussian elimination isomorphism.} \\ 
In this case, the only possibility is that it is the target of a G1 Gaussian elimination isomorphism, since $b$ is not barred in $s_1$ (see Figure \ref{f.config2ab}).  
In this case, the subword $w_1 \subseteq s_1$; that is, the support of the Gaussian elimination isomorphism $\alpha_1$ must necessarily be located in front of $w_1'$ by how we apply Algorithm \ref{ss.algorithm}. That is, $\init(w_1) < \init(w_1')$. However, then the location of $\AC(\varphi_1)$ contradicts the minimality of $b$. 

\item \textbf{$w_1'$ is not the target of a G1 or G2 Gaussian elimination} isomorphism. \\
In this case, with $b$ unbarred in $s_1$, the only possibility is that the crossing $d$ is barred between $a$ and $c$, as shown in Figure \ref{f.config2abd}. 
\begin{figure}[H]
\begin{center}
    \begin{tikzpicture}[scale=.8]
    \begin{scope}[rotate=-90, scale=.8]  
\draw[very thick] (-1, -.5) to (-1, 4.5);
\draw[very thick] (0,4.5) to (0,4) .. controls ++(0,-.5) and ++(0,-.5) .. (1,4) to (1,4.5);
\draw[very thick] 
(2,-0.5)  to (2,1.25) .. controls ++(0,.5) and ++(0,.5) .. 
(1,1.25) to (1,1)   .. controls ++(0,-0.5) and ++(0,-.5) .. 
(0,1) to (0,2.75) .. controls ++(0,.5) and ++(0,.5) .. 
(1,2.75) to (1,2.5) .. controls ++(0,-0.5) and ++(0,-.5) .. (2,2.5) to (2, 4.5);
\draw[very thick] (0,-0.5)  to (0,0) .. controls ++(0,.5) and ++(0,.5) .. (1,0)--(1,-0.5);
\node at (-0.5, 1.7){$b$}; 
\node at (0.5, 0){$a$}; 
\node at (0.5, 2.7){$c$}; 
\node at (1.5, 3){$d$}; 
\draw[line width = 1pt, dashed, red] (-1,1.5)--(0,1.5);
\end{scope}   \end{tikzpicture}
\end{center}
\caption{The subword $w_1'$. }\label{f.config2abd} 
\end{figure}
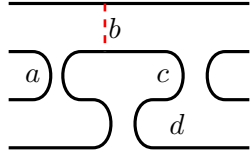
\end{enumerate}
At some point in the path, the barring on the crossing $d$ would have to be removed in order for the barring on $b$ to be removed as the active crossing of the Gaussian elimination isomorphism $\alpha_{p}$. Apply Lemma \ref{l.unbarred_diff} to $b$ to find a $b'$, which contradicts the minimality of $b$. 
\end{itemize}
This completes the proof.
\end{proof}
\end{theorem} 

\begin{definition}[whittled complex] \label{d.whittled_complex}
Let $\CKh(ft^k_n)$ be the Khovanov chain complex for the braid $\beta = ft^k_n$. Let $\G$ be the set of Gaussian elimination isomorphisms selected by Algorithm \ref{ss.algorithm}, and let $G$ be the corresponding graph as in Definition \ref{d.graphisocollection}. Let $<$ be a linear ordering on the $\mathcal{G}$ that exists by Theorem \ref{t.G-acyclic}, and number the elements according to the order as $\alpha_1 < \alpha_2 < \alpha_3 < \ldots < \alpha_{|\G|}$. The whittled complex is defined as the chain complex obtained by applying Gaussian eliminations following the order: 
\[ \FT^k_n := ((\CKh(ft^k_n)_{\alpha_1})_{\alpha_2})\cdots)_{\alpha_{|\G|}}.\]
\end{definition} 

\begin{remark}
    Note that $\FT^k_n$ is homotopy-equivalent to the original chain complex $\CKh(ft^k_n)$.
\end{remark}

\section{Reduction to Jones Normal Form} \label{s.reducetoJNF}
We use the results of the previous sections to prove Theorem \ref{t.deflate}, which we restate here as a proposition. 
\begin{proposition}
    \label{p.deflate}
Let $\FT_n^k$ denote the whittled complex obtained from whittling the Khovanov chain complex $\CKh(\ftbraid_n^k)$, and let $(\sigma, \epsilon)$ be an enhanced Kauffman state (Definition \ref{d.enhanced-K-state}) on $\ftbraid_n^k$. 
Suppose $(\sigma, \epsilon)$ appears in the whittled complex $\FT_n^k$ and let $W$  be the corresponding tangle in the Temperley-Lieb monoid $\TL_n$. 
Then either (1) the word is in the form 
\[ e_{n-1}^{k_0} V_0 e_{n-1}^{k_1} V_1 \cdots V_{r}e_{n-1}^{k_{r}},  \]
where each $V_j = e_{j_1}e_{j_2}\cdots e_{n-1}$ is a subword consisting of consecutive elements in $\TL_n$ ending at $e_{n-1}$ and $k_j \geq 2$ for all $j$. (2) Or, there exists a path of $\TL$-moves (see Definition \ref{n.tlmoves}) transforming $W$ to its Jones Normal Form $\hat W = \JNF(W)$ (Definition \ref{d.jnf}), where each move is one of the following three types:
\begin{itemize}
    \item[(D1)] \label{item:D1} $e_{n-1}^2 \to e_{n-1}$
    \item[(D2)] \label{item:D2} $e_i e_j \to e_je_i$, where $|i-j|\geq 2$
    \item[(D3)] \label{item:D3} $e_ie_{i-1}e_i \to e_i$, for $2 \leq i \leq n-1$
\end{itemize}
\end{proposition}

\begin{remark}
    
Our proof shares similarity with Proposition \ref{prop:monotone-path}, which shows the existence of the Jones Normal Form for a word in $\TL_n$. The key difference is that we consider a much smaller subset of the possible words representing generators in the whittled complex, thanks to whittling by distinguished Gaussian elimination isomorphisms in $\G$. In particular, any word in $\TL_n$ containing subwords of the form $e_i^2$ and $e_ie_{i+1}e_i$ with $1 \leq i \leq  n-2$ does not appear in the whittled complex. 
\end{remark}

Before providing the proof of Proposition \ref{p.deflate}, we consider the following motivating examples.
\subsubsection{Motivating examples.}  \label{ss.motivatingegs}
\begin{itemize}
\item 3-braids: \\
An element in $\TL_3$ is of the form: 
\[ e^{k_0}_2 e_1^{\ell_0} e_2^{k_1} e_1^{\ell_1} \cdots e^{k_n}_2 e^{\ell_m}_1.\]
Note that $\ell_i =1$ for all $0\leq i\leq m$, since otherwise the word would contain $e_1^2$, a subword which supports a G1 Gaussian elimination isomorphism. Thus a word is of one of the following forms:
\begin{enumerate}
\item $e^{k_0}_2 e_1 e_2^{k_1} e_1 \cdots e^{k_n}_2 e_1$, $k_j \geq 2$. 
\item $e^{k_0}_2 e_1 e_2^{k_1} e_1 \cdots e^{k_n}_2$, $k_j \geq 2$.  
\item $e_1 e_2^{k_1} e_1^{\ell_1} \cdots e^{k_n}_2 e_1$, $k_j \geq 2$. 
\item $e_1 e_2^{k_1} e_1^{\ell_1} \cdots e^{k_n}_2$, $k_j \geq 2$. 
\end{enumerate}

\item 4-braids: \\ 
An element in $\TL_4$ is of the form: 
\[ e^{k_0}_3V_0e^{k_1}_3V_1\cdots V_{m-1}e_3^{k_m},\]
where the $V_i$'s are words in $\TL_3$ in the Jones Normal Form. Following deflation as specified by Theorem \ref{t.deflate}, we get that a word is one of the following forms: 
\begin{enumerate}
\item $e^{k_0}_3e_{i_{k_0}}e_{i_{k_0}+1} \cdots e_2 e^{k_1}_3e_{i_{k_1}}e_{i_{k_1}+1}\cdots e_2 e_3^{k_{m-1}}\cdots e_2e_3^{k_m}$
\item $e_3^{k_0}Ve_3^{k_1}$, where $V$ is a word in $TL_3$ in the Jones Normal Form.
\end{enumerate} 
\item 5-braids: \\ 
An element in $TL_5$ has one of the following forms: 
\begin{enumerate}
\item $e^{k_0}_4e_{i_{k_0}}e_{i_{k_0}+1} \cdots e_3 e^{k_1}_4e_{i_{k_1}}e_{i_{k_1}+1}\cdots e_3 e_4^{k_{m-1}}\cdots e_{i_{k_{m-1}}}e_{i_{k_{m-1}}+1}\cdots e_3e_4^{k_m}$
\item $e_4^{k_0}Ve_4^{k_1}$, where $V$ is a word in $TL_3$ in the Jones Normal Form.
\end{enumerate} 
\end{itemize}

\begin{proof} 
We provide a proof of Proposition \ref{p.deflate} by induction on the number of strands of the braid $\ftbraid_n^k$.  

\paragraph{\textbf{The base case $n=2$.}}
On two strands, every word of the form $e_{1}^m$ for $m>1$ is reduced to a single copy of $e_{1}$ by the move D1. However, we can say more.  As the target of a G1 Gaussian elimination isomorphism, any enhanced Kauffman state with at least one closed component marked with a $-$ does not appear in the whittled complex $\FT^k$. The remaining state in $\FT^k$ is the enhanced Kauffman state which marks every closed component of $e_{1}^m$, for some $m$, by a $+$, and we can represent it as $e_1$ with the appropriate degree shift. This observation generalizes to Kauffman states on $n$-braids, for all $n$, where the states contain a subword of the form $e_{i}^m$, where $1 \leq i \leq n-2$. 

\paragraph{\textbf{The case $n\geq 3$.}} It is this case where moves D2 and D3 are used. 

Let $(s, \epsilon)$ be an enhanced Kauffman state in the whittled complex $\FT^k$, and let $S$ be the corresponding word of the Temperley-Lieb monoid $\TL_{n+1}$.

\textbf{Induction Hypothesis.} Assume that for all $n' \leq n$, a barred braid word on $n'$ strands, not in the form of Case (1) in the statement of Proposition \ref{p.deflate}, may be reduced to the Jones Normal Form via moves D1--D3, appropriate to $n'$. 

\textbf{Inductive Step.} We show that a word in $\TL_{n+1}$ may be reduced to the Jones Normal Form via moves D1-D3. 

Let $s$ be a barred braid word representing the Kauffman state of a generator $(\sigma, \epsilon)$ in $\FT_{n+1}^k$, with corresponding Temperley-Lieb word $S \in \TL_{n+1}$. 
First we locate all the consecutive $e_n$'s in the word $S$, and write 
\[ S = e^{k_0}_n \cdot V_0 \cdot e^{k_1}_n  \cdot V_1 \cdot e^{k_2}_n \cdots e^{k_{r-1}}_n \cdot V_{r-1} \cdot e^{k_{r}}_n \cdot V_{r} \cdot e^{k_{r+1}}_n. \]
Note that we can assume the copies of $e_n$ are consecutive in the braid word $\ftbraid_n^k$, meaning that in the barred braid word $s$ of $S$, there is no unbarred $n$ between two barred $n$'s, since otherwise the word $s$ would contain the support of a G1 Gaussian elimination isomorphism. In addition, all the closed components corresponding to the copies of $e_n$ are marked with a $+$. If any closed component is marked with a $-$, then the word contains the support of a G2 Gaussian elimination isomorphism. 

Assume that each $V_j \in \TL_{n}$, for $0\leq j \leq r+1$, is in Jones Normal Form. 
We focus on the last segment of the word $S$: 
\[ \cdots V_{r-1} \cdot  e^{k_r}_n \cdot V_{r} \cdot e^{k_{r+1}}_n \]
Since $V_{r-1}$ is nonempty, the last letter of $V_{r-1}$ is barred in $s$. Suppose the last letter of $V_{r-1}$ is not $e_{n-1}$.  Then if $k_r > 1$, $s$ contains a source of a G1 Gaussian elimination isomorphism, and therefore does not appear in the whittled complex. The same reasoning applies to $V_j$, for $0\leq j \leq r-1$. Thus we may assume all of them end with $e_{n-1}$. 
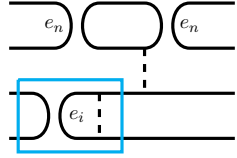
\begin{figure}[H]\centering 
    \begin{tikzpicture}[every node/.style={scale=.7}, scale=.6]

\begin{scope}[ rotate=-90]  
\draw[very thick] (1,1.5) to (1,0.5) .. controls ++(0,-.5) and ++(0,-.5) .. (0,0.5) to (0,1.5);
\draw[very thick] (2,1.5) to (2,-2) .. controls ++(0,-.5) and ++(0,-.5) .. (3,-2) to (3,1.5);
\draw[very thick] (1,-1.5) to (1,-0.5) .. controls ++(0,.5) and ++(0,.5) .. (0,-0.5)--(0,-1.5) .. controls ++(0,-0.5) and ++(0,-.5) .. (1,-1.5);
\node at (0.5, -2.5){$e_n$}; 
\node at (0.5, 0.5){$e_n$}; 
\node at (2.5, -2){$e_i$}; 
\draw[very thick] (0,-3.5)  to (0,-2.5) .. controls ++(0,.5) and ++(0,.5) .. (1,-2.5)--(1,-3.5);
\draw[very thick] (2,-3.5)  to (2,-3) .. controls ++(0,.5) and ++(0,.5) .. (3,-3)--(3,-3.5);
\draw[very thick, dashed] (2, -1.5) to (3, -1.5);
\draw[very thick, dashed] (1, -.5) to (2, -0.5);
\draw[very thick, cyan] (1.7, -3.3) to (1.7, -1) to (3.3, -1) to (3.3, -3.3) to (1.7, -3.3);
 \end{scope}
 \end{tikzpicture}
 \caption{ Source of G1 Gaussian elimination isomorphism.} \label{f.srgg1}
\end{figure}
Let $V_j'e_{n-1} = V_j$. We have the following picture near $V_j$ (see Figure \ref{f.top2strands}) involving the top two strands of the braid.  
\begin{figure}[H]
\begin{tikzpicture}[every node/.style={scale=.9}, scale=.6]
\begin{scope}[rotate=-90, scale=.8]

\draw[very thick] (0,2.5) to (0,3.5) .. controls ++(0,.5) and ++(0,.5) .. (-1,3.5)--(-1,2.5) .. controls ++(0,-0.5) and ++(0,-.5) .. (0,2.5);

\draw[very thick] (-1,10) to (-1,8.5).. controls ++(0,-.5) and ++(0,-.5) .. (0,8.5) to (0,10);

\draw[very thick] 
(1,4.5) to (1,6.5)
.. controls ++(0,.5) and ++(0,.5) .. 
(0,6.5) to (0,4.75) 
.. controls ++(0,-0.5) and ++(0,-.5) .. (-1,4.75) to (-1, 7.5)  
.. controls ++(0,.5) and ++(0,.5) .. 
(0, 7.5)  
.. controls ++(0,-0.5) and ++(0,-.5) .. 
(1,7.5) to (1,10)
;

\draw[very thick] 
(-1,-1.5)  to (-1,1.25) .. controls ++(0,.5) and ++(0,.5) .. 
(0,1.25) 
.. controls ++(0,-0.5) and ++(0,-.5) .. 
(1,1.25) to (1,4.75);

\draw[very thick] (0,-1.5)  to (0,0) .. controls ++(0,.5) and ++(0,.5) .. (1,0)--(1,-1.5);

\draw[very thick] (0,-0.5)  to (0,0) .. controls ++(0,.5) and ++(0,.5) .. (1,0)--(1,-0.5);

\draw[very thick, dotted, white]
(1,3.2) to (1,6);
\draw [decorate,decoration={brace,amplitude=5pt,mirror}]
  (1.5,3.2) -- (1.5,6) node[midway,yshift=-1em]{$V_j'$};
\draw[very thick, dotted, white]
(0,5.2) to (0,6);
\draw[very thick, dotted, white]
(-1,5.2) to (-1,6);
\node[very thick] at (-0.5, 5.5){$\cdots$}; 

\node at (-0.5, 1){$e_n$}; 
\node at (-0.5, 2.7){$e_n$};
\node at (-0.5, 9){$e_n$}; 
\node at (.5, -0.5){$e_{n-1}$}; 
\node at (.5, 8){$e_{n-1}$}; 
\end{scope}
  
   \end{tikzpicture}
\caption{ $S = \cdots e_{n-1} e_n V_j'e_{n-1}e_{n}  \cdots$}\label{f.top2strands}
\end{figure}
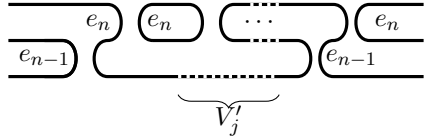

If $k_r>1$, we consider $V_j'=V_j''e_{k}e_{n-1}$. Suppose $k < n-2$, then the subword contains the support of a G1 Gaussian elimination isomorphism, and therefore does not appear in the whittled complex. Therefore $k = n-2$ and we can continue to reason in this way to find that $V_j = e_{j_1} e_{j_2}\cdots e_{n-1}$, where $j_{\ell+1} = j_{\ell} + 1$.   

If $V'_{j}$ does not contain $e_{n-2}$ and $k_r = 1$, then the subword $e_{n-1}e_n V'_j e_{n-1}e_n$ is the target $e_{n-1}e_ne_{n-1}$ of a G2 Gaussian elimination isomorphism. Therefore, we may assume that  $V'_j$ contains $e_{n-2}$, and we may write, for $V_j' = V''_{j}e_{n-2}$,  
\[ S = \cdots e_{n-1}e_n V''_{j}e_{n-2}e_{n-1}e_n V''_{j+1}e_{n-2}e_{n-1}e_n \cdots. \]
But the subword of $S$ shown, potentially contains a target for a G2 Gaussian isomorphism, namely  the subword $e_{n-2}e_{n-1}e_{n-2} \subset e_{n-1}e_n V''_{j}e_{n-2}e_{n-1}e_n V''_{j+1}e_{n-2}e_{n-1}e_n$, if $e_{n-3}\notin V''_{j+1}$. Therefore, we may assume $e_{n-3}\in V''_{j+1}$, and so on. Eventually, we hit the bottom strand of the braid, and so the generator does not appear in the whittled complex $\FT^k_n$.

In particular, this forces $r = 0$, and we have 
\[ S = e^{k_0}_n \cdot V_0 \cdot e^{k_1}_n. \]
Now since $V_0 \in \TL_n$, we may apply the induction hypothesis and replace $V_0$ by its Jones Normal Form. If $k_0 = 0$ or $k_1 = 0$, then we are done up to applying some number of commutation moves to push $e_n$ to the right of any elements in $V_0$ that it commutes with. Thus, we may assume $k_0, k_1 \not= 0$. At this point, the word $S$ has the form
\begin{align*}
S &= \underbrace{e_n \cdot e_n \cdots e_n}_{\text{$k_0$ times}} \cdot V_0 \cdot \underbrace{e_n\cdot e_n \cdots e_n}_{\text{$k_1$ times}} \\ 
&= \underbrace{e_n \cdot e_n \cdots e_n}_{\text{$k_0$ times}} \cdot V'_0 e_{n-1} \cdot \underbrace{e_n\cdot e_n \cdots e_n}_{\text{$k_1$ times}}. \\ 
\end{align*}

Now assume $k_1 > 1$. Then $V_0 = V_0'e_{n-1} = V_0''e_{n-2}e_{n-1} = V_0'''e_{n-3}e_{n-2}e_{n-1} = V''''_0\cdots$, and so on, to avoid the possibility that $S$ contains the source of a G1 Gaussian elimination isomorphism, as argued similarly in the previous paragraph where we established that $r=0$.  Thus we have the following: 
\[ S  = \underbrace{e_n \cdot e_n \cdots e_n}_{\text{$k_0$ times}} \cdot \underbrace{e_{v}e_{v+1}\cdots e_{n-2} e_{n-1}}_{V_0} \cdot \underbrace{e_n\cdot e_n \cdots e_n}_{\text{$k_1$ times}}\]

Applying move D1: $e^2_n \to e_n$ repeatedly to reduce $e_n^{k_0}$ and $e_n^{k_1}$ to $e_n$ yields the following: 
\[ S = e_n \cdot V'_0 e_{n-1}\cdot e_n. \]

Then applying move D3: $e_ie_{i-1}e_i \to e_i$, with $i=n$, sends $e_ne_{n-1}e_n$ to $e_n$, and then we are left with: 
\[ S = V'_0 \cdot e_n,\] which is in Jones Normal Form, as desired. 

Henceforth, we assume $k_1 = 1$. Recall that the word $V_0$ in Jones Normal Formal can be written as
\[ V_0 = (e_{i_1}e_{i_1-1} \cdots e_{j_1}) \cdots (e_{i_v}e_{i_v-1} \cdots e_{j_v}), \]
with 
  \[0 < i_1 < i_2 < \cdots <i_v < n, \qquad 0 < j_1 < j_2 < \cdots < j_v < n, \]
    and 
    \[j_1 \leq i_1, j_2 \leq i_2, \cdots, j_v \leq i_v.\]
Since 
\begin{align*} 
S &= e^{k_0}_n V_0 e_n \\ 
&= e^{k_0}_n (e_{i_1}e_{i_1-1} \cdots e_{j_1}) \cdots (e_{i_v}e_{i_v-1} \cdots e_{j_v}) e_n, \\ 
\intertext{we can apply move D1 to simplify $e^{k_0}_n$ to $e_n$. If $V_0$ commutes with $e_n$, then $e_n V_0 e_n \subset S$ is the target of a G1 Gaussian elimination isomorphism if the closed component corresponding to $e^2_n$ is marked with a $-$. Otherwise, the component remains. If $i_v < n-1$, then we may commute $e_n$ all the way to the right of $V_0$ and reduce the word via the move D1. The resulting word is in Jones Normal Form. Otherwise, we assume $i_v = n-1$. After possibly applying these moves, we have the following:}
S &\sim e_n (e_{i_1}e_{i_1-1} \cdots e_{j_1}) \cdots (e_{n-1}e_{n-2} \cdots e_{j_v}) e_n. \\
\intertext{By commutation move (D2) and the application of move D3: $e_ne_{n-1}e_n \to e_n$, we can remove $e_{n-1}$ and end up with the following:} 
S &\sim (e_{i_1}e_{i_1-1} \cdots e_{j_1}) \cdots (e_{n-2} \cdots e_{j_v}) e_n. \\
\intertext{It is possible to have $i_{v-1} = n-2$, in which case we have a copy of $e^2_{n-2}$ or $e_{n-2}e_{n-3}e_{n-2}$ to reduce:} 
S &\sim (e_{i_1}e_{i_1-1} \cdots e_{j_1}) \cdots (e_{n-2}e_{n-3} \cdots e_{j_{v-1}}) (e_{n-2} \cdots e_{j_v}) e_n. \\
\end{align*}
Repeated application of moves D1, D2, or D3 as needed, reduces the word to Jones Normal Form. 
\end{proof} 

\section{Enumerating the generators of the whittled complex} \label{s.enumerate}
We now reverse the process in Theorem \ref{t.deflate} (Proposition \ref{p.deflate}) to bound the number of generators of the whittled complex $\FT^k_n$ at a fixed homological grading. First for Case (2) of Theorem \ref{t.deflate}, we describe a more efficient way to obtain an enhanced Kauffman state in $\FT^k_n$  from a word in the Temperley-Lieb monoid $\TL_n$, compared to just taking a word with length less than or equal to the given homological degree and performing an arbitrary combination of the moves D1-D3. 

\begin{lemma} \label{l.enumerate_whittled} Fix the (nonzero) homological degree $h$ of the chain complex $\FT^k_n$, and let $s$ be a Kauffman state described by Case (2) of Theorem \ref{t.deflate} with homological grading $h(s) = h$ and associated word $S\in \TL_n$. Then the word $S$ can be obtained from a word $V \in \TL_{n-1}$ in Jones Normal Form with $\len(V)\leq h$ as follows: 
\begin{enumerate}
    \item Decide whether there is an $e_{n-1}$ in the word $S$. If not, then $V$ only appears in homological grading $h$ if $\len(V) = h$. 
    \item If there is an $e_{n-1}$ in the word $S$, then $S$ can be written in the form
    \[ e^{k_0}_{n-1} Ve^{k_1}_{n-1},  \] where at least one of $k_0, k_1 \not=0$. 
\end{enumerate}
\end{lemma}
\begin{proof}
    This follows from the proof of Case (2) of Theorem \ref{t.deflate} (Proposition  \ref{p.deflate}). 
\end{proof}

Therefore, to enumerate the number of generators in the whittled complex $\FT^k_n$, we simply need to enumerate the number of words in $\TL_{n-1}$ up to a fixed homological degree, then increase the length of the word by the application of the move D1 or D3 (note the D2 move which is commutation does not change the length of the word) to hit the homological degree. We believe the periodicity of the Kauffman states can also be derived, but we do not pursue a precise statement here.

\subsection{The ordered partition function}  
\begin{definition}[Partition of integer] \label{d.opartition} A \textit{partition} of a positive integer $n$ is a way of writing $n$ as a sum of positive integers. An \textit{ordered partition} of a nonnegative integers $n$ is a way of writing $n$ as a sum of positive integers, where the distinct orders of the positive integers (representable by a sequence) give distinct ordered partitions. We denote the number of ordered partitions of a nonnegative integers $n$ into $k$ parts by $p(n, k)$. 
\end{definition}

The following is a well-known application of the \textit{Stars-and-Bars} technique from combinatorics. 

\begin{theorem}[Stars and Bars] \label{t.starsbars}
    The number of ways to write a positive integer $n$ as an ordered sum of $k$ positive integers is 
    \[ p(n, k) = { {n-1} \choose {k-1}} . \]
\end{theorem}
\begin{proof}
We represent the positive integer $n$ as $n$ stars $\star$: $\underbrace{\star\star\cdots\star}_{\text{$n$ times}}$ and draw $k$ bars $|$. For a string of $n$ stars, there are $n-1$ locations for the possible placements of $k$ bars, which can be represented as a $(k-1)$-tuple that encodes an ordered integer partition of $n$ into $k$ parts. For example, 
\[ \star \star | \star \star \star | \star \]
represents an ordered integer partition dividing $6$ into 3 parts. Out of $n-1$ integers, we choose $k-1$ elements, and the binomial coefficient gives the formula. 
\end{proof}

\begin{lemma} \label{l.nTL}
The number of words in $\TL_n$ in Jones Normal Form of a fixed length $h$, denoted $N(n, h)$, is given by 
 \begin{equation} \label{e.N(n, h)} N(n, h) = \sum_{k \leq h} {{n-1} \choose {k}} p(h, k). \end{equation}
\end{lemma}
\begin{proof}
 It suffices to enumerate the number of words in Jones Normal Form with the fixed length $k$. According to Definition \ref{d.jnf}, the Jones Normal Form of a word in $\TL_n$ is determined by the $2k$-tuple $(i_1,\ldots, i_k, j_1, \ldots, j_k)$ satisfying 
 \begin{equation}  0 < i_1 < i_2 < \ldots <i_k < n, \qquad 0 < j_1 < j_2 < \ldots < j_k < n, \end{equation}
    and 
 \begin{equation} \label{e.jlessthani}  j_1 \leq i_1, j_2 \leq i_2, \ldots, j_k \leq i_k. \end{equation}
By \eqref{e.jlessthani}, we first enumerate the number of increasing sequences $0<i_1<i_2<\ldots < i_k <n$ with $k\leq h$. 

There are $n-1$ elements in the set $\{1, \ldots, n-1\}$. Take a $k$ element subset. There are ${n-1} \choose {k}$ such subsets. Each of these subsets can be rearranged into a strictly increasing sequence $0<i_1<i_2<\ldots < i_k <n$. 

Now a $k$-tuple $(j_1, j_2, \ldots, j_k)$ that satisfies the conditions in Definition \ref{d.jnf}, namely $j_{\ell}\leq i_{\ell}$ for all $1\leq \ell \leq k$, determines the Jones Normal Form of the Temperley-Lieb element. With the condition on the word specified by the $2k$-tupe $(i_1, i_2, \ldots, i_k, j_1, j_2, \ldots, j_k)$, this corresponds to an ordered partition of the nonnegative integer $h$ into $k$ parts. Thus for a given $k$-tuple $(i_1, i_2, \ldots, i_k)$, there are 
\[ p(h, k) \]
possibilities by Theorem \ref{t.starsbars}.
In total, we have 
\[ \sum_{k\leq h} {{n-1} \choose {k}} p(h, k)\]
possibilities.  
\end{proof}


We are now ready to prove Theorem \ref{t.count}, which we restate here as the proposition below, for the convenience of the reader. In the statement of the proposition, $C_n$ stands for the $n$th Catalan number.

\begin{proposition} \label{p.count}
            For a fixed homological degree $h$, let $C(\FT^k_n, h)$ be the number of Kauffman states $s$ for which $(s, \epsilon)$ appears in the whittled complex $\FT^k_n$ and the homological degree of $(s, \epsilon) = h$. Then  
            \begin{equation}\label{eqp.count}
                C(\FT^k_n, h) \leq  \left( \sum_{m\leq h} p(h, m) \right) + N(n, h) +  (p(n, 2)+2)C_n.
            \end{equation}   
\end{proposition}

\begin{proof}  
\noindent \textbf{Case 1.}
We start with the count for Case (1) of Theorem \ref{t.deflate}. In this case, a word $W$ corresponding to an enhanced Kauffman state $(w, \epsilon)$ in $\FT^k_n$ is of the form
\[ e_{n-1}^{k_0} V_0 e_{n-1}^{k_1} V_1 \cdots V_{r}e_{n-1}^{k_{r}},   \]
where $V_j = e_{j_1}e_{j_2}\cdots e_{n-1}$. The count depends on two tuples $(j_1)_{j=0}^r$ and $(k_0, \ldots, k_{r+1})$. Fixing the homological degree $h$, given the braid word $\ftbraid^k_n$, the sum
\begin{equation}\label{eq:m}
    m := \sum_{i=0}^{r}k_i \leq h.
\end{equation}
If $m = h,$ then the homological degree has been exhausted, and we count a single element $e_{n-1}^m$ in the enhanced Kauffman state.  In case of a strict inequality in Equation \eqref{eq:m}, let  
$ \ell:= h-m.$ Then we consider the number of ordered partitions of $h$ into $m$ parts. By Theorem \ref{t.starsbars}, this number is $p(h, m)$. Define $p(0, m) = 1$.  The total count for this case is then \begin{equation}\label{eq:1}
    \sum_{m\leq h} p(h, m).
\end{equation}

\noindent \textbf{Case 2. }
    We divide the count between two sets of Kauffman states. 
    \begin{itemize}
    \item[(2a)] $\len(S) = h$: The count is equal to the number of Temperley-Lieb elements of length $h$. By Lemma \ref{l.nTL}, this number is $N(n, h)$.
    \item[(2b)] $\len(S) < h$: As proven in Proposition \ref{p.deflate}, $S = e_{n-1}^{k_0}Ve_{n-1}^{k_1}$, where $V$ is in the Jones Normal Form. Note that the maximum number of $e_{n-1}$'s in the braid word $\ftbraid^k_n = (\sigma_1\sigma_2 \cdots \sigma_{n-1})^k$ is given by $k$. Thus $k_0 + k_1 \leq k$. For the choice of $V\in \TL_{n-1}$, we have the count of the number of elements in $\TL_{n-1}$ is the $n$th Catalan number $C_n = \frac{1}{n+1}{{2n} \choose {n}}$. In addition, the number of ordered partitions of $k$ into 2 parts corresponding to $k_0, k_1$ + 2 is given by $p(n, 2)$, by Theorem \ref{t.starsbars}. Hence the total in this case is 
    $C_n \cdot p(n, 2).$
    \end{itemize} 



    Counts in Case 1, and Cases (2a) and (2b) add up to the total count given in Equation \refeq{eqp.count}.
    \end{proof}

\begin{remark}
It is likely that the upper bound is not sharp and can be further reduced, since there is a limit to the number of elements of $\TL_n$ that can be supported on a given braid $\ftbraid^k_n$ for a fixed homological degree $h$. 
\end{remark}

\subsection*{Acknowledgments} We are grateful to  Melissa Zhang, Matt Hogancamp, Mikhail Khovanov, and Eugene Gorsky for enlightening conversations about their work in relation to this project. We would especially like to thank Melissa Zhang for supplying the reference and proof of Proposition \ref{prop:monotone-path} and suggesting the name ``whittled" for our complex. We would like to acknowledge an AIM Square grant, a Summer Research in Mathematics (SWiM) at SLMath (formerly MSRI) grant, and an ICERM collaboration grant for providing excellent working conditions at the American Institute of Mathematics and at the Institute for Computational and  Experimental Research in Mathematics. C. Caprau was partially supported by NSF-RUI grant DMS 2204386. C. Lee was partially supported by NSF grants DMS 1907010 (University of South Alabama), DMS 2244923 (Texas State University) and CAREER-DMS 2440680 (Texas State University).  R. Sazdanovic was partially supported by NSF grant DMS-1854705 (NC State University).

\bibliographystyle{alpha}
\bibliography{references}

\end{document}

%% file: commands.tex
\usepackage{amsmath, amsthm, mathrsfs,amssymb, amsfonts}
\usepackage{xr}
\usepackage{cite}

\usepackage{rotating}

\usepackage{graphicx}
\usepackage{color}
\usepackage{textcomp}
\usepackage{mathtools}
\usepackage{ytableau}
\usepackage{stmaryrd}
\usepackage{soul}

\usepackage{tikz}
\usepackage{tikz-cd}
\usepackage[utf8]{inputenc}
\usepackage{cmap}
\usepackage[T1]{fontenc}
\usepackage{tikz,environ}
\usepackage{float}
\usetikzlibrary{patterns,snakes}
\usetikzlibrary{cd}
\usepackage{ dsfont }
\usetikzlibrary{arrows}
\usepackage{multicol}

\usepackage{etoolbox}

\usepackage{enumitem}
\makeatletter
\def\namedlabel#1#2{\begingroup
    #2%
    \def\@currentlabel{#2}%
    \phantomsection\label{#1}\endgroup
}

\usepackage{url}

\usepackage[bookmarks=true,
    colorlinks=true,
    linkcolor=purple,
    citecolor=red,
    filecolor=orange,
    menucolor=purple
    urlcolor=green,
    breaklinks=true]{hyperref}

\usepackage{fullpage}

\input xy
\usepackage[all]{xy}
\xyoption{line}
\xyoption{arrow}
\xyoption{color}
\SelectTips{cm}{}

\usepackage{tikz}
\usetikzlibrary{decorations.markings}
\usetikzlibrary{decorations.pathreplacing}

\tikzstyle directed=[postaction={decorate,decoration={markings,
    mark=at position #1 with {\arrow{>}}}}]
\tikzstyle rdirected=[postaction={decorate,decoration={markings,
    mark=at position #1 with {\arrow{<}}}}]

\usetikzlibrary{calc}
\usepackage{relsize}

\tikzset{fontscale/.style = {font=\relsize{#1}}
    }

\usetikzlibrary{decorations.pathmorphing}
\usetikzlibrary{decorations.text}

\tikzset{snake it/.style={decorate, decoration=snake}}

\usepackage{graphicx}
\usepackage{color}

\def\FT{\mathcal{FT}}  
\def\Cob{\mathsf{Cob}} 

\usepackage{bbm}

\theoremstyle{plain}
\newtheorem{theorem}{Theorem}[section]

\newtheorem{proposition}[theorem]{Proposition}
\newtheorem{lemma}[theorem]{Lemma}
\newtheorem{algorithm}[theorem]{Algorithm}

\theoremstyle{definition}
\newtheorem{definition}[theorem]{Definition}
\newtheorem{notation}[theorem]{Notation}

\theoremstyle{remark}
\newtheorem{remark}[theorem]{Remark}

\numberwithin{equation}{section}

\renewcommand{\to}{\rightarrow}

\def\1{\mathsf{1}}

\newcommand\nc{\newcommand}
\nc\rnc{\renewcommand}
\nc\Kar{\operatorname{Kar}} 
\nc\End{\operatorname{End}}

\newlength\cellsize 

\savebox2{%
\begin{picture}(10,10)
\put(0,0){\line(1,0){10}}
\put(0,0){\line(0,1){10}}
\put(10,0){\line(0,1){10}}
\put(0,10){\line(1,0){10}}
\end{picture}}

\newcommand\cellify[1]{\def\thearg{#1}\def\nothing{}%
\ifx\thearg\nothing\vrule width0pt height\cellsize depth0pt%
  \else\hbox to 0pt{\usebox2\hss}\fi%
  \vbox to 10\unitlength{\vss\hbox to 10\unitlength{\hss$#1$\hss}\vss}}

\newcommand\tableau[1]{\vtop{\let\\=\cr
\setlength\baselineskip{-10000pt}
\setlength\lineskiplimit{10000pt}
\setlength\lineskip{0pt}
\halign{&\cellify{##}\cr#1\crcr}}}

\let\hat=\widehat

\let\epsilon=\varepsilon

\allowdisplaybreaks

\usepackage{verbatim}

\newtoggle{details}

\toggletrue{details}   

\iftoggle{details}{%
  \newcommand{\details}[1]{
      \ \\
      {
        \textbf{Details:} #1
      }
      \\
  }
}{%
  \newcommand{\details}[1]{}
}

\usepackage{xcolor}
\usepackage[colorinlistoftodos]{todonotes}

\newcommand{\map}[1]{\xrightarrow{#1}}

\newcommand{\TL}{\mathrm{TL}} 
\newcommand{\JNF}{\mathrm{JNF}} 
\newcommand{\len}{\mathrm{len}} 

\newcommand{\inv}{^{-1}} 

\newcommand{\CKh}{\mathrm{CKh}} 
\newcommand{\Kh}{\mathrm{Kh}} 

\newcommand{\ftbraid}{ft} 

\newcommand{\kap}{
    \begin{scope}[xscale=.5]
    \draw (0,1) arc (90:270:1cm);
    \draw[dotted] (0,1) arc (90:-90:1cm);
    \end{scope}
    \begin{scope}[xscale=1.5]
        \draw (0,1) arc (90:-90:1cm);
    \end{scope}
}
\newcommand{\kapdot}{
    \kap
    \node (dot) at (1,0) {$\bullet$};
}
\newcommand{\kup}{
    \begin{scope}[xscale=-.5]
    \draw (0,1) arc (90:270:1cm);
    \draw (0,1) arc (90:-90:1cm);
    \end{scope}
    \begin{scope}[xscale=-1.5]
        \draw (0,1) arc (90:-90:1cm);
    \end{scope}
}
\newcommand{\kupdot}{
    \kup
    \node (dot) at (-1,0) {$\bullet$};
}
